\documentclass[final,leqno]{siamltex}

	\usepackage{xcolor}
	\usepackage{framed}
	\usepackage{shadethm}
\usepackage{graphicx}
\usepackage{verbatim}
\usepackage{subfigure}
	\newshadetheorem{thm}{Theorem}
	\definecolor{shadethmcolor}{HTML}{EDF8FF}
	\definecolor{shaderulecolor}{HTML}{45CFFF}
	\definecolor{shaderulecolor}{gray}{0}
	\colorlet{shadecolor}{orange!15}

	\newshadetheorem{alg}{Algorithm}
	\definecolor{shadethmcolor}{HTML}{EDF8FF}
	\definecolor{shaderulecolor}{HTML}{45CFFF}
	\definecolor{shaderulecolor}{gray}{0}
	\colorlet{shadecolor}{orange!15}

	\newshadetheorem{prop}{Proposition}
	\definecolor{shadethmcolor}{HTML}{EDF8FF}
	\definecolor{shaderulecolor}{HTML}{45CFFF}
\definecolor{shaderulecolor}{gray}{0}
 	\colorlet{shadecolor}{orange!15}

 	\newshadetheorem{cor}{Corollary}
	\definecolor{shadethmcolor}{HTML}{EDF8FF}
	\definecolor{shaderulecolor}{HTML}{45CFFF}
\definecolor{shaderulecolor}{gray}{0}
 	\colorlet{shadecolor}{orange!15}

\usepackage{amsmath}
\usepackage{amsfonts}
\usepackage{amssymb}
\usepackage{amsbsy}

\usepackage[pdftex,colorlinks=true,urlcolor=blue,citecolor=black,anchorcolor=black,linkcolor=black]{hyperref}

\usepackage{amsmath, amssymb, bbm, xspace}

\newcommand{\Real}{\ensuremath{\mathbb{R}}}

\newcommand{\minimize}[1]{\displaystyle\minim_{#1}}
\newcommand{\minim}{\mathop{\hbox{\rm minimize}}}

\DeclareMathOperator*{\st}{subject\;to}

\def\subject{\hbox{\rm subject to}}

\def\spose#1{\hbox to 0pt{#1\hss}}

\def\text #1{\hbox{\quad#1\quad}}

\def\nthinsp{\mskip -2   mu}

\def\superstar{^{\raise 0.5pt\hbox{$\nthinsp *$}}}
\def\SUPERSTAR{^{\raise 0.5pt\hbox{$*$}}}

\def\lamstarT {\lambda^{\raise 0.5pt\hbox{$\nthinsp *$}T}}

\def\Oscr{{\cal O}}

\def\hbar{\skew{4.2}\bar h}

		\def\bkE{\mathbb{E}}
		\def\bk1{{\rm 1\kern-.17em l}}
		\def\bkD{{\rm I\kern-.17em D}}
		\def\bkR{{\rm I\kern-.17em R}}
		\def\bkP{{\rm I\kern-.17em P}}
		\def\bkY{{\bf \kern-.17em Y}}
		\def\bkZ{{\bf \kern-.17em Z}}

		\def\beq{\begin{eqnarray}}
		\def\bc{\begin{center}}
		\def\be{\begin{enumerate}}
		\def\bi{\begin{itemize}}
		\def\bs{\begin{small}}
		\def\bS{\begin{slide}}
		\def\ec{\end{center}}
		\def\ee{\end{enumerate}}
		\def\ei{\end{itemize}}
		\def\es{\end{small}}
		\def\eS{\end{slide}}
		\def\eeq{\end{eqnarray}}
		\def\qed{\quad \vrule height7.5pt width4.17pt depth0pt}

		\def\problemjoint#1#2#3#4#5{\fbox
		 {\begin{tabular*}{0.4\textwidth}
			{@{}l@{\extracolsep{\fill}}l@{\extracolsep{6pt}}l@{\extracolsep{\fill}}c@{}}
				#1 & $\minimize{#2}$ & $#3$ & $ $ \\[5pt]
					 & $\minimize{#4}\ $    & $#5$ & $ $
			\end{tabular*}}
			}

		\def\problemsmall#1#2#3#4{\fbox
		 {\begin{tabular*}{0.4\textwidth}
			{@{}l@{\extracolsep{\fill}}l@{\extracolsep{6pt}}l@{\extracolsep{\fill}}c@{}}
				#1 & $\minimize{#2}$ & $#3$ & $ $ \\[5pt]
					 & $\subject\ $    & $#4$ & $ $
			\end{tabular*}}
			}

	\def\cp2problem#1#2#3#4{\fbox
		 {\begin{tabular*}{0.9\textwidth}
			{@{}l@{\extracolsep{\fill}}l@{\extracolsep{6pt}}l@{\extracolsep{\fill}}c@{}}
				#1 & & $#4 $ 
			\end{tabular*}}}

\newcommand{\pmat}[1]{\begin{pmatrix} #1 \end{pmatrix}}
		
		\renewcommand{\emph}[1]{\textbf{#1}}

		\def\bkE{{\rm I\kern-.17em E}}
		\def\bk1{{\rm 1\kern-.17em l}}
		\def\bkD{{\rm I\kern-.17em D}}
		\def\bkR{{\rm I\kern-.17em R}}
		\def\bkP{{\rm I\kern-.17em P}}
		
		\def\bkZ{{\bf{Z}}}

\newcommand {\beeq}[1]{\begin{equation}\label{#1}}
\newcommand {\eeeq}{\end{equation}}
\newcommand {\bea}{\begin{eqnarray}}
\newcommand {\eea}{\end{eqnarray}}

\def\texitem#1{\par\smallskip\noindent\hangindent 25pt
               \hbox to 25pt {\hss #1 ~}\ignorespaces}

\def\st{\mbox{subject to}}

\def\argmin{\mathop{\rm argmin}}

\newtheorem{assumption}{Assumption}
\newtheorem{defi}{Definition}
\newtheorem{lem}[theorem]{Lemma}

\title{On the solution of stochastic optimization and variational
	problems in imperfect information regimes}

\author{Hao Jiang
\and
Uday V. Shanbhag\thanks{Jiang (\texttt{jiang23@illinois.edu}) and
	Shanbhag (\texttt{udaybag@psu.edu}) are at the Industrial and
	Enterprise Systems Engineering at University of Illinois at
		Urbana-Champaign and the Industrial and Manufacturing
		Engineering at the Pennsylvania State University at University
		Park, respectively. A conference version of this paper was
		presented at the Winter Simulation
		Conference~\cite{jiang13solution}. This research has been partially funded by NSF Award CMMI-1246887 (CAREER) and NSF Award CMMI-1400217.}
}

\begin{document}
\maketitle
\begin{abstract}
We consider the solution of a stochastic convex optimization problem
$\mathbb{E}[f(x;\theta^*,\xi)]$ over a closed and convex set $X$ in a
regime where $\theta^*$ is unavailable and $\xi$ is a suitably defined
random variable. Instead, $\theta^*$ may be obtained through the
solution of a learning problem that requires minimizing a metric
$\mathbb{E}[g(\theta;\eta)]$ in $\theta$ over a closed and convex set
$\Theta$. Traditional approaches have been either sequential or
direct variational approaches. In the case of the former, this entails the
following steps: (i) a solution to the learning problem, namely $\theta^*$, is obtained; and
(ii) a solution is obtained to the associated computational problem
which is parametrized by $\theta^*$.  Such avenues prove difficult to
adopt particularly since the learning process has to be terminated
finitely and consequently, in large-scale instances, sequential
approaches may often be corrupted by error. On the other hand, a variational approach
requires that the problem
may be recast as a possibly non-monotone stochastic variational
inequality problem in the $(x,\theta)$ space; but there are no known
first-order stochastic approximation schemes are currently available for
the solution of this problem.  To resolve the absence of convergent
efficient schemes, we present a coupled stochastic approximation
scheme which simultaneously solves {\em both} the computational and the
learning problems. The obtained schemes are shown
to be equipped with almost sure convergence properties in regimes when
the function $f$ is either strongly convex as well as merely convex.
Importantly, the scheme displays the optimal rate for strongly convex
problems while in merely convex regimes, through an averaging approach,
we quantify the degradation associated with learning by noting that the
error in function value after $K$ steps is $\Oscr\left(\sqrt{{\ln(K)}\slash{K}}\right)$,
	  rather than $\Oscr\left(\sqrt{{1}\slash{K}}\right)$ when $\theta^*$
	  is available. Notably, when the averaging window is modified suitably, it can be see that the originakl rate of $\Oscr\left(\sqrt{{1}\slash{K}}\right)$ is recovered. Additionally, we consider an online counterpart of
	  the misspecified optimization problem and provide a non-asymptotic
	  bound on the average regret with respect to an offline
	  counterpart.  In the second part of the paper, we extend these
	  statements to a class of stochastic variational inequality
	  problems, an object that unifies stochastic convex optimization
	  problems and a range of stochastic equilibrium problems.
	  Analogous almost-sure convergence statements are provided in
	  strongly monotone and merely monotone regimes, the latter
	  facilitated by using an iterative Tikhonov regularization. In the
	  merely monotone regime, under a weak-sharpness requirement, we
	  quantify the degradation associated with learning and show that
	  expected error associated with $\mbox{dist}(x_k,X^*)$ is
	  $\Oscr\left(\sqrt{{\ln(K)}\slash{K}}\right)$.  Preliminary numerics
	  demonstrate the performance of the prescribed schemes.
	  \end{abstract}



\pagestyle{myheadings}
\thispagestyle{plain}
\markboth{HAO JIANG AND UDAY V. SHANBHAG}{STOCHASTIC
	OPTIMIZATION AND VARIATIONAL PROBLEMS IN IMPERFECT INFORMATION REGIMES}

\section{Introduction} \label{sec:intro}
In the last two decades, robust
optimization~\cite{bental09,bertsimas11theory} approaches have grown in
relevance when decision-makers are faced with optimization problems with
uncertain parameters. Succinctly, in such an approach, given an
uncertainty set that captures the realizations assumed by such a parameter, the {\em robust} solution represents the {\em worst-case} over
this set of realizations. Naturally, an appropriate choice of such an
uncertainty set is crucial and as the availability of data reaches
levels hitherto unseen, there is growing interest in data-driven
approaches~\cite{bertsimas13} for constructing such sets. Our interest is in closely
related yet distinct settings driven by data in which  {the point estimate} of a  parameter may be obtained through a learning problem,
		 suitably defined through the aggregation of data. We provide
		 two instances of such problems:

\paragraph{{\em (i)} Portfolio optimization} Portfolio optimization
problems prescribe the optimal constructions of portfolios over a set of
assets, for which the mean and covariance of returns are not necessarily
known. Traditional approaches have assumed that such returns are
available while more recent robust optimization models have
utilized factor-based models in constructing uncertainty
sets~\cite{goldfarb03robust,bertsimas04,bertsimas08}. An
alternate, and possibly less conservative, data-driven model of such a problem that
employs a point estimate of the mean and covariance matrix requires the
solution of two coupled problems: (1) A portfolio optimization problem
parametrized by $(\theta^*,\Sigma^*)$ representing the mean and
covariance matrix of returns; and (2) A learning
problem that utilizes data to obtain the best $(\theta^*,\Sigma^*)$.
\paragraph{{\em (ii)} Power systems operation} The operation of power
grids relies on the solution of hourly (or more frequent) commitment and dispatch problems, each
of which is reliant on a range of parameters that are often uncertain.
These parameters include supply-side information regarding capacity of
wind-power as well as load forecasts. Recently robust optimization
approaches have proved to be exceedingly
popular~\cite{twostage14jiang,jiang14unified,jiang14robust}. An
	alternate formulation is given by the following two coupled problems: (1)
An economic dispatch problem parametrized by $\theta^*$, a vector that
captures the unknown supply and demand side parameters; and (2) A
learning problem that computes $\theta^*$ through the accumulation of
data.

We believe that such coupled formulations have broad applicability
beyond merely the settings mentioned above in (i) and (ii).
They may also find application in inventory control problems with
stochastic demand~\cite{robinv1,robinv2,robinv3,robinv4}, robust network
design~\cite{koster13robust}, robust routing in communication
networks~\cite{hijazi13robust}, amongst others.  To recap the difference between the two problem frameworks, it can be
seen that (R-Opt), a robust optimization framework, minimizes the
worst-case of the optimal value $f(x;\theta)$ over the uncertainty set
${\cal U}_{\theta}$ while (L-Opt) considers the joint solution of an
optimization problem in $x$, parametrized by $\theta^*$, where
$\theta^*$ is a solution to a learning problem with a
metric $g(\theta)$. The
following formulations may provide a clearer comparison:
$$\problemsmall{R-Opt}
	{}
	{\displaystyle \max_{\theta \in {\cal U_{\theta}}} \ f(x;\theta)}
		 {\begin{array}{r@{\ }c@{\ }l}
				 x \in X.
				 \end{array}} \qquad
 \problemjoint{L-Opt}
	{x \in X}
	{ f(x;\theta^*)}
{\theta \in \Theta}	 {g(\theta)}
	$$
We consider regimes where the function $f(x;\theta)$ is a convex
expected-value function and the resulting problem is given by the
following:
\begin{align} \label{problem_stoch_optim_0}
\tag{${\cal P}^o_x(\theta^*)$}
    \min_{x \in X} \ \mathbb{E}[f(x,\xi(\omega);\theta^*)],
\end{align}
where $X \subseteq \Real^n$ is a closed and convex set, $\xi: \Omega \to
\Real^d$ is a $d-$dimensional random variable defined on a
probability space $(\Omega,\mathcal{F}_x,\mathbb{P}_x)$, $f:X \times \Real^d \times
\Real^m \to \Real$ is a real-valued function, and $\theta^*$ denotes an
$m-$dimensional vector of parameters.  Estimating such
parameters often requires the resolution of a suitably defined learning
problem, given by a stochastic optimization problem
{(${\cal L}_\theta$)}, and defined next:
\begin{align} \label{problem_stoch_optim_theta}
\tag{${\cal L}_\theta$}
    \min_{\theta \in \Theta} \quad & g(\theta) \triangleq
	\mathbb{E}[g(\theta;\eta)],
\end{align}
where $\Theta \subseteq \Real^m$ is a closed and convex set,  $\eta:
\Lambda\to \Real^p$ is a random variable defined on a probability space
$(\Lambda,\mathcal{F}_{\theta},\mathbb{P}_{\theta})$, and $g:
\Theta \times \Lambda \to \Real$ is a real-valued function.  When one considers the joint problem of learning and optimization, then
there are at least two obvious approaches that immediately emerge as
possibilities:\\
\noindent {\bf {(a)} Sequential approach:} Consider an inherently serial process wherein the first stage
incorporates a model/parameter specification phase based on statistical
learning while the second stage leverages these findings in developing
and solving the actual optimization problem of interest.  Such an
ordering relies on  the learning problems being relatively small and
tractable compared to the optimization problems, ensuring that accurate
solutions are available within a reasonable time period.  Strictly
speaking, if one terminates the learning process prematurely with
an estimator $\hat \theta$, the
resulting estimator is essentially corrupted by error in that $\hat
\theta \neq \theta^*$.  This error
propagates into the solution $\hat x$ of the computational problem, denoted by
${\cal P}^o_x(\hat \theta)$ and the associated gap might be quite
significant. Note that unless the learning problem is solvable via a
finite termination algorithm, such a approach cannot provide asymptotic
statements but can, at best, provide approximate solutions.
Consequently, an inherently serial process reliant on a prematurely
truncated learning scheme often fails to provide accurate solutions to
the computational problem.\\
\noindent {\bf {(b)} Variational approach:} Under suitable convexity
and differentiability requirements, the following holds: $$ x^* \mbox{ solves } {({\cal P}^o_x(\theta^*))}
	\mbox{ and } \theta^* \mbox{ solves } ({\cal L}_\theta), $$ if and only if
	$(x^*,\theta^*)$ is a solution to the (stochastic) variational
	inequality problem  VI$(Z, F)$  \cite{Pang03I} where
	$$ Z \triangleq X \times \Theta \mbox{ and } H(z) \triangleq \pmat{
		\mathbb{E} [\nabla_x f(x,\theta;\xi)] \\
		 \mathbb{E} [\nabla_{\theta} g(\theta;\eta)]}. $$ Recall that $z^*$ is a
		solution to VI$(Z,F)$ if $(z-z^*)^T F(z) \geq 0$ for all $z \in
		Z$.  Furthermore, if $x^*$  and $\theta^*$ denote solutions to
		{$({\cal P}^o_x(\theta^*))$} and (${\cal L}_\theta$), respectively, then an
		 oft-used avenue in obtaining a solution $(x^*,\theta^*)$
		 entails obtaining a solution to VI$(Z,F)$. However, unless
		 rather strong assumptions are imposed, the map $H$ is not
		 necessarily monotone, precluding the use of recently developed
		 stochastic approximation schemes for solving
		 monotone stochastic variational inequality
		 problems~\cite{jiang08stochastic,koshal13regularized,wintersim13}, 
		 extragradient-based
		 variants~\cite{juditsky11,yousefian14optimal}, and accelerated
		 approaches~\cite{chen14}.\\
\noindent {\bf Simultaneous approach:} This paper is motivated by the inadequacy of available approaches and,
more generally, the  absence of {\em asymptotically convergent schemes with
provable non-asymptotic rates}. We present
a framework where the  learning and the computational problems are
solved {\bf simultaneously} via a joint set of stochastic approximation
schemes.  Such an avenue has several advantages. First, under such an
approach, one can provide rigorous statements of asymptotic convergence
of the obtained estimators for both, the solution to the computational
problem and the associated learning problem. Second, error bounds on the
expected error can be provided for a fixed number of steps under a
regime with constant and diminishing steplengths. Third, the statements
may be extended to the variational regime in which the computational
problem is given by the variational counterpart of
		{$({\cal P}^o_x(\theta^*))$}, given by  {$({\cal P}^v_x(\theta^*))$}; such a problem requires an $x^* \in X$ such that
\begin{align} \label{problem_stoch_optim_1}
\tag{{${\cal P}^v_x(\theta^*)$}}
     \mathbb{E}[F(x^*,\xi(\omega);\theta^*)]^T(\, x-x^*\, )\, \geq \, 0, \qquad \forall \, x \,
	 \in \, X,
\end{align}
where $X \subseteq \Real^n$ is a closed and convex set, $\xi: \Omega \to
\Real^d$ is a $d-$dimensional random variable defined on a
probability space $(\Omega,\mathcal{F}_x,\mathbb{P}_x)$, $F:X \times \Real^d \times
\Real^m \to \Real^n$ is a real-valued continuous mapping. Note that when
$F(x^*,\xi(\omega);\theta^*) \triangleq \nabla_x f(x^*;\xi;\theta^*)$,
	this reduces to a convex optimization problem. Furthermore,  the choice of
using a variational problem, rather than merely an optimization
problem, is founded on the need to model a variety of multiagent
settings complicated by a breadth of strategic interactions, ranging
from purely cooperative to distinctly
noncooperative~\cite{facchinei02finite}.
\subsection{Related decision-making models}
While unaware of the availability of general purpose
tools that can resolve precisely such problems, we describe 
settings where such questions have assumed relevance:\\
\noindent {\bf Adaptive control~\cite{astrom94adaptive}:} In tracking
problems in  adaptive
control~\cite{AguechMP00}, the authors consider a perturbation approach for
  analyzing a adaptive tracking algorithm and consider three estimation
  schemes, specifically least mean squares (LMS) 
  scheme, its recursive variant (RLMS), and the Kalman filter (which
  requires some distributional assumptions on the noise). First, much of this treatment is in the unconstrained
  regime with tractable (often quadratic estimation objectives),
  allowing for deriving closed-form (and often linear) update rules.
   Second, when the noise in the estimation process is
  Gaussian, the Kalman filter provides a minimum variance estimator. If
  on the other hand, the noise is non-Gaussian, then the Kalman filter
  provides the optimal linear estimator (in the sense that no linear
		  filter provides smaller variance).  In fact, these assumptions often form the basis of most adaptive
	control algorithms (cf.~\cite{LjungG90} and ~\cite{Kushner10} for a
			discussion adaptive control and stochastic approximation.)
	Our focus is on static stochastic problems with far less assumptions
	on the nature of the problem and the associated distributions.
	Specifically, we allow for more general stochastic convex objectives
	(or monotone maps in the context of VIs) in either the optimization
	or the learning problem, allow for convex feasibility sets for both
	the optimization or the learning problems, and impose relatively
	mild moment assumptions on the noise (unlike the Gaussian
			assumptions that are necessary in some of the estimation
			models).

\noindent {\bf Iterative learning control:}
A related avenue lies in iterative learning control  (ILC) has its roots in the studies
  by Uchiyama~\cite{Uchiyama78} and Arimoto et al.~\cite{Arimoto84}.
    ILC~\cite{Moore93control} is a form of tracking control employed for
repetitive control problems, instances being chemical batch processes,
robot arm manipulators, and reliability testing rigs. Our problem is
more restrictive in its focus (static problems) but allow for more
general settings in terms of nonlinearity and the underlying
distributional requirements.

 \noindent {\bf Multi-armed bandit problems:}  The multi-armed bandit
 (MAB) problem considers  the question of
  how to play given a collection of slot machines faced by a gambler.  Each machine provides a random reward from a
  distribution specific to that machine. The gambler aims
  to maximize the expected sum of rewards earned through a sequence of lever
  pulls.   
  {The total discounted reward is}
  maximized by the index policy that pulls the bandit having greatest
  value of the Gittins index \cite{gittins89}. In effect, the reward
  function needs to be learnt while optimizing the system. There has been
  significant research on such problems
  over the last several decades, including on the question of
  computation~\cite{Katehakis87} and 
  finite-time analysis~\cite{auer02}.

Finally, related questions
also been studied in revenue management where  \cite{cooper06models}
examined the devastating effect of learning with an incorrect model while
maximizing revenue.

\subsection{Outline and contributions}
Broadly speaking, this paper focuses on the development of {\em
	stochastic approximation schemes} that generate iterates $\{x_k\}$
	and $\{\theta_k\}$ and makes the following contributions.  (i) In
	Section~\ref{sec:II}, we prove the a.s. convergence of the produced
	iterates to the prescribed solutions and derive error bounds in a
	standard and an averaging regime. In particular, we quantify the
	degradation in the convergence rate from introducing an additional
	learning phase; (ii)  Section~\ref{sec:II} concludes with a precise
	non-asymptotic bound on the average regret associated with employing the
	proposed scheme instead of an offline algorithm; (iii) In Section~\ref{sec:III}, we extend the a.s.
	convergence results to accommodate stochastic variational inequality
	problems, rather than merely convex optimization problems. Error
	analysis is carried out under a suitably defined growth property{;}
	(iv) In Section~\ref{sec:IV}, we provide some supporting numerics
	and conclude in Section~\ref{sec:V}. 
	Finally, throughout the paper,
	we use $\|x\|$ to denote the Euclidean norm of a vector $x$, i.e.,
	$\|x\|=\sqrt{x^Tx}$ and $\Pi_K$ to denote the Euclidean projection
	operator onto a set $K$, i.e., $\Pi_K(x)\triangleq \argmin_{y\in
		K}\|x-y\|$.

\section{Stochastic optimization problems with imperfect information} \label{sec:II}
In this section, we focus on examining {$({\cal P}^o_x(\theta^*))$} under various
assumptions.  We begin by stating the coupled stochastic approximation scheme and
providing the necessary assumptions in Section~\ref{sec:II.I}.
Convergence analysis of the presented scheme is provided in
Section~\ref{sec:II.II} while diminishing and constant steplength rate
analysis is performed in
Section~\ref{sec:II.III}. We conclude with a discussion of
an online algorithm with the associated bounds on the decay of average
regret in Section~\ref{sec:II.IV}.

\subsection{Algorithm statement and assumptions}\label{sec:II.I}
As mentioned in the previous section, we propose a set of coupled
stochastic approximation schemes for computing $x^*$ and $\theta^*$.

\begin{alg}[{\bf Coupled SA schemes for stochastic optimization
	problems}] \label{alg:stoch_strongly_optim}
\noindent {\bf Step 0.} Given $x_0 \in X, \theta_0 \in \Theta$ and sequences $\{\gamma_{k,x},
	  \gamma_{k,\theta}\}$, $k := 0$\\
\noindent {\bf Step 1.}
\begin{align}
\tag{Opt$_k$}    x^{k+1} & := \Pi_{X}  \left(x^{k}-\gamma_{k,x}(\nabla_x
				f(x^k;\theta^k) + w^k)  \right), \qquad \, k \geq 0 \\
   \tag{Learn$_k$} \theta^{k+1} &: = \Pi_{\Theta}  \left(  \theta^{k} -
			\gamma_{k,\theta} ( \nabla_{\theta} g(\theta^{k}) + v^k )
			\right),\quad\qquad \quad \, k \geq 0
\end{align}
where $w^k \triangleq \nabla_x f(x^k;\theta^k,\xi^k) - \nabla_x
f(x^k;\theta^k)$ and $v^k \triangleq \nabla_{\theta}
g(\theta^{k};\eta^k) - \nabla_{\theta} g(\theta^{k})$.

\noindent {\bf Step 2.} If $k > K$, stop; else $k:=k+1$, go to Step. 1.
\end{alg}
\vspace{0.2in}

We begin by stating an assumption on the functions $f$ and $g$.
\begin{assumption}[Problem properties, A\ref{assump:stoch_strongly_optim}-1] \label{assump:stoch_strongly_optim}
Suppose the following hold:
\begin{enumerate}
  \item[(i)] For every {$\theta \in \Theta$}, $f(x;\theta)$ is strongly convex and continuously differentiable with Lipschitz continuous gradients in $x$ with convexity constant $\mu_{x}$ and Lipschitz constant $L_{x}$, respectively.
  \item[(ii)] For every {$x \in X$}, the gradient $\nabla_x f(x;\theta)$ is Lipschitz continuous in $\theta$ with
  constant $L_{\theta}$.
  \item[(iii)] The function $g(\theta)$ is strongly convex and continuously differentiable with Lipschitz continuous gradients in $\theta$ with
  convexity constant $\mu_{\theta}$ and Lipschitz constant $C_{\theta}$, respectively.
\end{enumerate}
\end{assumption}

Under Assumption (A\ref{assump:stoch_strongly_optim}-1), the coupled
problem admits a unique solution, as shown next.
\begin{lem}[Solvability]
Consider the problems {$({\cal P}^o_x(\theta^*))$} and {$({\cal L}_\theta)$} and suppose
assumption (A1) holds. Then {$({\cal P}^o_x(\theta^*))$} and {$({\cal L}_\theta)$}
	collectively admit a unique solution.
\end{lem}
\begin{proof} This follows from the strong convexity of $g$ over
$\Theta$ and the strong {convexity of $f(\bullet;\theta)$ }over $X$.
\end{proof}

Additionally, we make the following assumptions on the steplength
sequences employed in the algorithm.
\begin{assumption}[Steplength requirements, A\ref{assump:stoch_steplength_strongly_optim}-1] \label{assump:stoch_steplength_strongly_optim}
Let $\{\gamma_{k,x}\}$ and  $\{\gamma_{k,\theta}\}$ be chosen such that:
\begin{enumerate}
  \item[(i)] $\sum_{k=0}^{\infty}\gamma_{k,x}=\infty$, $\sum_{k=0}^{\infty}\gamma_{k,x}^{2} < \infty$
  \item[(ii)]  $\gamma_{k,\theta} = \gamma_{k,x} L_{\theta}^2/(\mu_{x}\mu_{\theta})$.
\end{enumerate}
\end{assumption}
We define a new probability space $(Z,\mathcal{F},\mathbb{P})$, where
$Z\triangleq\Omega\times\Lambda$,
	$\mathcal{F}\triangleq\mathcal{F}_x\times\mathcal{F}_{\theta}$ and
	$\mathbb{P}\triangleq\mathbb{P}_x\times\mathbb{P}_{\theta}$.  We use
	$\mathcal{F}_k$ to denote the sigma-field generated by the initial
	points $(x^0,\theta^0)$ and errors $(w^l,v^l)$ for $l =
	0,1,\cdots,k-1$, i.e., $\mathcal{F}_0
	=\left\{(x^0,\theta^0)\right\}$ and $ \mathcal{F}_k =
	\left\{(x^0,\theta^0), \left( (w^l,v^l), l = 0,1,\cdots,k-1
			\right)\right\} $ for $k\geq 1.$ We make the following
	assumptions on the filtration and errors.
\begin{assumption} [A\ref{assump:filtration}] \label{assump:filtration}
Let the following hold:
\begin{enumerate}
  \item[(i)] $\mathbb{E}[w^k \mid \mathcal{F}_k] = 0$ and
  $\mathbb{E}[v^k \mid \mathcal{F}_k] = 0$ $a.s.$ for all $k$.
  \item[(ii)] $\mathbb{E}[\|w^{k}\|^2 \mid \mathcal{F}_k] \leq
	  \nu^2_{x}$ and $\mathbb{E}[\|v^{k}\|^2 \mid \mathcal{F}_k]  \leq
		  \nu^2_{\theta}$ $a.s.$ for all $k$.
\end{enumerate}
\end{assumption}

We conclude this subsection by stating three results (without proof)
that will be subsequently employed in developing our convergence
statements. The first two of these are relatively well-known
super-martingale convergence results (cf.~\cite[Lemma 10,
Pg.~49--50]{Polyak87})
\begin{lem}  \label{lem:supermartingale}
Let $v_k$ be a sequence of
nonnegative random variables adapted to $\sigma$-algebra $\mathcal{F}_k$ and
such that
\begin{align*}
  \mathbb{E}[v_{k+1}|\mathcal{F}_k] \leq (1-u_k)v_k+\beta_k \quad
  \textrm{ for all } k\geq 0 \quad\mbox{almost surely, }
\end{align*}
where $0\leq u_k \leq 1$, $\beta_k\geq 0$, and $\sum_{k=0}^{\infty}u_k=\infty$, $\sum_{k=0}^{\infty}\beta_k<\infty$ and $\lim_{k\to\infty}\frac{\beta_k}{u_k}=0$. Then, $v_k\to 0$ $a.s.$
\end{lem}

\begin{lem}\label{lem:supermartingale2}
Let $v_k$, $u_k$, $\beta_k$ and $\gamma_k$ be non-negative random
variables adapted to $\sigma$-algebra $\mathcal{F}_k$. If
$\sum_{k=0}^{\infty}u_k<\infty$, $\sum_{k=0}^{\infty}\beta_k<\infty$ and
\begin{align*}
  \mathbb{E}[v_{k+1}|\mathcal{F}_k] \leq (1+u_k)v_k-\gamma_k+\beta_k
  \quad  \textrm{ for all } k\geq 0 \quad\mbox{almost surely.}
\end{align*}
Then, $\{v_k\}$ is convergent  and $\sum_{k=0}^{\infty}\gamma_k<\infty$
almost surely.
\end{lem}
Finally, we present a contraction result reliant on
monotonicity and Lipschitz continuity requirements {(cf.~\cite[Theorem 12.1.2,
Pg.~1109]{Pang03II})}.
\begin{lem} \label{lem:strongly_Lipschitz}
Let $H:K\to\mathbb{R}^n$ be a mapping that is strongly
		monotone over $K$ with constant $\mu$, and Lipschitz
		continuous over $K$ with constant $L$. If $q \triangleq
		\sqrt{1-2\mu \gamma + \gamma^2 L^2}$, then for any $\gamma > 0$,
		we have that for any $x,y$, we have 
$\| \Pi_K(x-\gamma H(x))- \Pi_K(y-\gamma H(y))\| \leq q \|x-y\|.$
\end{lem}

\subsection{Almost-sure convergence}\label{sec:II.II}
Our first convergence result shows that under the prescribed assumptions, Algorithm
\ref{alg:stoch_strongly_optim} generates a sequence of iterates that
converges to the unique solution.\vspace{0.2in}
\begin{prop}[{\bf Almost-sure convergence under strong convexity of $f$}] \label{thm:stoch_strongly_optim}
Suppose (A\ref{assump:stoch_strongly_optim}-1), (A\ref{assump:stoch_steplength_strongly_optim}-1) and (A\ref{assump:filtration}) hold.  Let $\{x^k,\theta^k\}$ be computed via Algorithm \ref{alg:stoch_strongly_optim}.
Then, $x^{k}\to x^{*}$ and $\theta^{k} \to \theta^{*}$ $a.s.$ as
$k\rightarrow\infty$, where $\theta^*$ denotes the unique solution
of \eqref{problem_stoch_optim_theta}
 and $x^*$ denotes the unique solution to \eqref{problem_stoch_optim_0}.
\end{prop}\vspace{0.2in}

\begin{proof}
Note that $x^* = \Pi_{X} ( x^{*} - \gamma_{k,x}
		\nabla_x f(x^*;\theta^*)).$
Then, by the nonexpansivity of the Euclidean projector, $\|x^{k+1}-x^*\|^2$ may be bounded as follows:
 \begin{align*}
\notag     \|x^{k+1}-x^{*}\|^2 & =\|\Pi_{X}(x^{k}-\gamma_{k,x}(\nabla_x f(x^k;\theta^{k}) + w^k) )  -\Pi_{X}(x^{*}-\gamma_{k,x}\nabla_x f(x^{*};\theta^{*}))\|^2\\
\notag  &
\leq\|(x^{k}-x^{*})-\gamma_{k,x}(\nabla_x f(x^{k};\theta^{k})-\nabla_x f(x^{*};\theta^{*}))
-\gamma_{k,x} w^k \|^2.
\end{align*}
By adding and subtracting $\gamma_{k,x} \nabla_x f(x^*,\theta^k)$, this expression can be further expanded as follows:
\begin{align*}
\notag & \ \quad  \|(x^{k}-x^{*})-\gamma_{k,x}(\nabla_x f(x^{k};\theta^{k})-\nabla_x f(x^{*};\theta^{k}))    - \gamma_{k,x}(\nabla_x f(x^{*};\theta^{k})-\nabla_x f(x^{*};\theta^{*})) -\gamma_{k,x} w^k \|^2 \\
\notag  & = \|(x^{k}-x^{*})-\gamma_{k,x}(\nabla_x f(x^{k};\theta^{k})-\nabla_x f(x^{*};\theta^{k}))\|^2   + \gamma_{k,x}^2 \|\nabla_x f(x^{*};\theta^{k})-\nabla_x f(x^{*};\theta^{*})\|^2 + \gamma_{k,x}^2 \| w^k \|^2 \\
\notag        & \ \quad - 2 \gamma_{k,x} [ (x^{k}-x^{*})-\gamma_{k,x}(\nabla_x f(x^{k};\theta^{k})-\nabla_x f(x^{*};\theta^{k})) ]^T(\nabla_x f(x^{*};\theta^{k})-\nabla_x f(x^{*};\theta^{*})) \\
       & \ \quad - 2 \gamma_{k,x} [ (x^{k}-x^{*})-\gamma_{k,x}(\nabla_x f(x^{k};\theta^{k})-\nabla_x f(x^{*};\theta^{k})) ]^T w^k    + 2 \gamma_{k,x}^2 (\nabla_x f(x^{*};\theta^{k})-\nabla_x f(x^{*};\theta^{*}))^T w^k.
\end{align*}
By leveraging the fact that $\mathbb{E}[w^k \mid \mathcal{F}_k] = 0$, we have
 \begin{align}  \label{eq:stoch_strongly_optim_exp}
  \mathbb{E}[\|x^{k+1}-x^{*}\|^2 \mid \mathcal{F}_k]
   & \leq \textbf{Term 1} + \textbf{Term 2} +\textbf{Term 3}  + \gamma_{k,x}^2 \mathbb{E}[\| w^k \|^2\mid\mathcal{F}_k],
\end{align}
where {\bf Terms 1 -- 3} are defined as follows:
		  \begin{align*}
\mbox{\bf Term 1}  &\triangleq \|(x^{k}-x^{*})-\gamma_{k,x}(\nabla_x
		f(x^{k};\theta^{k})-\nabla_x f(x^{*};\theta^{k}))\|^2,\\
\mbox{\bf Term 2}   & \triangleq \gamma_{k,x}^2 \|\nabla_x f(x^{*};\theta^{k})-\nabla_x f(x^{*};\theta^{*})\|^2, \\
\mbox{ and } \mbox{\bf Term 3}  &\triangleq - 2 \gamma_{k,x} [
(x^{k}-x^{*})-\gamma_{k,x}(\nabla_x f(x^{k};\theta^{k})-\nabla_x
		f(x^{*};\theta^{k})) ]^T  (\nabla_x f(x^{*};\theta^{k})-\nabla_x f(x^{*};\theta^{*})).
\end{align*}
By Lemma \ref{lem:strongly_Lipschitz} and
(A\ref{assump:stoch_strongly_optim}-1), it follows that
\begin{align} \label{eq:stoch_strongly_optim_exp_term1}
\textbf{Term 1} \leq (1-2\gamma_{k,x} \mu_x + \gamma_{k,x}^2 L_x^2) \|x^{k}-x^{*}\|^2.
\end{align}
Furthermore, the Lipschitz continuity of $\nabla_x f(x^*;\theta)$ in
$\theta$ (A\ref{assump:stoch_strongly_optim}-1) allows for deriving the
following bound:
\begin{align} \label{eq:stoch_strongly_optim_exp_term2}
\textbf{Term 2} \leq  \gamma_{k,x}^2 L_{\theta}^2 \|\theta^{k}-\theta^{*}\|^2.
\end{align}
Finally, {\bf Term 3} can be bounded by invoking the Cauchy-Schwarz
inequality, Lemma \ref{lem:strongly_Lipschitz},
	(A\ref{assump:stoch_strongly_optim}-1) and the triangle inequality, we
	obtain
 \begin{align}
 \begin{aligned}\label{eq:stoch_strongly_optim_exp_term3}
     & \quad\ 2\gamma_{k,x}  \| (x^{k}-x^{*})-\gamma_{k,x}(\nabla_x f(x^{k};\theta^{k})-\nabla_x f(x^{*};\theta^{k})) \| \| \nabla_x f(x^{*};\theta^{k})-\nabla_x f(x^{*};\theta^{*}) \| \\
   &  \leq  2\gamma_{k,x}\sqrt{1-2\gamma_{k,x} \mu_x + \gamma_{k,x}^2 L_x^2} \|x^{k}-x^{*}\| L_{\theta} \|\theta^{k}-\theta^{*}\| \\
   &  \leq 2 \gamma_{k,x} L_{\theta} \|x^{k}-x^{*}\| \|\theta^{k}-\theta^{*}\| \\
 & \leq  \gamma_{k,x} \mu_{x} \|x^{k}-x^{*}\|^2 + \gamma_{k,x} (L_{\theta}^2/\mu_{x}) \|\theta^{k}-\theta^{*}\|^2,
 \end{aligned}
\end{align}
where the last inequality follows from $2a^Tb \leq \|a\|^2+\|b\|^2.$
Combining \eqref{eq:stoch_strongly_optim_exp},
		  \eqref{eq:stoch_strongly_optim_exp_term1},
		  \eqref{eq:stoch_strongly_optim_exp_term2} and
		  \eqref{eq:stoch_strongly_optim_exp_term3}, we get
\begin{align}
 \begin{aligned}   \label{eq:stoch_strongly_optim_exp_final}
 \mathbb{E}[\|x^{k+1}-x^{*}\|^2\mid\mathcal{F}_k]
      & \leq (1-\gamma_{k,x} \mu_x + \gamma_{k,x}^2 L_x^2) \|x^{k}-x^{*}\|^2 \\
      & + (\gamma_{k,x} L_{\theta}^2/\mu_{x} + \gamma_{k,x}^2 L_{\theta}^2) \|\theta^{k}-\theta^{*}\|^2 + \gamma_{k,x}^2  \nu_x^2.
       \end{aligned}
\end{align}
Recall  that $\theta^*$ satisfies the fixed point relationship $\theta^* = \Pi_{\Theta} ( \theta^{*} - \gamma_{\theta,k}
		\nabla_{\theta} g(\theta^*)),$ which, together with
non-expansivity of the Euclidean projector, allows for
deriving the following bound on $\|\theta^{k+1}-\theta^*\|^2$:
\begin{align*} 
& \notag \quad\ \|\theta^{k+1}-\theta^*\|^2
   =  \| \Pi_{\Theta}  (  \theta^{k} - \gamma_{\theta,k} ( \nabla_{\theta} g(\theta^{k}) + v^k ) ) - \Pi_{\Theta} ( \theta^{*} - \gamma_{\theta,k}
		\nabla_{\theta} g(\theta^*)) \|^2 \\
 \notag & \leq \| \theta^{k}-\theta^{*} - \gamma_{\theta,k}(\nabla_{\theta} g(\theta^{k}) - \nabla_{\theta} g(\theta^*)) - \gamma_{\theta,k}v^k \|^2 \\
  & = \| \theta^{k}-\theta^{*} - \gamma_{\theta,k}(\nabla_{\theta} g(\theta^{k}) - \nabla_{\theta} g(\theta^*))\|^2 + \gamma_{\theta,k}^2 \|v^k \|^2 - 2 ( \theta^{k}-\theta^{*} - \gamma_{\theta,k}(\nabla_{\theta} g(\theta^{k}) - \nabla_{\theta} g(\theta^*)) )^T v^k.
\end{align*}
By taking conditional expectations and by recalling that $\mathbb{E}[v^k
\mid \mathcal{F}_k] = 0$, we obtain the following bound:
 \begin{align}  \label{eq:stoch_strongly_optim_exp_final_theta}
 \begin{aligned}
 \mathbb{E}[\|\theta^{k+1}-\theta^{*}\|^2 \mid \mathcal{F}_k] & \leq \| \theta^{k}-\theta^{*} - \gamma_{k,\theta}(\nabla_{\theta} g(\theta^{k}) - \nabla_{\theta} g(\theta^*))\|^2  + \gamma_{k,\theta}^2 \mathbb{E}[\| v^k \|^2\mid\mathcal{F}_k] \\
   & \leq q_{k,\theta}^2 \|\theta^{k}-\theta^{*}\|^2 + \gamma_{k,\theta}^2 \nu_{\theta}^2,
    \end{aligned}
\end{align}
where $q_{k,\theta} \triangleq \sqrt{1-2\gamma_{k,\theta} \mu_{\theta} +
	\gamma_{k,\theta}^2 C_{\theta}^2}$.
Next, by adding $\eqref{eq:stoch_strongly_optim_exp_final}$ and
\eqref{eq:stoch_strongly_optim_exp_final_theta} and by invoking
(A\ref{assump:stoch_steplength_strongly_optim}-1), we obtain the following
bound.
\begin{align*} 
  \notag& \quad\ \mathbb{E}[\|x^{k+1}-x^{*}\|^2 \mid \mathcal{F}_k]  +
  \mathbb{E}[ \|\theta^{k+1}-\theta^{*}\|^2\mid \mathcal{F}_k] \\
	 \notag  &\leq (1-\gamma_{k,x} \mu_x + \gamma_{k,x}^2 L_x^2) \|x^{k}-x^{*}\|^2   + (q_{k,\theta}^2 + \gamma_{k,x} L_{\theta}^2/\mu_{x} + \gamma_{k,x}^2 L_{\theta}^2 ) \|\theta^{k}-\theta^*\|^2 + \gamma_{k,x}^2  \nu_x^2  + \gamma_{k,\theta}^2 \nu_{\theta}^2\\
 	 \notag & = (1-\gamma_{k,x} \mu_x + \gamma_{k,x}^2 L_x^2)
	 \|x^{k}-x^{*}\|^2   + (1- \gamma_{k,x} L_{\theta}^2/\mu_{x} +
			 \gamma_{k,x}^2 (L_{\theta}^2+ L_{\theta}^4
				 C_{\theta}^2/(\mu_{x}^2\mu_{\theta}^2) ) )
	\|\theta^{k}-\theta^*\|^2  \\ &\quad\ + \gamma_{k,x}^2  \nu_x^2 + \gamma_{k,x}^2 \nu_{\theta}^2 L_{\theta}^4 /(\mu_{x}^2\mu_{\theta}^2)\\
     \notag & \leq (1-\alpha \gamma_{k,x} + \beta  \gamma_{k,x}^2 ) ( \|x^{k}-x^{*}\|^2 + \|\theta^{k}-\theta^*\|^2 ) + \delta \gamma_{k,x}^2,
\end{align*}
where $\alpha=\min\{\mu_x, L_{\theta}^2/\mu_{x}\}$, $\beta = \max\{L_x^2, L_{\theta}^2+ L_{\theta}^4 C_{\theta}^2/(\mu_{x}^2\mu_{\theta}^2)  \}$ and
$\delta = \nu_x^2 + \nu_{\theta}^2 L_{\theta}^4 /(\mu_{x}^2\mu_{\theta}^2)$.
 From
(A\ref{assump:stoch_steplength_strongly_optim}-1), we have that
   $\sum_{k=0}^{\infty} (\alpha \gamma_{k,x} - \beta  \gamma_{k,x}^2) =\infty, \quad
   \sum_{k=0}^{\infty} \delta\gamma_{k,x}^2  < \infty,$ and
$$\lim_{k\to \infty} \frac{\delta\gamma_{k,x}^2 }{\alpha \gamma_{k,x} - \beta  \gamma_{k,x}^2}  = 0.$$
Then, by invoking the super-martingale convergence theorem (Lemma
		\ref{lem:supermartingale}), we have that
$\|x^{k}-x^{*}\|^2+\|\theta^k-\theta^*\|^2\to 0$ $a.s.$ as $k\to\infty$, which implies that
$x^k \to x^*$  and $\theta^k \to \theta^*$ $a.s.$ as $k\to\infty$.
\end{proof}
\\

Next we weaken the strong convexity requirement on the function $f$
through the following assumption.
\begin{assumption}[A1-2] \label{assump:stoch_monotone_optim}
Suppose the following holds in addition to (A\ref{assump:stoch_strongly_optim}-1 (ii)) and (A\ref{assump:stoch_strongly_optim}-1 (iii)).
\begin{enumerate}
  \item[(i)] For every {$\theta \in \Theta$}, $f(x;\theta)$ is convex and continuously differentiable with Lipschitz continuous gradients in $x$ with  Lipschitz constant $L_{x}$.
\end{enumerate}
\end{assumption}
Furthermore, we make the following assumptions on the steplength sequences employed
in the algorithm.
\begin{assumption}[A2-2] \label{assump:stoch_steplength_monotone_optim}
Let $\{\gamma_{k,x}\}$, $\{\gamma_{k,\theta}\}$ and some constant $\tau\in(0,1)$ be chosen such that:
\begin{enumerate}
  \item[(i)] $\sum_{k=0}^{\infty}\gamma_{k,x}^{2-\tau} < \infty$  and $\sum_{k=0}^{\infty} \gamma_{k,\theta}^2 < \infty$,
  \item[(iii)] $\sum_{k=0}^{\infty}\gamma_{k,x} = \infty$ and $\sum_{k=0}^{\infty} \gamma_{k,\theta}=\infty$,
  \item[(iii)] $\beta_k = \frac{\gamma_{k,x}^{\tau}}{2\gamma_{k,\theta} \mu_{\theta} } \downarrow 0$ as $k\to \infty$.
\end{enumerate}
\end{assumption}
Proceeding as in the previous result, we present a convergence result
under these weakened conditions.\vspace{0.2in}
\begin{thm} [{\bf Almost-sure convergence under convexity of $f$}]\label{thm:stoch_monotone_optim}
Suppose  (A1-2), (A2-2) and (A\ref{assump:filtration})
	hold. Suppose
	$X$ is bounded and the solution set $X^*$ of
	\eqref{problem_stoch_optim_0} is nonempty. Let $\{x^k,\theta^k\}$ be computed via Algorithm \ref{alg:stoch_strongly_optim}.
Then, $\theta^{k} \to \theta^{*}$ $a.s.$ as $k\rightarrow\infty$, and $x^{k}$
converges to a random point in $X^*$ $a.s.$ as $k\rightarrow\infty$,
		  where $\theta^*$ denotes the unique solution
of \eqref{problem_stoch_optim_theta} and $X^*$ denotes the solution set
of \eqref{problem_stoch_optim_0}.
\end{thm}\vspace{0.2in}

\begin{proof}
By the nonexpansivity of the Euclidean projector, we have for any $x^*\in X^*$ that
 \begin{align*}
\notag     \|x^{k+1}-x^{*}\|^2 & =\|\Pi_{X}(x^{k}-\gamma_{k,x}(\nabla_x f(x^k;\theta^{k}) + w^k) )  -\Pi_{X}(x^{*})\|^2 \\
&\leq\|(x^{k}-x^{*})-\gamma_{k,x}\nabla_x f(x^{k};\theta^{k})
-\gamma_{k,x} w^k \|^2.
\end{align*}
By adding and subtracting $\gamma_{k,x} \nabla_x f(x^*,\theta^k)$, this expression can be further expanded as follows:
\begin{align*}
\notag & \quad\ \|(x^{k}-x^{*})-\gamma_{k,x}\nabla_x f(x^{k};\theta^{*})-\gamma_{k,x}(\nabla_x f(x^{k};\theta^{k})-\nabla_x f(x^{k};\theta^{*}))  -\gamma_{k,x} w^k \|^2 \\
\notag  & = \|(x^{k}-x^{*})-\gamma_{k,x}\nabla_x f(x^{k};\theta^{*})\|^2
+ \gamma_{k,x}^2 \|\nabla_x f(x^{k};\theta^{k})-\nabla_x f(x^{k};\theta^{*})\|^2 + \gamma_{k,x}^2 \| w^k \|^2 \\
\notag        & \quad\ - 2 \gamma_{k,x} [ (x^{k}-x^{*})-\gamma_{k,x}\nabla_x f(x^{k};\theta^{*}) ]^T  (\nabla_x f(x^{k};\theta^{k})-\nabla_x f(x^{k};\theta^{*})) \\
 & \quad\ - 2 \gamma_{k,x} [ (x^{k}-x^{*})-\gamma_{k,x}\nabla_x f(x^{k};\theta^{*}) ]^T w^k + 2 \gamma_{k,x}^2 (\nabla_x f(x^{k};\theta^{k})-\nabla_x f(x^{k};\theta^{*}))^T w^k.
\end{align*}
Noting that $\mathbb{E}[w^k \mid \mathcal{F}_k] = 0$, we have
 \begin{align}  \label{eq:stoch_monotone_optim_exp}
  \mathbb{E}[\|x^{k+1}-x^{*}\|^2 \mid \mathcal{F}_k]
   & \leq \textbf{Term 1} + \textbf{Term 2} +\textbf{Term 3}  + \gamma_{k,x}^2 \mathbb{E}[\| w^k \|^2\mid\mathcal{F}_k],
\end{align}
where {\bf Terms 1 -- 3} are defined as follows:
\begin{align*}
\mbox{\bf Term 1} & \triangleq \|(x^{k}-x^{*})-\gamma_{k,x}\nabla_x f(x^{k};\theta^{*})\|^2,\\
\mbox{\bf Term 2}  &\triangleq\gamma_{k,x}^2 \|\nabla_x f(x^{k};\theta^{k})-\nabla_x f(x^{k};\theta^{*})\|^2, \\
\mbox{ and \bf Term 3} & \triangleq- 2 \gamma_{k,x} [ (x^{k}-x^{*})-\gamma_{k,x}\nabla_x f(x^{k};\theta^{*}) ]^T  (\nabla_x f(x^{k};\theta^{k})-\nabla_x f(x^{k};\theta^{*})).
\end{align*}
By invoking the convexity of $f(x;\theta)$ in $x$ and the gradient
inequality
(see A1-2), we have that
\begin{align*}\notag \  \textbf{Term 1} &  = \|x^{k}-x^{*}\|^2 + \gamma_{k,x}^2 \|\nabla_x f(x^{k};\theta^{*})\|^2 - 2\gamma_{k,x}(x^{k}-x^{*})^T \nabla_x f(x^{k};\theta^{*}) \\
& \leq \|x^{k}-x^{*}\|^2 + \gamma_{k,x}^2 \|\nabla_x f(x^{k};\theta^{*})\|^2 - 2 \gamma_{k,x} ( f(x^{k};\theta^{*}) - f(x^{*};\theta^{*}) ) \\
\notag & \leq \|x^{k}-x^{*}\|^2 + 2\gamma_{k,x}^2 \|\nabla_x f(x^{k};\theta^{*}) - \nabla_x f(x^{*};\theta^{*})\|^2 +
   2\gamma_{k,x}^2 \| \nabla_x f(x^{*};\theta^{*})\|^2  - 2 \gamma_{k,x} ( f(x^{k};\theta^{*}) - f(x^{*};\theta^{*}) ),
\end{align*}
where the last inequality follows from the identity $\|(a-b)+b\|^2 \leq
2\|a-b\|^2 + 2 \|b\|^2.$ From the Lipschitz continuity of $\nabla_x
f(x;\theta)$ in $x$, the right hand side can be
bounded as follows:
\begin{align} \label{eq:stoch_monotone_optim_exp_term1}
 \notag    \|x^{k}-x^{*}\|^2 & + 2\gamma_{k,x}^2 \|\nabla_x f(x^{k};\theta^{*}) - \nabla_x f(x^{*};\theta^{*})\|^2 +
   2\gamma_{k,x}^2 \| \nabla_x f(x^{*};\theta^{*})\|^2 - 2 \gamma_{k,x} ( f(x^{k};\theta^{*}) - f(x^{*};\theta^{*}) ) \\
  & \leq (1+2\gamma_{k,x}^2 L_x^2)\|x^{k}-x^{*}\|^2 + 2\gamma_{k,x}^2 \| \nabla_x f(x^{*};\theta^{*})\|^2  - 2 \gamma_{k,x} ( f(x^{k};\theta^{*}) - f(x^{*};\theta^{*}) ).
\end{align}
By the Lipschitz continuity of $\nabla_x f(x;\theta)$ in $\theta$ (A1-2),
\begin{align} \label{eq:stoch_monotone_optim_exp_term2}
\textbf{Term 2} \leq  \gamma_{k,x}^2 L_{\theta}^2 \|\theta^{k}-\theta^{*}\|^2.
\end{align}
By adding and subtracting $\nabla_x f(x^{*};\theta^{*})$, and by
invoking the Lipschitz continuity of $\nabla_x f(x;\theta)$ in $x$ (A1-2) and the triangle
   inequality, we may derive a bound for {\bf Term 3} as follows:
 \begin{align*}
  \notag\textbf{Term 3}  & \leq 2\gamma_{k,x}  \| (x^{k}-x^{*})-\gamma_{k,x} \nabla_x f(x^{k};\theta^{*}) \| \| \nabla_x f(x^{k};\theta^{k})-\nabla_x f(x^{k};\theta^{*}) \| \\
 \notag &   \leq  2\gamma_{k,x}\| (x^{k}-x^{*})-\gamma_{k,x} (\nabla_x f(x^{k};\theta^{*})-\nabla_x f(x^{*};\theta^{*})) -\gamma_{k,x}\nabla_x f(x^{*};\theta^{*})\| L_{\theta} \|\theta^{k}-\theta^{*}\|  \\
 \notag & \leq 2 \gamma_{k,x}\left( (1+\gamma_{k,x}L_x)\|x^{k}-x^{*}\| +
		 \gamma_{k,x}\|\nabla_x f(x^{*};\theta^{*})\|\right) L_{\theta}  \|\theta^{k}-\theta^{*}\| \\
  & = 2 \gamma_{k,x}L_{\theta} \|x^{k}-x^{*}\| \|\theta^{k}-\theta^{*}\| +   2 \gamma_{k,x}^2L_{\theta}L_x\|x^{k}-x^{*}\| \|\theta^{k}-\theta^{*}\|+ 2 \gamma_{k,x}^2 L_{\theta} \|\nabla_x f(x^{*};\theta^{*})\| \|\theta^{k}-\theta^{*}\|.
 \end{align*}
 By using the fact that $2ab\leq a^2+b^2$, we have further that
  \begin{align}         \label{eq:stoch_monotone_optim_exp_term3}
   \begin{aligned}
  \textbf{Term 3}   & \leq  \gamma_{k,x}^{2-\tau}L_{\theta}^2 \|x^{k}-x^{*}\|^2 +  \gamma_{k,x}^{\tau}\|\theta^{k}-\theta^{*}\|^2 +   \gamma_{k,x}^2L_{\theta}L_x\|x^{k}-x^{*}\|^2 \\
   & +  \gamma_{k,x}^2L_{\theta}L_x\|\theta^{k}-\theta^{*}\|^2 +  \gamma_{k,x}^2 L_{\theta}^2   \|\theta^{k}-\theta^{*}\|^2+ \gamma_{k,x}^2 \|\nabla_x f(x^{*};\theta^{*})\|^2,
   \end{aligned}
\end{align}
where $\tau\in(0,1)$ is chosen to satisfy (A2-2).
Combining \eqref{eq:stoch_monotone_optim_exp},
		  \eqref{eq:stoch_monotone_optim_exp_term1},
		  \eqref{eq:stoch_monotone_optim_exp_term2} and
		  \eqref{eq:stoch_monotone_optim_exp_term3}, we obtain the
		  following bound on the conditional error.
\begin{align} \label{eq:stoch_monotone_optim_exp_final}
\notag \mathbb{E}[\|x^{k+1}-x^{*}\|^2\mid\mathcal{F}_k]
       & \leq (1+ \gamma_{k,x}^{2-\tau}L_{\theta}^2 +
		\gamma_{k,x}^2 (2L_x^2+L_{\theta}L_x) )\|x^{k}-x^{*}\|^2
+
      ( \gamma_{k,x}^{\tau} + \gamma_{k,x}^2 (2L_{\theta}^2+L_{\theta}L_x)) \|\theta^{k}-\theta^{*}\|^2 \\
      & + 3\gamma_{k,x}^2 \| \nabla_x f(x^{*};\theta^{*})\|^2 - 2 \gamma_{k,x} ( f(x^{k};\theta^{*}) - f(x^{*};\theta^{*}) ).
\end{align}
From \eqref{eq:stoch_strongly_optim_exp_final_theta}, we have that
\begin{align}  \label{eq:stoch_monotone_optim_exp_final_theta}
   \mathbb{E}[\|\theta^{k+1}-\theta^{*}\|^2 \mid \mathcal{F}_k]
   & \leq  q_{k,\theta}^2 \|\theta^{k}-\theta^{*}\|^2 + \gamma_{k,\theta}^2 \nu_{\theta}^2,
\end{align}
where $q_{k,\theta} \triangleq \sqrt{1-2\gamma_{k,\theta} \mu_{\theta} + \gamma_{k,\theta}^2 C_{\theta}^2}$.
Choose $\beta_k = \frac{\gamma_{k,x}^{\tau}}{2\gamma_{k,\theta} \mu_{\theta} }$ by (A2-2). Note that by assumption $\beta_{k+1}\leq \beta_k$.
By multiplying
the left hand side of \eqref{eq:stoch_monotone_optim_exp_final_theta} by $\beta_{k+1}$ and adding to the left hand side of $\eqref{eq:stoch_monotone_optim_exp_final}$,
we get
\begin{align}
   \label{eq:stoch_monotone_rand_exp_x_theta}
   & \mathbb{E}[\|x^{k+1}-x^*\|^2\mid\mathcal{F}_k] + \beta_{k+1} \mathbb{E}[\|\theta^{k+1}-\theta^{*}\|^2 \mid\mathcal{F}_k]
    \leq \mathbb{E}[\|x^{k+1}-x^*\|^2\mid\mathcal{F}_k] + \beta_{k}
   \mathbb{E}[\|\theta^{k+1}-\theta^{*}\|^2 \mid\mathcal{F}_k]    \\
         \notag  & \leq (1+ \gamma_{k,x}^{2-\tau}L_{\theta}^2 + \gamma_{k,x}^2 (2L_x^2+L_{\theta}L_x) )\|x^{k}-x^{*}\|^2   +   (\beta_{k}q_{k,\theta}^2+ \gamma_{k,x}^{\tau} + \gamma_{k,x}^2 (2L_{\theta}^2+L_{\theta}L_x)) \|\theta^{k}-\theta^{*}\|^2 \\
       & \notag + 3\gamma_{k,x}^2 \| \nabla_x f(x^{*};\theta^{*})\|^2 + \beta_{k}\gamma_{k,\theta}^2 \nu_{\theta}^2 - 2 \gamma_{k,x} ( f(x^{k};\theta^{*}) - f(x^{*};\theta^{*}) )\\
          \notag   & \leq (1+ \gamma_{k,x}^{2-\tau}L_{\theta}^2 + \gamma_{k,x}^2 (2L_x^2+L_{\theta}L_x) )\|x^{k}-x^{*}\|^2  +   \underbrace{\frac{\beta_{k}q_{k,\theta}^2+ \gamma_{k,x}^{\tau} + \gamma_{k,x}^2 (2L_{\theta}^2+L_{\theta}L_x)}{\beta_{k}}}_{\textbf{Term 4}} \cdot \beta_{k} \|\theta^{k}-\theta^{*}\|^2 \\
     \notag &  + 3\gamma_{k,x}^2 \| \nabla_x f(x^{*};\theta^{*})\|^2 + \beta_{k}\gamma_{k,\theta}^2 \nu_{\theta}^2 - 2 \gamma_{k,x} ( f(x^{k};\theta^{*}) - f(x^{*};\theta^{*}) ).
 	\end{align}
Term 4 on the right hand side of  \eqref{eq:stoch_monotone_rand_exp_x_theta} can be further expanded as
\begin{align}
\begin{aligned}  \label{eq:stoch_monotone_rand_exp_x_theta4}
  \frac{\beta_k q_{k,\theta}^2 + \gamma_{k,x}^{\tau} +\gamma_{k,x}^2 (2L_{\theta}^2+L_{\theta}L_x)}{\beta_k}
  &  = q_{k,\theta}^2 +  \frac{\gamma_{k,x}^{\tau} +\gamma_{k,x}^2 (2L_{\theta}^2+L_{\theta}L_x)}{\beta_k} \\
 & = 1-2\gamma_{k,\theta} \mu_{\theta}  + \gamma_{k,\theta}^2 C_{\theta}^2 + \frac{\gamma_{k,x}^{\tau}}{\beta_k} + \frac{\gamma_{k,x}^2 (2L_{\theta}^2+L_{\theta}L_x)}{\beta_k} \\
     & = 1 + \gamma_{k,\theta}^2 C_{\theta}^2 + 2\gamma_{k,\theta} \gamma_{k,x}^{2-\tau} \mu_{\theta}  (2L_{\theta}^2+L_{\theta}L_x).
     \end{aligned}
\end{align}
Combining \eqref{eq:stoch_monotone_rand_exp_x_theta} and \eqref{eq:stoch_monotone_rand_exp_x_theta4}, we get
\begin{align*}
& \quad \mathbb{E}[\|x^{k+1}-x^*\|^2\mid\mathcal{F}_k] + \beta_{k+1}
\mathbb{E}[\|\theta^{k+1}-\theta^{*}\|^2 \mid\mathcal{F}_k]\\
& \leq (1+ \gamma_{k,x}^{2-\tau}L_{\theta}^2 + \gamma_{k,x}^2 (2L_x^2+L_{\theta}L_x) )\|x^{k}-x^{*}\|^2   +   ( 1 + \gamma_{k,\theta}^2 C_{\theta}^2 + 2\gamma_{k,\theta} \gamma_{k,x}^{2-\tau} \mu_{\theta}  (2L_{\theta}^2+L_{\theta}L_x)  ) \beta_{k} \|\theta^{k}-\theta^{*}\|^2 \\
  \notag    & \quad\ + 3\gamma_{k,x}^2 \| \nabla_x f(x^{*};\theta^{*})\|^2 + \beta_{k}\gamma_{k,\theta}^2 \nu_{\theta}^2 - 2 \gamma_{k,x} ( f(x^{k};\theta^{*}) - f(x^{*};\theta^{*}) )\\
   \notag&\leq ( 1 + \gamma_{k,\theta}^2 C_{\theta}^2 +
		   2\gamma_{k,\theta} \gamma_{k,x}^{2-\tau} \mu_{\theta}
		   (2L_{\theta}^2+L_{\theta}L_x)  ) (\|x^{k}-x^{*}\|^2 +
			   \beta_{k} \|\theta^{k}-\theta^{*}\|^2 )    + (\gamma_{k,x}^{2-\tau}L_{\theta}^2 + \gamma_{k,x}^2 (2L_x^2+L_{\theta}L_x) )\|x^{k}-x^{*}\|^2 \\
       & \quad\ + 3\gamma_{k,x}^2 \| \nabla_x f(x^{*};\theta^{*})\|^2 + \beta_{k}\gamma_{k,\theta}^2 \nu_{\theta}^2 - 2 \gamma_{k,x} ( f(x^{k};\theta^{*}) - f(x^{*};\theta^{*}) ).
\end{align*}
We define the following:
\begin{align*}
  u_k&\triangleq\gamma_{k,\theta}^2 C_{\theta}^2 + 2\gamma_{k,\theta} \gamma_{k,x}^{2-\tau} \mu_{\theta}  (2L_{\theta}^2+L_{\theta}L_x),  \sigma_k \triangleq 2 \gamma_{k,x} ( f(x^{k};\theta^{*}) - f(x^{*};\theta^{*}) ),\\
  \mbox{ and } \rho_k& \triangleq (\gamma_{k,x}^{2-\tau}L_{\theta}^2 +
		  \gamma_{k,x}^2 (2L_x^2+L_{\theta}L_x) )\|x^{k}-x^{*}\|^2 +
  3\gamma_{k,x}^2 \| \nabla_x f(x^{*};\theta^{*})\|^2 +
  \beta_{k}\gamma_{k,\theta}^2 \nu_{\theta}^2.
\end{align*}
Then, we have
\begin{align*}
 \notag  & \quad \ \mathbb{E}[\|x^{k+1}-x^*\|^2\mid\mathcal{F}_k] + \beta_{k+1} \mathbb{E}[\|\theta^{k+1}-\theta^{*}\|^2 \mid\mathcal{F}_k]  \leq (1+u_k)(\|x^{k}-x^{*}\|^2 + \beta_{k} \|\theta^{k}-\theta^{*}\|^2 )  + \rho_k - \sigma_k.
\end{align*}
By boundedness of $X$ and (A2-2), we have that
$\sum_{k=0}^{\infty}u_k<\infty$ and $\sum_{k=0}^{\infty}\rho_k<\infty$.
So, by Lemma \ref{lem:supermartingale2} we get that there exists a random variable $V$ such that
    $\|x^{k}-x^{*}\|^2 + \beta_{k} \|\theta^{k}-\theta^{*}\|^2  \to V$
	in an almost sure sense as $k\to\infty$ and
	$\sum_{k=0}^{\infty}\sigma_k=\sum_{k=0}^{\infty} 2\gamma_{k,x}(
			f(x^{k};\theta^{*}) - f(x^{*};\theta^{*}) ) < \infty.$

By (A2-2), Lemma \ref{lem:supermartingale} and \eqref{eq:stoch_monotone_optim_exp_final_theta}, we can get that $\|\theta^{k}-\theta^{*}\|  \to 0$  $a.s.$ as $k\to\infty$. Thus, it follows that $\|x^{k}-x^{*}\|  \to V$  $a.s.$  $\textrm{as } k\to\infty$.
Since $\sum_{k=0}^{\infty}\gamma_{k,x} = \infty$, we get $\liminf_{k\to\infty} f(x^{k};\theta^{*}) = f(x^{*};\theta^{*})$ $a.s.$ $\textrm{as } k\to\infty$.
Since the set $X$ is closed, all accumulation points of $\{x^k\}$ lie in $X$. Furthermore, since $f(x^{k};\theta^{*})\to f(x^{*};\theta^{*})$ along a subsequence $a.s.$,  by continuity of $f$ it follows that $\{x^k\}$ has a subsequence converging   $a.s.$ to some point in $X$, say $\tilde{x}$, which satisfies $f(\tilde{x};\theta^{*})=f(x^{*};\theta^{*})$. That means $\tilde{x}$ is some random point in $X^*$.  Moreover, since $\|x^{k}-x^{*}\|$ is convergent for any $x^*\in X^*$ $a.s.$, the entire sequence $\{x^k\}$ converges to some random point in $X^*$ $a.s.$
\end{proof}

\subsection{Diminishing and constant steplength rate analysis}
\label{sec:II.III}
While the previous section focused on the almost sure convergence of the
prescribed learning and computational schemes, a natural question is
whether  one can develop rate statements. We begin with an
examination of the global rate of convergence and show that
$\Oscr(1/K)$ rate estimate  is derived for an upper bound on the
mean-squared error in the solution $x_K$ when $f(\bullet;\theta^*)$ is strongly
convex in $(\bullet)$ and $K$ represents the number of steps, consistent
with the result obtained for stochastic approximation
({cf.~\cite{Nemirovski09,Shapiro09lecturesSA}}). In addition, it is seen that when the
function loses strong convexity, an analogous rate estimate is available by
using averaging, akin to an approach first employed in
\cite{Polyak92acceleration}, where longer stepsizes were suggested with
consequent averaging of the obtained iterates.\\

\begin{prop}[{\bf Rate estimates for strongly convex $f$}] \label{thm:stoch_strongly_optim_error}
Suppose (A\ref{assump:stoch_strongly_optim}-1) and
(A\ref{assump:filtration}) hold. Suppose $\gamma_{x,k}=\lambda_{x}/k$
and $\gamma_{\theta,k}=\lambda_{\theta}/k$ with $\lambda_{x}>1/\mu_{x}$
and $\lambda_{\theta}>1/(2\mu_{\theta})$.  Let
$\mathbb{E}[\|\nabla_x f(x^{k};\theta^{k})+ w^k \|^2]\leq M^2$  and
$\mathbb{E}[\|\nabla_{\theta} g(\theta^{k})+ v^k \|^2]\leq M_{\theta}^2$  for all
$x^k\in X$ and $\theta^k\in\Theta$.
Let $\{x^k,\theta^k\}$
be computed via Algorithm \ref{alg:stoch_strongly_optim}.  Then, the
following hold after $K$ iterations:
\begin{align*}
  \mathbb{E}[\|\theta^{K} - \theta^{*} \|^2 ] \leq
  \frac{Q_{\theta}(\lambda_{\theta})}{K} \mbox{ and }
   \mathbb{E}[\|x^{K} - x^{*} \|^2 ] \leq \frac{Q_{x}(\lambda_{x})}{K},
\end{align*}
\begin{align*}
 {\it \mbox{where }} Q_{\theta}(\lambda_{\theta})&\triangleq \max\left\{ \lambda_{\theta}^2 M_{\theta}^2 (2\mu_{\theta}\lambda_{\theta} - 1)^{-1}, \mathbb{E}[\|\theta^{1} - \theta^{*} \|^2 ] \right\},\\
  Q_{x}(\lambda_{x})&\triangleq \max\left\{ \lambda_{x}^2 \widetilde{M}^2
(\mu_{x}\lambda_{x} - 1)^{-1}, \mathbb{E}[\|x^{1} - x^{*} \|^2 ]
\right\}, {\it \mbox{ and }} \widetilde{M} \triangleq \sqrt{ M^2 + \frac{L_{\theta}^2Q_{\theta}(\lambda_{\theta})}{\mu_x \lambda_{x}} }.
\end{align*}
\end{prop}
\vspace{0.2in}
\begin{proof}
Suppose $A_k \triangleq \frac{1}{2} \|x^k - x^*\|^2$ and $a_k \triangleq
\mathbb{E}[A_k]$. Then, $A_{k+1}$ may be bounded as follows by using the
non-expansivity of the Euclidean projector:
\begin{align}
\begin{aligned}  \label{alg:stoch_grad_grad_strongly_constant_A}
   A_{k+1}  & = \frac{1}{2}\|x^{k+1}-x^*\|^2  = \frac{1}{2}\left\| \Pi_{X}
\left(x^{k}-\gamma_{x,k}(\nabla_x f(x^{k};\theta^{k}) + w^k)  \right) - \Pi_{X} (
		x^{*}) \right\|^2 \\
   &\leq\frac{1}{2} \| x^{k}-x^{*} -
 \gamma_{x,k}(\nabla_x f(x^{k};\theta^{k})+ w^k) \|^2 \\
	 & = A_k + \frac{1}{2} \gamma_{x,k}^2 \|\nabla_x f(x^{k};\theta^{k})+ w^k \|^2 - \gamma_{x,k} (x^{k}-x^{*})^T ( \nabla_x f(x^{k};\theta^{k})+ w^k ).
\end{aligned}
\end{align}
Note that $\mathbb{E}[(x^{k}-x^{*})^Tw^k] = \mathbb{E}[\mathbb{E}[(x^{k}-x^{*})^Tw^k|\mathcal{F}_k]]=\mathbb{E}[(x^{k}-x^{*})^T\mathbb{E}[w^k|\mathcal{F}_k]]=0$.
By taking expectations on both sides of
\eqref{alg:stoch_grad_grad_strongly_constant_A} and by invoking the
bounds
$\mathbb{E}[\|\nabla_x f(x^{k};\theta^{k})+ w^k \|^2]\leq M^2$  and
$\mathbb{E}[\|\nabla_{\theta} g(\theta^{k})+ v^k \|^2]\leq M_{\theta}^2$,
it follows that
\begin{align} \label{alg:stoch_grad_grad_strongly_constant_a}
  a_{k+1}
   & \leq a_k + \frac{1}{2} \gamma_{x,k}^2 M^2 - \gamma_{x,k} \mathbb{E}[(x^{k}-x^{*})^T  \nabla_x f(x^{k};\theta^{k})].
\end{align}
But $f(x;\theta)$ is strongly convex in $x$ with constant $\mu_x$
for every $\theta \in \Theta$, leading to the following
expression:
\begin{align}
\begin{aligned} \label{alg:stoch_grad_grad_strongly_constant_a_term2}
 \mathbb{E}[(x^{k}-x^{*})^T  \nabla_x f(x^{k};\theta^{k})]
  &= \mathbb{E}[(x^{k}-x^{*})^T  ( \nabla_x f(x^{k};\theta^{k}) -
		  \nabla_x f(x^{*};\theta^{k}) ) ]    \\
	 &  + \mathbb{E}[(x^{k}-x^{*})^T
  ( \nabla_x f(x^{*};\theta^{k}) - \nabla_x f(x^{*};\theta^{*}) ) ]   +
  \mathbb{E}[(x^{k}-x^{*})^T   \nabla_x f(x^{*};\theta^{*})  ]   \\
  & \geq\mu_x\mathbb{E}[\|x^{k}-x^{*}\|^2] + \mathbb{E}[(x^{k}-x^{*})^T
  ( \nabla_x f(x^{*};\theta^{k}) - \nabla_x f(x^{*};\theta^{*}) ) ].
\end{aligned}
\end{align}
Combining \eqref{alg:stoch_grad_grad_strongly_constant_a} and \eqref{alg:stoch_grad_grad_strongly_constant_a_term2}, we get
\begin{align}
 \notag a_{k+1}
    &\leq(1-2\gamma_{x,k}\mu_x)a_k + \frac{1}{2} \gamma_{x,k}^2 M^2  -
	\gamma_{x,k} \mathbb{E}[(x^{k}-x^{*})^T  ( \nabla_x
			f(x^{*};\theta^{k}) - \nabla_x f(x^{*};\theta^{*}) ) ]\\
\notag   &\leq(1-2\gamma_{x,k}\mu_x)a_k + \frac{1}{2} \gamma_{x,k}^2 M^2
     +\frac{1}{2} \gamma_{x,k}\mu_x \mathbb{E}[\|x^{k}-x^{*}\|^2] +
   \frac{1}{2}\frac{\gamma_{x,k}}{\mu_x}\mathbb{E}[ \|\nabla_x
   f(x^{*};\theta^{k}) - \nabla_x f(x^{*};\theta^{*}) \|^2 ]\\
   &\leq(1-\gamma_{x,k}\mu_x)a_k + \frac{1}{2} \gamma_{x,k}^2 M^2  + \frac{1}{2}\frac{\gamma_{x,k}}{\mu_x}L_{\theta}^2\mathbb{E}[\|\theta^{k} - \theta^{*} \|^2 ].
\label{eq:strongly_constant_a}
\end{align}
Suppose $\gamma_{\theta,k}=\lambda_{\theta}/k$. Since the function $g(\theta)$ is strongly convex, we can use the standard rate estimate (cf. inequality (5.292) in \cite{Shapiro09lecturesSA}) to get the following
\begin{align} \label{eq:error_theta_Q}
  \mathbb{E}[\|\theta^{k} - \theta^{*} \|^2 ] \leq \frac{Q_{\theta}(\lambda_{\theta})}{k},
\end{align}
where $Q_{\theta}(\lambda_{\theta})\triangleq \max\left\{ \lambda_{\theta}^2 M_{\theta}^2 (2\mu_{\theta}\lambda_{\theta} - 1)^{-1}, \mathbb{E}[\|\theta^{1} - \theta^{*} \|^2 ] \right\}$ with $\lambda_{\theta}>1/(2\mu_{\theta})$.
Suppose $\gamma_{x,k}=\lambda_{x}/k$, allowing us to claim the
following:
\begin{align*}
     a_{k+1}
      &\leq\left(1-\frac{\mu_x\lambda_{x}}{k}\right)a_k + \frac{1}{2}
	  \frac{\lambda_{x}^2M^2}{k^2}   +
	  \frac{1}{2}\frac{\lambda_{x}L_{\theta}^2Q_{\theta}(\lambda_{\theta})}{\mu_x
		  k^2}  = \left(1-\frac{\mu_x\lambda_{x}}{k}\right)a_k + \frac{1}{2} \frac{\lambda_{x}^2\widetilde{M}^2}{k^2},
\end{align*}
where $\widetilde{M}  \triangleq \sqrt{ M^2 + \frac{L_{\theta}^2Q_{\theta}(\lambda_{\theta})}{\mu_x \lambda_{x}} }$.
By assuming that $\lambda_{x}>1/\mu_{x}$, the result follows by
observing that
\begin{align*}
   \mathbb{E}[\|x^{k} - x^{*} \|^2 ] \leq \frac{Q_{x}(\lambda_{x})}{k},
\end{align*}
where $Q_{x}(\lambda_{x})\triangleq \max\left\{ \lambda_{x}^2
\widetilde{M}^2 (\mu_{x}\lambda_{x} - 1)^{-1}, \mathbb{E}[\|x^{1} -
x^{*} \|^2 ] \right\}$.
\end{proof}\\
{\noindent {\bf Remark:}
Notice that here we assume that $f$ and $g$ are both smooth and strongly
	convex. A more general framework is that of composite objectives
	where the objective is a sume of nonsmooth and smooth stochastic
	components. Lan~\cite{Lan12} proposed the accelerated stochastic
	approximation (AC-SA) algorithm for solving stochastic composite
	optimization (SCO) problems and proved that it achieves the optimal
	rate. In related work, Ghadimi and Lan \cite{GhadimiL12,GhadimiL13}
propose a multi-stage AC-SA algorithm, which possesses an optimal rate
	of convergence for solving strongly convex SCO problems in terms of
	the dependence on different problem parameters. While this problem
	class is beyond the current scope, this approach may aid in
	refinement of the constants in the Proposition
	\ref{thm:stoch_strongly_optim_error} in some regimes.
}

\indent A shortcoming of the previous result is the need for strong convexity of
$f(x,\theta)$ in $x$ for every $\theta \in \Theta$. In our next result,
	we weaken this requirement and allow for a merely convex $f$,
	extending the optimal constant stepsize result in
	\cite{Shapiro09lecturesSA}. Specifically, given a prescribed number
	of iterations, say $K$, the optimal ``constant stepsize'' derives
	the error minimizing steplength; in other words, $\gamma_k = \gamma$ for $1 \leq k \leq K$. This is in contrast with the
	constant stepsize result presented in Proposition
	\ref{prop:error_optim}, where $\gamma_k = \gamma$ for all $k$.
	steps. The following Lipschitzian assumption is imposed on the
	function $f(x;\theta)$.

\begin{assumption}[A\ref{assump:stoch_monotone_optim_error}] \label{assump:stoch_monotone_optim_error}
Suppose the following holds in addition to (A1-2).
\begin{enumerate}
  \item[(i)] For every {$x \in X$},  $f(x;\theta)$ is Lipschitz continuous in $\theta$ with
  constant $D_{\theta}$.\\
 \end{enumerate}
\end{assumption}

\begin{thm} [{\bf Rate estimates under convexity of $f$}] \label{thm:stoch_monotone_optim_error}
Suppose (A\ref{assump:filtration}) and (A\ref{assump:stoch_monotone_optim_error}) hold.
Suppose $\mathbb{E}[\|x^{k}-x^{*}\|^2]\leq M_x^2$,
$\mathbb{E}[\|\nabla_x f(x^{k};\theta^{k})+ w^k \|^2]\leq M^2$ and
$\mathbb{E}[\|\nabla_{\theta} g(\theta^{k})+ v^k \|^2]\leq M_{\theta}^2$
for all $x^{k}\in X$ and $\theta^{k}\in\Theta$.  Let $\{x^k,\theta^k\}$
be computed via Algorithm \ref{alg:stoch_strongly_optim}.  For $1\leq
i,t\leq k$, we define $v_t\triangleq\frac{\gamma_{x,t}}{\sum_{s=i}^k
	\gamma_{x,s}}$, $\tilde{x}_{i,k}\triangleq \sum_{t=i}^k v_t x^t$ and
	$D_X\triangleq\max_{x\in X}\|x-x^1\|$.  Suppose
 for $1\leq t\leq K$, $\gamma_x$ is defined as follows:
$$\gamma_x = \sqrt{\frac{4D_X^2  + L_{\theta}^2Q_{\theta}(\lambda_{\theta}) (1+\ln K)}{(M^2 +  M_x^2)K}},$$
where $Q_{\theta}(\lambda_{\theta})\triangleq \max\left\{ \lambda_{\theta}^2 M_{\theta}^2 (2\mu_{\theta}\lambda_{\theta} - 1)^{-1}, \mathbb{E}[\|\theta^{1} - \theta^{*} \|^2 ] \right\}$,
and $\gamma_{\theta,k}=\lambda_{\theta}/K$ with
$\lambda_{\theta}>1/(2\mu_{\theta})$.  Then the following holds for
$1\leq i\leq K$:
\begin{align}
\notag  \left| \mathbb{E}[ f(\tilde{x}_{i,K};\theta^{K}) - f(x^{*};\theta^{*})  ]\right|
  &\leq  \frac{\sqrt{Q_{\theta}(\lambda_{\theta})}D_{\theta} +C_{i,K}\sqrt{ B_K }}{\sqrt{K}},
\end{align}
where
$C_{i,K}=\frac{K}{K-i+1}$ and $B_K=(4D_X^2  + L_{\theta}^2Q_{\theta}(\lambda_{\theta}) (1+\ln K))(M^2 +  M_x^2)$.
\end{thm}
\vspace{0.1in}
\begin{proof}
By using the same notation in Proposition  \ref{thm:stoch_strongly_optim_error}, we have from \eqref{alg:stoch_grad_grad_strongly_constant_a} that
\begin{align}
\notag  a_{k+1} &
    \leq a_k + \frac{1}{2} \gamma_{x,k}^2 M^2 - \gamma_{x,k}
   \mathbb{E}[(x^{k}-x^{*})^T  \nabla_x f(x^{k};\theta^{k})] \\
  & \leq a_k + \frac{1}{2} \gamma_{x,k}^2 M^2 - \gamma_{x,k}
   \mathbb{E}[(x^{k}-x^{*})^T \nabla_x f(x^{k};\theta^{*})] -   \gamma_{x,k} \mathbb{E}[(x^{k}-x^{*})^T (\nabla_x f(x^{k};\theta^{k})
		   - \nabla_x f(x^{k};\theta^{*}) ) ].
 \label{alg:stoch_grad_grad_monotone_constant_a}
\end{align}
Note that $f(x;\theta)$ is convex in $x$ for every $\theta \in
\Theta$, allowing us to leverage the gradient inequality.
\begin{align} \label{alg:stoch_grad_grad_monotone_constant_a_term2}
  \mathbb{E}[(x^{k}-x^{*})^T  \nabla_x f(x^{k};\theta^{*})]
     & \geq \mathbb{E}[f(x^{k};\theta^{*}) - f(x^{*};\theta^{*}) ].
\end{align}
Combining \eqref{alg:stoch_grad_grad_monotone_constant_a} and
\eqref{alg:stoch_grad_grad_monotone_constant_a_term2}, we obtain the
following:
\begin{align*}
\begin{aligned}
    a_{k+1}
    &\leq a_k + \frac{1}{2} \gamma_{x,k}^2 M^2 - \gamma_{x,k}
	\mathbb{E}[f(x^{k};\theta^{*}) - f(x^{*};\theta^{*}) ]  -	\gamma_{x,k} \mathbb{E}[(x^{k}-x^{*})^T (\nabla_x
			f(x^{k};\theta^{k}) - \nabla_x f(x^{k};\theta^{*}) ) ].
\end{aligned}
\end{align*}
This allows for constructing the following bounds:
\begin{align}
 \notag  \gamma_{x,k} \mathbb{E}[f(x^{k};\theta^{*}) - f(x^{*};\theta^{*})  ]
&\leq
      a_k -a_{k+1}+ \frac{1}{2} \gamma_{x,k}^2 M^2 - \gamma_{x,k}
\mathbb{E}[(x^{k}-x^{*})^T (\nabla_x f(x^{k};\theta^{k}) - \nabla_x f(x^{k};\theta^{*}) ) ] \\
   \notag  &\leq a_k -a_{k+1}+ \frac{1}{2} \gamma_{x,k}^2 M^2  +
\frac{1}{2}\gamma_{x,k}^{2} \mathbb{E}[\|x^{k}-x^{*}\|^2] +\frac{1}{2}
\mathbb{E}[\|\nabla_x f(x^{k};\theta^{k}) - \nabla_x f(x^{k};\theta^{*}) \|^2 ] \\
\notag     &\leq a_k -a_{k+1}+ \frac{1}{2} \gamma_{x,k}^2 M^2 + \frac{1}{2}\gamma_{x,k}^{2} M_x^2   +\frac{1}{2} L_{\theta}^2\mathbb{E}[\|\theta^{k} -  \theta^{*} \|^2] \\
    &\leq a_k -a_{k+1}+ \frac{1}{2} \gamma_{x,k}^2 (M^2 +  M_x^2) +\frac{1}{2} \frac{L_{\theta}^2Q_{\theta}(\lambda_{\theta})}{k},
\label{eq:convex_constant_a}
\end{align}
where the second inequality follows from the fact that $2ab\leq
a^2+b^2$, the third inequality follows from the boundedness of
$\mathbb{E}[\|x^{k}-x^{*}\|^2]$ and Lipschitz continuity of $\nabla_x f(x;\theta)$ in $\theta$, and the last inequality follows from  \eqref{eq:error_theta_Q}.
As a result, for $1\leq i\leq k$, we have the following:
\begin{align}
\notag \sum_{t=i}^k\gamma_{x,t} \mathbb{E}[f(x^{t};\theta^{*}) - f(x^{*};\theta^{*})  ]
& \leq \sum_{t=i}^k (a_t  -a_{t+1} )+ \frac{1}{2} \sum_{t=i}^k \gamma_{x,t}^2 (M^2 +  M_x^2) +  \frac{1}{2} \sum_{t=i}^k \frac{L_{\theta}^2Q_{\theta}(\lambda_{\theta})}{t}   \\
 \label{eq:optim_bound}
\notag &\leq  a_i + \frac{1}{2} \sum_{t=i}^k \gamma_{x,t}^2 (M^2 +  M_x^2) +  \frac{1}{2} \sum_{t=i}^k \frac{L_{\theta}^2Q_{\theta}(\lambda_{\theta})}{t} \\
 &\leq  a_i + \frac{1}{2} \sum_{t=i}^k \gamma_{x,t}^2 (M^2 +  M_x^2) +  \frac{1}{2} L_{\theta}^2Q_{\theta}(\lambda_{\theta}) (1+\ln k).
\end{align}
Next, we define $v_t\triangleq\frac{\gamma_{x,t}}{\sum_{s=i}^k
	\gamma_{x,s}}$ and $D_X\triangleq\displaystyle \max_{x\in X}\|x-x^1\|$.
The following holds invoking these definitions:
\begin{align}
\begin{aligned}  \label{alg:stoch_grad_grad_monotone_constant_ab_final_average}
  \mathbb{E}\left[\sum_{t=i}^kv_t f(x^{t};\theta^{*}) -
   f(x^{*};\theta^{*})  \right]
     &\leq  \frac{a_i  +\frac{1}{2} \sum_{t=i}^k \gamma_{x,t}^2 (M^2 +  M_x^2) +  \frac{1}{2} L_{\theta}^2Q_{\theta}(\lambda_{\theta}) (1+\ln k)}{\sum_{t=i}^k\gamma_{x,t}}.
\end{aligned}
\end{align}
Next, we consider points given by $\tilde{x}_{i,k}\triangleq \sum_{t=i}^k v_t x^t$. By convexity of $X$, we have that $\tilde{x}_{i,k}\in X$ and by the convexity of $f(x;\theta^{*})$ in $x$, we have
$f(\tilde{x}_{i,k};\theta^{*})\leq \sum_{t=i}^k v_t f(x^t)$.
From \eqref{alg:stoch_grad_grad_monotone_constant_ab_final_average} and by noting that $a_1\leq\frac{1}{2}D_X^2$ and $a_i\leq 2 D_X^2$ for $i>1$, we obtain the following for $1\leq i \leq k$
\begin{align}  \label{alg:stoch_grad_grad_monotone_constant_ab_final_average_convex}
 \quad\ \mathbb{E}[ f(\tilde{x}_{i,k};\theta^{*}) - f(x^{*};\theta^{*})  ]  \leq  \frac{4D_X^2   +\sum_{t=i}^k \gamma_{x,t}^2 (M^2 +  M_x^2) + L_{\theta}^2Q_{\theta}(\lambda_{\theta}) (1+\ln k)}{2\sum_{t=i}^k\gamma_{x,t}}.
\end{align}
Suppose $\gamma_{x,t}=\gamma_{x}$ for $t=1,\ldots,k$. Then, it follows that
\begin{align} \label{optim_bound2}
   \mathbb{E}[ f(\tilde{x}_{1,k};\theta^{*}) - f(x^{*};\theta^{*})  ]  &\leq  \frac{4D_X^2   + k\gamma_{x}^{2 } (M^2 +  M_x^2)+ L_{\theta}^2Q_{\theta}(\lambda_{\theta}) (1+\ln k)}{2k\gamma_{x}}.
\end{align}
By minimizing the right hand side in $\gamma_x>0$, we obtain that
$$\gamma_x = \sqrt{\frac{4D_X^2  + L_{\theta}^2Q_{\theta}(\lambda_{\theta}) (1+\ln k)}{(M^2 +  M_x^2)k}}.$$
This implies the following bound:
\begin{align*}
   \mathbb{E}[ f(\tilde{x}_{1,k};\theta^{*}) - f(x^{*};\theta^{*})  ]
   &\leq   \sqrt{\frac{ B_k}{k}},
\end{align*}
where $B_k \triangleq (4D_X^2  + L_{\theta}^2Q_{\theta}(\lambda_{\theta}) (1+\ln k))(M^2 +  M_x^2)$.
Next, we can also claim that for $1\leq i\leq k$,
\begin{align} \label{eq:error_monotone_f}
   \mathbb{E}[ f(\tilde{x}_{i,k};\theta^{*}) - f(x^{*};\theta^{*})  ]  &\leq  C_{i,k}\sqrt{\frac{ B_k}{k}},
\end{align}
where $C_{i,k}=\frac{k}{k-i+1}$.
Thus, by employing \eqref{eq:error_theta_Q}, \eqref{eq:error_monotone_f}
and the Lipschitz continuity of $f(x;\theta)$ in $\theta$, we have the
required result:
\begin{align*}
{\left|\mathbb{E}[ f(\tilde{x}_{i,k};\theta^{k}) - f(x^{*};\theta^{*})  ]\right|}
\notag &\leq  \left|\mathbb{E}[ f(\tilde{x}_{i,k};\theta^{k}) - f(\tilde{x}_{i,k};\theta^{*})  ]\right| +
  \left|\mathbb{E}[ f(\tilde{x}_{i,k};\theta^{*}) - f(x^{*};\theta^{*})  ]\right| \\
\notag  &\leq D_{\theta}\mathbb{E}[\|\theta^{k} - \theta^{*} \|] +
  \mathbb{E}[ f(\tilde{x}_{i,k};\theta^{*}) - f(x^{*};\theta^{*})  ]\\
 \notag &\leq   \frac{\sqrt{Q_{\theta}(\lambda_{\theta})}D_{\theta}}{\sqrt{k}}+
  \mathbb{E}[ f(\tilde{x}_{i,k};\theta^{*}) - f(x^{*};\theta^{*})  ] \leq  \frac{\sqrt{Q_{\theta}(\lambda_{\theta})}D_{\theta} +C_{i,k}\sqrt{ B_k }}{\sqrt{k}}.
\end{align*}
\end{proof}

\noindent {\bf Remark:} In effect, in the context of learning and optimization, the averaging
approach leads to a complexity bound given loosely by
$$\Oscr\left(\frac{a_{\theta}}{\sqrt{K}}+\underbrace{\frac{b+c_{\theta}\sqrt{\ln(K)}}{\sqrt{K}}}_{\rm
		\tiny \mbox{degradation from learning }}\right),$$
where $a_{\theta}, b, c_{\theta}$ are suitably defined.  If $\theta^*$
is available, then $a_{\theta}, c_{\theta} = 0$, leading to the standard
bound of $\Oscr(1/\sqrt{K})$. While it is not surprising that
the requirement to learn $\theta^*$ imposes a degradation, it appears
that this degradation is not severe. However, by changing the averaging
window, this degradation disappears from a rate standpoint.
Specifically, the next result is a  corollary of Theorem
\ref{thm:stoch_monotone_optim_error} and uses a modified averaging
window, as seen in \cite{Nemirovski09}.
\begin{cor} [{\bf Rate estimates under convexity of $f$}]  \label{cor:optim}
Suppose (A\ref{assump:filtration}) and (A\ref{assump:stoch_monotone_optim_error}) hold.
Suppose $\mathbb{E}[\|x^{k}-x^{*}\|^2]\leq M_x^2$,
$\mathbb{E}[\|\nabla_x f(x^{k};\theta^{k})+ w^k \|^2]\leq M^2$ and
$\mathbb{E}[\|\nabla_{\theta} g(\theta^{k})+ v^k \|^2]\leq M_{\theta}^2$
for all $x^{k}\in X$ and $\theta^{k}\in\Theta$.  Let $\{x^k,\theta^k\}$
be computed via Algorithm \ref{alg:stoch_strongly_optim}. Let $k$ be a positive even number.  For $k/2\leq
 t\leq k$, we define $v_t\triangleq\frac{\gamma_{x,t}}{\sum_{s=k/2}^k
	\gamma_{x,s}}$, $\tilde{x}_{k/2,k}\triangleq \sum_{t=k/2}^k v_t x^t$ and
	$D_X\triangleq\max_{x\in X}\|x-x^1\|$.  Suppose
 for $1\leq t\leq K$, $\gamma_x$ is defined as follows:
$$\gamma_x = \sqrt{\frac{4D_X^2  + L_{\theta}^2Q_{\theta}(\lambda_{\theta}) (1+\ln 2)}{(M^2 +  M_x^2)k}},$$
where $Q_{\theta}(\lambda_{\theta})\triangleq \max\left\{ \lambda_{\theta}^2 M_{\theta}^2 (2\mu_{\theta}\lambda_{\theta} - 1)^{-1}, \mathbb{E}[\|\theta^{1} - \theta^{*} \|^2 ] \right\}$,
and $\gamma_{\theta,k}=\lambda_{\theta}/K$ with
$\lambda_{\theta}>1/(2\mu_{\theta})$.  Then the following holds:
\begin{align}
\notag  \left| \mathbb{E}[ f(\tilde{x}_{K/2,K};\theta^{K}) - f(x^{*};\theta^{*})  ]\right|
  &\leq  \frac{\sqrt{Q_{\theta}(\lambda_{\theta})}D_{\theta} +2\sqrt{ B  }}{\sqrt{K}},
\end{align}
where
$B  \triangleq (4D_X^2  + L_{\theta}^2Q_{\theta}(\lambda_{\theta}) (1+\ln 2))(M^2 +  M_x^2)$.
\end{cor}

\begin{proof}
When $i=k/2$ where $k$ is a positive even number,
the second inequality of  \eqref{eq:optim_bound} becomes
\begin{align}
\notag \sum_{t=k/2}^k\gamma_{x,t} \mathbb{E}[f(x^{t};\theta^{*}) - f(x^{*};\theta^{*})  ]
\notag &\leq  a_{k/2} + \frac{1}{2} \sum_{t=k/2}^k \gamma_{x,t}^2 (M^2 +  M_x^2) +  \frac{1}{2} \sum_{t=k/2}^k \frac{L_{\theta}^2Q_{\theta}(\lambda_{\theta})}{t} \\
\notag &\leq  a_{k/2} + \frac{1}{2} \sum_{t=k/2}^k \gamma_{x,t}^2 (M^2 +  M_x^2) +  \frac{1}{2}L_{\theta}^2Q_{\theta}(\lambda_{\theta}) \left[ \sum_{t=1}^k \frac{1}{t} - \sum_{t=k/2-1}^k \frac{1}{t} \right] \\
\notag &\leq  a_{k/2} + \frac{1}{2} \sum_{t=k/2}^k \gamma_{x,t}^2 (M^2 +  M_x^2) +  \frac{1}{2}L_{\theta}^2Q_{\theta}(\lambda_{\theta}) \left[ 1+\ln(k) - \ln(k/2) \right] \\
 &\leq  a_{k/2} + \frac{1}{2} \sum_{t=k/2}^k \gamma_{x,t}^2 (M^2 +  M_x^2) +  \frac{1}{2}L_{\theta}^2Q_{\theta}(\lambda_{\theta}) ( 1+\ln 2 ).
\end{align}
Then,  \eqref{optim_bound2} becomes
\begin{align*}
   \mathbb{E}[ f(\tilde{x}_{1,k};\theta^{*}) - f(x^{*};\theta^{*})  ]  &\leq  \frac{4D_X^2   + k\gamma_{x}^{2 } (M^2 +  M_x^2)+ L_{\theta}^2Q_{\theta}(\lambda_{\theta}) (1+\ln 2)}{2k\gamma_{x}}.
\end{align*}
By minimizing the right hand side in $\gamma_x>0$, we obtain that
$$\gamma_x = \sqrt{\frac{4D_X^2  + L_{\theta}^2Q_{\theta}(\lambda_{\theta}) (1+\ln 2)}{(M^2 +  M_x^2)k}}.$$
This implies the following bound:
\begin{align*}
   \mathbb{E}[ f(\tilde{x}_{1,k};\theta^{*}) - f(x^{*};\theta^{*})  ]
   &\leq   \sqrt{\frac{ B }{k}},
\end{align*}
where $B  \triangleq (4D_X^2  + L_{\theta}^2Q_{\theta}(\lambda_{\theta}) (1+\ln 2))(M^2 +  M_x^2)$.
Next, we can also claim that,
\begin{align*}
   \mathbb{E}[ f(\tilde{x}_{k/2,k};\theta^{*}) - f(x^{*};\theta^{*})  ]  &\leq  C_{k}\sqrt{\frac{ B }{k}},
\end{align*}
where $C_{ k}=\frac{ k}{k-k/2+1}\leq 2$.
Thus, we have the
required result:
\begin{align*}
{\left|\mathbb{E}[ f(\tilde{x}_{k/2,k};\theta^{k}) - f(x^{*};\theta^{*})  ]\right|}
\notag
  &\leq  \frac{\sqrt{Q_{\theta}(\lambda_{\theta})}D_{\theta} +2\sqrt{ B  }}{\sqrt{k}}.
\end{align*}
\end{proof}

We now present a constant steplength error bound where the steplength is
fixed over the entire algorithm. As mentioned before, this differs from  Theorem
\ref{thm:stoch_monotone_optim_error} in that the number of iterations is
not fixed. Constant steplength statements are particularly relevant in
networked regimes where the coordination of changing steplength
sequences across a collection of agents may prove complicated.
\vspace{0.2in}\\
\begin{prop}[{\bf Constant steplength error bound}] \label{prop:error_optim}
	Suppose (A\ref{assump:filtration}) holds. Suppose $\gamma_{\theta,k} := \gamma_{\theta}$
and $\gamma_{x,k} := \gamma_{x}$.  Suppose $\mathbb{E}[\|x^{k}-x^{*}\|^2]\leq M_x^2$ and
$\mathbb{E}[\|\nabla_x f(x^{k};\theta^{k})+ w^k \|^2]\leq M^2$    for all
$x^k\in X$.
Suppose $A_k \triangleq \frac{1}{2} \|x^k - x^*\|^2$ and $a_k \triangleq \mathbb{E}[A_k]$.
Let $\{x^k,\theta^k\}$
be computed via Algorithm \ref{alg:stoch_strongly_optim}.
\begin{itemize}
   \item[(i)] Suppose (A\ref{assump:stoch_strongly_optim}-1) holds. Then, the following holds:
  \begin{align*}
 \limsup_{k \to \infty}   a_{k} &  \leq  \frac{1}{2\mu_x} \gamma_{x} M^2  +
	 \frac{1}{2\mu^2_x} \frac{ \gamma_\theta
		\nu_{\theta}^2L_{\theta}^2}{2\mu_{\theta} - \gamma_{\theta} C_{\theta}^2};
		\end{align*}
  \item[(ii)] Suppose (A1-2) and
	  (A\ref{assump:stoch_monotone_optim_error}) hold and $0<\tau<1$. Then, the following holds:
  \begin{align*}
 \limsup_{k \to \infty}\left|\mathbb{E}[ f(x^k;\theta^{k}) - f(x^*;\theta^{*})  ]\right|
 & \leq\frac{1}{2} \gamma_{x} M^2 + \frac{1}{2}\gamma_{x}^{1-\tau}  M_x^2 +\frac{1}{2} \gamma_{x}^{\tau-1} L_{\theta}^2 \frac{ \gamma_{ \theta}
		\nu_{\theta}^2}{2\mu_{\theta} - \gamma_{ \theta} C_{\theta}^2} +D_{\theta}\sqrt{\frac{ \gamma_{ \theta}
		\nu_{\theta}^2}{2\mu_{\theta} - \gamma_{ \theta} C_{\theta}^2}}.
 \end{align*}
\end{itemize}
\end{prop}
\vspace{0.1in}
\begin{proof}
By  \eqref{eq:stoch_strongly_optim_exp_final_theta}, we get the
following:
\begin{align*}
   \mathbb{E}[\|\theta^{k+1}-\theta^{*}\|^2 \mid \mathcal{F}_k]
   & \leq q_{k,\theta}^2 \|\theta^{k}-\theta^{*}\|^2 + \gamma_{k,\theta}^2 \nu_{\theta}^2,
\end{align*}
where $q_{k,\theta} \triangleq \sqrt{1-2\gamma_{k,\theta} \mu_{\theta} +
	\gamma_{k,\theta}^2 C_{\theta}^2}$.
Suppose $\gamma_{\theta,k} := \gamma_{\theta}$ is chosen such that $(1-q_{\theta}) < 1$ where $q_{\theta,k} := q_{\theta}$.
By taking the expectation and limit supremum on both sides, we have
\begin{align*}
   \limsup_{k \to \infty}\mathbb{E}[\|\theta^{k+1}-\theta^{*}\|^2]
   & \leq q_{ \theta}^2 \limsup_{k \to \infty} \mathbb{E}[\|\theta^{k}-\theta^{*}\|^2] + \gamma_{ \theta}^2 \nu_{\theta}^2,
\end{align*}
or,
\begin{align} \label{eq:theta_bound}
  \limsup_{k \to \infty} \mathbb{E}[\|\theta^k - \theta^*\|^2] \leq \frac{ \gamma_{ \theta}
		\nu_{\theta}^2}{2\mu_{\theta} - \gamma_{ \theta} C_{\theta}^2}.
\end{align}
\begin{enumerate}
\item[(i)] $f$ is strongly convex:
		From \eqref{eq:strongly_constant_a}, for $\gamma_{x,k} := \gamma_x$ where
	$\gamma_x $ is sufficiently small, we have the following:
\begin{align*}
    a_{k+1} \leq(1-\gamma_{x}\mu_x)a_k + \frac{1}{2} \gamma_{x}^2 M^2  + \frac{1}{2}\frac{\gamma_{x}}{\mu_x}L_{\theta}^2\mathbb{E}[\|\theta^{k} - \theta^{*} \|^2 ].
\end{align*}
It follows that
\begin{align*}
 \limsup_{k \to \infty}   a_{k+1} &  \leq (1-\gamma_{x}\mu_x) \limsup_{k
	 \to \infty} a_k + \frac{1}{2} \gamma_{x}^2 M^2  +
	 \frac{1}{2}\frac{\gamma_{x}}{\mu_x}L_{\theta}^2 \limsup_{k \to
		 \infty} \mathbb{E}[\|\theta^{k} - \theta^{*} \|^2 ] \\
		 		& \leq  (1-\gamma_{x}\mu_x) \limsup_{k
	 \to \infty} a_k + \frac{1}{2} \gamma_{x}^2 M^2  +
	 \frac{1}{2}\frac{\gamma_{x}}{\mu_x}L_{\theta}^2 \frac{ \gamma_\theta
		\nu_{\theta}^2}{2\mu_{\theta} - \gamma_{\theta} C_{\theta}^2}.
		\end{align*}
It follows that
\begin{align*}
 \limsup_{k \to \infty}   a_{k} &  \leq  \frac{1}{2\mu_x} \gamma_{x} M^2  +
	 \frac{1}{2}\frac{1}{\mu^2_x}L_{\theta}^2 \frac{ \gamma_\theta
		\nu_{\theta}^2}{2\mu_{\theta} - \gamma_{\theta} C_{\theta}^2}.
		\end{align*}
\item[(ii)] $f$ is convex:
From \eqref{eq:convex_constant_a}, for $\gamma_{x,k} := \gamma_x$, we have the following:
\begin{align*}
  \gamma_{x} \mathbb{E}[f(x^{k};\theta^{*}) - f(x^{*};\theta^{*})  ]
 &\leq
      a_k -a_{k+1}+ \frac{1}{2} \gamma_{x}^2 M^2 - \gamma_{x}
\mathbb{E}[(x^{k}-x^{*})^T (\nabla_x f(x^{k};\theta^{k}) - \nabla_x f(x^{k};\theta^{*}) ) ]\\
&\leq a_k -a_{k+1}+ \frac{1}{2} \gamma_{x}^2 M^2 +
\frac{1}{2}\gamma_{x}^{2-\tau} \mathbb{E}[\|x^{k}-x^{*}\|^2] \\
&\quad\ +\frac{1}{2}\gamma_{x}^{\tau}
\mathbb{E}[\|\nabla_x f(x^{k};\theta^{k}) - \nabla_x f(x^{k};\theta^{*}) \|^2 ] \\
\notag    &\leq a_k -a_{k+1}+ \frac{1}{2} \gamma_{x}^2 M^2 + \frac{1}{2}\gamma_{x}^{2-\tau} M_x^2 +\frac{1}{2} \gamma_{x}^{\tau} L_{\theta}^2\mathbb{E}[\|\theta^{k} -  \theta^{*} \|^2],
\end{align*}
where $0<\tau<1$.
It follows that
\begin{align*}
\gamma_{x}  \limsup_{k \to \infty} \mathbb{E}[f(x^{k};\theta^{*}) - f(x^{*};\theta^{*})  ]    &\leq \limsup_{k \to \infty}  a_k - \limsup_{k \to \infty}  a_{k+1}+ \frac{1}{2} \gamma_{x}^2 M^2 + \frac{1}{2}\gamma_{x}^{2-  \tau } M_x^2 \\
&\quad\ +\frac{1}{2}\gamma_{x}^{\tau} L_{\theta}^2\limsup_{k \to \infty} \mathbb{E}[\|\theta^{k} -  \theta^{*} \|^2] \\
\notag    &\leq  \frac{1}{2} \gamma_{x}^2 M^2 + \frac{1}{2}\gamma_{x}^{2-\tau} M_x^2 +\frac{1}{2}\gamma_{x}^{\tau} L_{\theta}^2 \frac{ \gamma_{ \theta}
		\nu_{\theta}^2}{2\mu_{\theta} - \gamma_{ \theta} C_{\theta}^2}.
\end{align*}
It follows that
\begin{align*}
 \limsup_{k \to \infty} \mathbb{E}[f(x^{k};\theta^{*}) - f(x^{*};\theta^{*})  ] &  \leq  \frac{1}{2} \gamma_{x} M^2 + \frac{1}{2}\gamma_{x}^{1-\tau}  M_x^2 +\frac{1}{2} \gamma_{x}^{\tau-1} L_{\theta}^2 \frac{ \gamma_{ \theta}
		\nu_{\theta}^2}{2\mu_{\theta} - \gamma_{ \theta} C_{\theta}^2}.
\end{align*}
{By the Lipschitz continuity of $f(x;\theta)$ in $\theta$ (A\ref{assump:stoch_monotone_optim_error}(i)), H$\ddot{\textrm{o}}$lder's inequality  and
	\eqref{eq:theta_bound},  we have
\begin{align*}
 \limsup_{k \to \infty}\left|\mathbb{E}[ f(x^k;\theta^{k}) - f(x^k;\theta^{*})  ]\right|
\notag &\leq    D_{\theta}\limsup_{k \to \infty}\mathbb{E}[\|\theta^{k} - \theta^{*} \|]  \\
\notag &\leq    D_{\theta}\limsup_{k \to \infty}\sqrt{\mathbb{E}[\|\theta^{k} - \theta^{*} \|^2]}\\
\notag &=    D_{\theta}\sqrt{\limsup_{k \to \infty}\mathbb{E}[\|\theta^{k} - \theta^{*} \|^2]}\\
&\leq D_{\theta}\sqrt{\frac{ \gamma_{ \theta}
		\nu_{\theta}^2}{2\mu_{\theta} - \gamma_{ \theta} C_{\theta}^2}}.
\end{align*}
Therefore,
\begin{align*}
 \limsup_{k \to \infty}\left|\mathbb{E}[ f(x^k;\theta^{k}) - f(x^*;\theta^{*})  ]\right| &\leq  \limsup_{k \to \infty}\left|\mathbb{E}[ f(x^k;\theta^{k}) - f(x^k;\theta^{*})  ]\right| +
 \limsup_{k \to \infty} \left|\mathbb{E}[f(x^{k};\theta^{*}) - f(x^{*};\theta^{*})  ]\right|\\
 & \leq\frac{1}{2} \gamma_{x} M^2 + \frac{1}{2}\gamma_{x}^{1-\tau}  M_x^2 +\frac{1}{2} \gamma_{x}^{\tau-1} L_{\theta}^2 \frac{ \gamma_{ \theta}
		\nu_{\theta}^2}{2\mu_{\theta} - \gamma_{ \theta} C_{\theta}^2} +D_{\theta}\sqrt{\frac{ \gamma_{ \theta}
		\nu_{\theta}^2}{2\mu_{\theta} - \gamma_{ \theta} C_{\theta}^2}}.
 \end{align*}}
\end{enumerate}
\end{proof}

\subsection{Regret analysis} \label{sec:II.IV}
In this subsection, we consider the problem of online convex
programming in a misspecified regime. In online convex programming
problems, a decision-maker sees an infinite sequence of functions $c_1, c_2, \hdots$ where
each function is convex in its argument over a closed and convex set
$X$. An
{\em online} convex programming algorithm~\cite{zinkevich03} generates
an iterate $x_k$ at each time epoch $k$ and a metric of performance  is the
{\em regret} associated with not using an offline algorithm that
considers the following problem:
$ \min_{x \in X} \quad \sum_{k=1}^K c_k(x). $
If an online convex algorithm generates iterates $x_1, x_2, \hdots, $
then the regret $R_K$ is defined as
$$ R_K \triangleq \left[\sum_{k=1}^K c_k(x_k) - \min_{x \in X} \,
\sum_{k=1}^K c_k(x)\right].$$
A desirable feature of an online convex programming algorithm is that it
is characterized by sublinear regret~\cite{zinkevich03}. \\
	
Often the model prescribed in an online optimization regime can be
refined to a setting where the functions are related across time rather
than being a sequence of unrelated functions. We consider one particular
regime in which the decision-maker sees a sequence of
functions given by $f(\bullet;\theta_1)$, $f(\bullet;\theta_2), \hdots.$
Furthermore, neither the values $\theta_1, \theta_2, \hdots$ are known
to the decision-maker nor is the fact that $\theta_k \to \theta^*$ as $k \to
\infty$. As earlier, we assume that the decision-maker has to furnish $x_1,
x_2, \hdots$ and we define the {\em misspecified regret} after $K$ steps associated with our generated sequence $\{x^k,\theta^k\}$ as follows:
\begin{align*}
  R_K \triangleq \mathbb{E}\left[\sum_{k=1}^K f(x^k;\theta^k,\xi) -  K
  f(x^*;\theta^*,\xi)\right].
\end{align*}
Unlike the traditional definition, we consider the departure from
$f(x^*,\theta^*)$ and should be contrasted with the standard regret metric given by
$R^{\rm std}_K \triangleq \mathbb{E}\left[\sum_{k=1}^K f(x^k;\theta^*,\xi) -  K
  f(x^*;\theta^*,\xi)\right]$. For
  purposes of deriving analytical bounds, we define the following variant of
  regret as follows:
\begin{align*}
  \widehat{R}_K \triangleq \mathbb{E}\left[\sum_{k=1}^K f(x^k;\theta^k,\xi)
  -  \sum_{k=1}^K f( {y_K^*};\theta^k,\xi)\right],
\mbox{ where } y_K^* \triangleq \displaystyle{\argmin_{y\in X} \mathbb{E}\left[\sum_{k=1}^K
	f( y;\theta^k,\xi)\right]}.
\end{align*}
Next, we provide  a rate
of decay of the upper bound of average regret. \vspace{0.2in} \\
\begin{thm} [{\bf Regret under convexity of $f$}] \label{thm:regret}
Suppose (A\ref{assump:filtration}) and (A\ref{assump:stoch_monotone_optim_error}) hold.
Suppose $\mathbb{E}[\|x-x^{*}\|^2]\leq M_x^2$, $\mathbb{E}[\|\nabla_x f(x;\theta)+ w^k \|^2]\leq M^2$  and $\mathbb{E}[\|\nabla_{\theta} g(\theta)+ v^k \|^2]\leq M_{\theta}^2$  for all $x\in X$ and $\theta\in\Theta$.
Suppose $\mathbb{E}[\|\nabla_x f(y^*_K;\theta^{k})+ u^k \|^2]\leq M^2$, where
$u^k \triangleq   \mathbb{E}[\nabla_x f(y^*_K;\theta^k,\xi) ]-\nabla_x
f( y^*_K;\theta^k)$.
 Let $\{x^k,\theta^k\}$ be computed via Algorithm \ref{alg:stoch_strongly_optim}.
Suppose
 $\gamma_{k,x}=k^{-\alpha}$ with $0.5\leq\alpha<1$,
and $\gamma_{\theta,k}=\lambda_{\theta}/k$ with
$\lambda_{\theta}>1/(2\mu_{\theta})$. If $0<\beta <1$, then the following
holds:
\begin{align*}
\frac{{ R_K }}{K} & \leq
 \frac{M_x^2 K^{\alpha-1}}{2} +
  \frac{M^2(K^{1-\alpha}- \alpha )}{2(1-\alpha)K} +   \frac{D_{\theta} \sqrt{Q_{\theta}(\lambda_{\theta})}(2\sqrt{K}-1) }{K} { + \frac{  M_x^2}{2 K^{\beta}}
+\frac{L_{\theta}^2Q_{\theta}(\lambda_{\theta})(\ln(K) +1) }{2K^{1-\beta}},}
\end{align*}
where $\beta>0$. Furthermore,
\begin{align*}
{\limsup_{K \to \infty}\frac{R(K) }{K}\leq 0.}
\end{align*}

\end{thm}\vspace{0.1in}
\begin{proof}
By using the proof in Theorem 1 in \cite{zinkevich03} (cf. Theorem \ref{thm:zinkevich_regret} in Appendix \ref{appendix:zinkevich}), we obtain that $\widehat{R}_K/K$ is bounded as
		follows:
\begin{align*}
  \widehat{R}_K \leq \frac{M_x^2}{2 \gamma_{K,x}} + \frac{M^2}{2} \sum_{k=1}^K \gamma_{k,x}.
\end{align*}
Next, if $\gamma_{k,x}=k^{-\alpha}$ with $0.5\leq\alpha<1$, then we have
the following bound on $\sum_{k=1}^K \gamma_{k,x}$:
\begin{align*}
  \sum_{k=1}^K \gamma_{k,x} = \sum_{k=1}^K k^{-\alpha} \leq 1+\int_{1}^K
  x^{-\alpha} dx = \frac{1}{1-\alpha} (K^{1-\alpha}-{\alpha}).
\end{align*}
Therefore, we obtain the following bound on $\widehat R_K$:
\begin{align} \label{eq:R_1}
  \widehat{R}_K \leq \frac{M_x^2 K^{\alpha}}{2} +
  \frac{M^2(K^{1-\alpha}-{\alpha})}{2(1-\alpha)}.
\end{align}
Recall that the difference between the real {regret} and {misspecified
	regret} is given by the following:{
\begin{align*}
\begin{aligned}
  \left|R_K-\widehat{R}_K\right| &= \left|\mathbb{E}\left[\sum_{k=1}^K f(
		  {y^*_{K}};\theta^k,\xi) - K f(x^*;\theta^*,\xi)\right]\right|\\
  & \leq   \left|\mathbb{E}\left[\sum_{k=1}^K f( {y^*_{K}};\theta^k,\xi) - K
  f(y^*_{K};\theta^*,\xi)\right]\right| + \left|\mathbb{E}\left[K\left( f(
			  {y^*_{K}};\theta^*,\xi) -   f(x^*;\theta^*,\xi)\right)\right]\right|,
\end{aligned}
\end{align*}
or
\begin{align}
\begin{aligned} \label{eq:regret_difference}
  \frac{\left|R_K-\widehat{R}_K\right|}{K}   & \leq
  \underbrace{\left|\mathbb{E}\left[\frac{1}{K}\sum_{k=1}^K f(
		  {y^*_{{K}}};\theta^k,\xi)
	  -   f(y^*_{{K}};\theta^*,\xi)\right]\right|}_{\textbf{Term 1}} +
	  \underbrace{\left|\mathbb{E}\left[ f( y^*_{{K}};\theta^*,\xi) -
			  f(x^*;\theta^*,\xi)\right]\right|}_{\textbf{Term 2}}.
\end{aligned}
\end{align}}
We proceed to derive bounds for {\bf Terms} 1 and 2.
{\bf Term 1} in \eqref{eq:regret_difference} may be bounded as follows:
\begin{align*}
\begin{aligned}
 \left|\mathbb{E}\left[\frac{1}{K}\sum_{k=1}^K f( {y^*_{K}};\theta^k,\xi) -
 f(y^*_{K};\theta^*,\xi)\right]\right| &\leq
 \frac{1}{K}\sum_{k=1}^K \mathbb{E}\left[\left|f( y^*_{{K}};\theta^k,\xi) -
 f(y_{{K}}^*;\theta^*,\xi)\right|\right] \\
 & \leq  \frac{D_{\theta}}{K}\sum_{k=1}^K \mathbb{E} [\|\theta^k  -    \theta^*\|] \\
 & \leq  \frac{D_{\theta}}{K}\sum_{k=1}^K
 \sqrt{\frac{Q_{\theta}(\lambda_{\theta})}{k}}.
\end{aligned}
\end{align*}
where the second and third inequalities follow from the Lipschitz continuity of
$\nabla f(y^*;\theta)$ in $\theta$
(A\ref{assump:stoch_monotone_optim_error}) and \eqref{eq:error_theta_Q}. Through some analysis, the
right hand side may be further bounded as follows:
	\begin{align}\label{eq:R2_1}
\begin{aligned}
\frac{D_{\theta}}{K}\sum_{k=1}^K
 \sqrt{\frac{Q_{\theta}(\lambda_{\theta})}{k}}
& \leq \frac{D_{\theta}{\sqrt{Q_{\theta}(\lambda_{\theta})}}}{K}\left({1+}\int_{1}^K
 {\frac{1}{\sqrt{x}}}dx\right)
 \leq
\frac{D_{\theta}{\sqrt{Q_{\theta}(\lambda_{\theta})}({2\sqrt{K}-1})}}{K}.
\end{aligned}
\end{align}
This implies that ${\bf Term~ 1}$ in
\eqref{eq:regret_difference} converges to zero as $K\to \infty$.  Next,
	we consider ${\bf Term~ 2}$ in \eqref{eq:regret_difference}.
{By the optimality condition for $y_K^*$,
	we have the following expression:
\begin{align}
\notag 0 &\geq
 \sum_{k=1}^K\mathbb{E}[(y_K^* -x^*)^T \nabla_x f( y_K^*;\theta^k,\xi)] \\
 & = \sum_{k=1}^K \mathbb{E}[(y_K^* -x^*)^T \nabla_x f(
		 y_K^*;\theta^*,\xi)]  + \sum_{k=1}^K \mathbb{E}[(y_K^*
			 -x^*)^T (\nabla_x f(y_K^*;\theta^k,\xi)-\nabla_xf( y_K^*;\theta^*,\xi))].
\label{eq:regret_term2_new}
\end{align}
Since $f(x;\theta)$ is convex in $x$ for every $\theta \in
\Theta$, we may leverage the gradient inequality.
\begin{align} \label{eq:regret_term3}
\notag \mathbb{E}[f(x^*;\theta^*,\xi)] & \geq
\mathbb{E}[f(y^*_K;\theta^*,\xi)] + \mathbb{E}[\nabla_x
f(y^*_K;\theta^*,\xi)^T(x^*-y^*_K)] \\
		\implies \mathbb{E}[(y_K^*-x^{*})^T  \nabla_x f(y_K^*;\theta^{*},\xi)]
      & \geq \mathbb{E}[f(y_K^*;\theta^{*},\xi) - f(x^{*};\theta^{*},\xi) ].
\end{align}
Combining \eqref{eq:regret_term2_new} and \eqref{eq:regret_term3}, we get
the following lower bound:
\begin{align*}
0 &\geq
 \sum_{k=1}^K \mathbb{E}[f(y_K^*;\theta^{*},\xi) - f(x^{*};\theta^{*},\xi) ] + \sum_{k=1}^K \mathbb{E}[(y_K^* -x^*)^T (\nabla_x
		 f(y_K^*;\theta^k,\xi)- \nabla_x f( y^*;\theta^*,\xi))].
\end{align*}
This allows for constructing the following bound on
$\sum_{k=1}^K \mathbb{E}[f(y_K^*;\theta^{*},\xi) - f(x^{*};\theta^{*},\xi)]$:
\begin{align}
\label{eq:regret_final}
  \sum_{k=1}^K \mathbb{E}[f(y_K^*;\theta^{*},\xi) - f(x^{*};\theta^{*},\xi) ]
 &\leq    -
 \sum_{k=1}^K \mathbb{E}[(y_K^* -x^*)^T (\nabla_x
		 f(y_K^*;\theta^k,\xi)- \nabla_x f( y_K^*;\theta^*,\xi))] \notag \\
 &\leq
\frac{1}{2}\sum_{k=1}^K\delta_K \mathbb{E}[\|y_K^{*}-x^{*}\|^2] +\frac{1}{2}\sum_{k=1}^K \frac{1}{\delta_K}
\mathbb{E}[\|\nabla_x f(y_K^{*};\theta^{k},\xi) - \nabla_x
f(y_K^{*};\theta^{*},\xi) \|^2 ] \notag \\
     &\leq
	 \frac{1}{2}\sum_{k=1}^K\delta_K M_x^2
	 +\frac{1}{2}\sum_{k=1}^K \frac{1}{\delta_K} L_{\theta}^2\mathbb{E}[\|\theta^{k} -
	 \theta^{*} \|^2] \notag \\
    &\leq  \frac{1}{2} \sum_{k=1}^K\delta_K  M_x^2 +\frac{1}{2}\sum_{k=1}^K \frac{1}{\delta_K} \frac{L_{\theta}^2Q_{\theta}(\lambda_{\theta})}{k},
\end{align}
where $\delta_{K}=K^{-\beta}$ with $0<\beta<1$ and the last inequality follows from  \eqref{eq:error_theta_Q}.
{Note that
$
   \sum_{k=1}^K \frac{1}{k}   \leq   \ln(K) +1  .
$
\begin{align}
\begin{aligned} \label{eq:R2_2}
 \mbox{Thus, } \left|\mathbb{E}\left[ f( y^*_{{K}};\theta^*,\xi) -
			  f(x^*;\theta^*,\xi)\right]\right|  &= \mathbb{E}[f(y_K^*;\theta^{*}) - f(x^{*};\theta^{*}) ] \\& \leq \frac{  M_x^2}{2 K^{\beta}}  +\frac{\sum_{k=1}^K \frac{L_{\theta}^2Q_{\theta}(\lambda_{\theta})}{k}}{2 K\delta_{K}}\\
&   \leq \frac{  M_x^2}{2 K^{\beta}}
+\frac{L_{\theta}^2Q_{\theta}(\lambda_{\theta})(\ln(K) +1) }{2K^{1-\beta}}.
\end{aligned}
\end{align}
Combining \eqref{eq:R_1}, \eqref{eq:regret_difference}, \eqref{eq:R2_1}, and \eqref{eq:R2_2}, we have that ${R_K/K}$ can be bounded as follows:
\begin{align*}
{\frac{R_K}{K} } &\leq  {\frac{ \widehat{R}_K }{K} } + \frac{R_K-\widehat{R}_K}{K}
 \leq  {\frac{ \widehat{R}_K }{K} } + \frac{\left|R_K-\widehat{R}_K\right|}{K} \\
& \leq
  \frac{M_x^2 K^{\alpha-1}}{2} +
  \frac{M^2(K^{1-\alpha}- \alpha )}{2(1-\alpha)K} +  \frac{D_{\theta}{\sqrt{Q_{\theta}(\lambda_{\theta})}( 2\sqrt{K} -1)}}{K} + \frac{  M_x^2}{2 K^{\beta}}
+\frac{L_{\theta}^2Q_{\theta}(\lambda_{\theta})(\ln(K) +1) }{2K^{1-\beta}}.
\end{align*}}
Furthermore, this implies that the limit superior of the average regret is nonpositive.}
\end{proof}

{\bf Remark:} In effect, in the context of learning and optimization, the averaging
approach leads to a complexity bound given loosely by
$$\Oscr\left(\frac{a}{K^{1-\alpha}}+\frac{b}{K^{\alpha}}+\frac{d}{K^{\beta}}+\underbrace{\frac{c_{\theta}}{\sqrt{K}}{+\frac{e_{\theta}
		\ln(K) }{K^{1-\beta}}}}_{\scriptsize \mbox{contribution from
		learning}}\right),$$
where $a, b, c_{\theta}, d, e_{\theta}$ are suitably defined.  If $\theta^*$
is available,  then $c_{\theta}, e_{\theta} =
0$. Furthermore, by setting $\alpha = 0.5$ and $\beta = 0.5$, this leads to the
bound of {$\Oscr(\ln K/\sqrt{K})$}, which is  a degradation as the result of learning $\theta^*$.
\qed

\section{Stochastic variational inequality problems with imperfect information}
\label{sec:III}
Several shortcomings exist in the optimization based formulation
represented by {$({\cal P}^o_x(\theta^*))$}. First, the misspecification arises
entirely in the objectives while the constraints are known with
certainty. Second, the underlying problem need not be an optimization
problem, but could instead be captured by a variational inequality
problem. Such problems~\cite{Pang03I} can capture a range of
problems including economic equilibrium problems, traffic equilibrium
problems, and convex Nash games. In fact, variational inequality
problems can effectively capture optimization problems with misspecified
constraints.
This motivates the consideration of the
misspecified stochastic variational inequality problem
{$({\cal P}^v_x(\theta^*))$} where $\theta^*$ can be learnt through the solution
of the following problem:
\begin{align} \label{problem_stoch_optim_theta_2}
\tag{${\cal L}^v_\theta$}
    (\vartheta - \theta)^T\mathbb{E}[G(\theta;\eta)] \geq 0, \qquad \forall \vartheta \in
	\Theta,
\end{align}
where $G: \theta \times \Real^p \to \Real^m$, and $\Theta$ and $\eta$
abide by the previous specifications. In the majority of problem
settings, $G(\theta;\theta) \triangleq \nabla_\theta g(\theta;\eta)$ but
we employ the variational structure to introduce generality.  In this section, we extend the results of the
previous section to this regime.  Specifically, we develop the convergence theory under settings where
the variational map $F$ is
both strongly monotone and merely monotone in $x$ for
every $\theta \in \Theta$ in Section~\ref{sec:III.II}  and provide rate
statements in Section~\ref{sec:III.III}.

\subsection{Almost-sure convergence} \label{sec:III.II}
As in Section \ref{sec:II}, we propose a set of coupled
stochastic approximation schemes for computing $x^*$ and $\theta^*$.
Given $x^0 \in X$ and $\theta^0 \in \Theta$, the coupled SA schemes are
stated next:\vspace{0.2in}
\begin{alg} [{\bf Coupled SA schemes for stochastic variational inequality
	problems}]   \label{alg:VI_stoch_grad_grad_strongly}

\noindent {\bf Step 0.} Given $x_0 \in X, \theta_0 \in \Theta$ and sequences $\{\gamma_{k,x},
	  \gamma_{k,\theta}\}$, $k := 0$

\noindent {\bf Step 1.}
\begin{align}
\tag{Comp$_k$}       x^{k+1} & := \Pi_{X}  \left(x^{k}-\gamma_{k,x}(F(x^k;\theta^k) + w^k)  \right) \\
   \tag{Learn$_k$} \theta^{k+1} &: = \Pi_{\Theta}  \left(  \theta^{k} -
			\gamma_{k,\theta} ( G(\theta^{k}) + v^k )
			\right),
\end{align}
where $w^k \triangleq F(x^k;\theta^k,\xi^k) - F(x^k;\theta^k)$ and $v^k \triangleq G(\theta^{k};\eta^k) - G(\theta^{k})$.

\noindent {\bf Step 2.} If $k > K$, stop; else $k:k+1$, go to Step. 1.
\end{alg}\vspace{0.2in}

We begin by stating an assumption similar to (A\ref{assump:stoch_strongly_optim}-1) on the
mappings $F$ and $G$.
\begin{assumption}[A1-3] \label{assump:VI_stoch_imperfect_strongly}
Suppose the following hold:
\begin{enumerate}
  \item[(i)] For every {$\theta \in \Theta$}, $F(x;\theta)$ is both strongly monotone and Lipschitz continuous in $x$ with constants $\mu_{x}$ and $L_{x}$, respectively.
  \item[(ii)] For every {$x \in X$}, $F(x;\theta)$ is Lipschitz continuous in $\theta$ with
  constant $L_{\theta}$.
   \item[(iii)] $G(\theta)$ is strongly monotone and Lipschitz continuous in $\theta$ with
    constants $\mu_{\theta}$ and $C_{\theta}$, respectively.
\end{enumerate}
\end{assumption}

Now, we can leverage the results in Section \ref{sec:II.II} to examine
the convergence properties for Algorithm
\ref{alg:VI_stoch_grad_grad_strongly}.\vspace{0.2in}
\begin{prop}[{\bf Almost-sure convergence under strong monotonicity of $F$}] \label{thm:VI_stoch_grad_grad_strongly}
Suppose (A1-3), (A\ref{assump:stoch_steplength_strongly_optim}-1) and (A\ref{assump:filtration}) hold.  Let $\{x^k,\theta^k\}$ be computed via Algorithm \ref{alg:VI_stoch_grad_grad_strongly}.
Then, $x^{k}\to x^{*}$ $a.s.$ and $\theta^{k} \to \theta^{*}$ $a.s.$ as
$k\rightarrow\infty$,  where $x^*$ is the unique solution to \eqref{problem_stoch_optim_1}
and $\theta^*$ is the unique solution to
\eqref{problem_stoch_optim_theta_2}.
\end{prop}\vspace{0.2in}

\begin{proof}
Note that $x^* = \Pi_{X} ( x^{*} - \gamma_{k,x}
		F(x^*;\theta^*))$ and
$\theta^{*} = \Pi_{\Theta} ( \theta^{*} - \gamma_{k,\theta}
		   G(\theta^{*}))$.
If we replace
$\nabla_x f$ and $\nabla_{\theta} g$ by $F$ and $G$ in Proposition \ref{thm:stoch_strongly_optim}, respectively,
then by the proof of Proposition \ref{thm:stoch_strongly_optim}, we get
$x^{k}\to x^{*}$ $a.s.$ and $\theta^{k} \to \theta^{*}$ $a.s.$ as $k\rightarrow\infty$.
\end{proof}

Next, we weaken the rather stringent requirement of strong monotonicity
of the map by using an iterative Tikhonov regularization, which can be stated as
follows.\vspace{0.2in}
\begin{alg} [{\bf Coupled regularized SA schemes for stochastic variational inequality
	problems}]   \label{alg:VI_stoch_grad_grad_monotone}

\noindent {\bf Step 0.} Given $x_0 \in X, \theta_0 \in \Theta$ and sequences $\{\gamma_{k,x},
	  \gamma_{k,\theta}\}$, $k := 0$

\noindent {\bf Step 1.}
\begin{align}
\tag{Comp$_k$}       x^{k+1} & := \Pi_{X}
\left(x^{k}-\gamma_{k,x}(F(x^k;\theta^k)+\epsilon_k x^k + w^k)  \right) \\
   \tag{Learn$_k$} \theta^{k+1} &: = \Pi_{\Theta}  \left(  \theta^{k} -
			\gamma_{k,\theta} ( G(\theta^{k}) + v^k )
			\right),
\end{align}
where $w^k \triangleq F(x^k;\theta^k,\xi^k) - F(x^k;\theta^k)$ and $v^k \triangleq G(\theta^{k};\eta^k) - G(\theta^{k})$.

\noindent {\bf Step 2.} If $k > K$, stop; else $k:=k+1$, go to Step. 1.

\end{alg}\vspace{0.2in}
Unlike in standard Tikhonov regularization, such a scheme updates the
regularization parameter $\epsilon_k$ after every step. Tikhonov
regularization and its iterative counterpart has a long
history~\cite{Polyak87} while iterative regularization schemes have seen
relatively less study in the context of variational inequality problems
(cf.~\cite{konnov07equilibrium,kannan10online}). Of note is the
extension to distributed schemes to accommodate  monotone Cartesian
stochastic variational inequality problems~\cite{koshal13regularized}.
We employ such techniques in developing single-loop stochastic
approximation schemes in the context of learning and optimization.  The
following assumptions will be made on both the decision variable and
parameter.
\begin{assumption}[A1-4] \label{assump:VI_stoch_imperfect_monotone}
Suppose the following holds in addition to (A1-3 (ii)) and (A1-3 (iii)).
\begin{enumerate}
  \item[(i)] For every $\theta \in \Theta$, $F(x;\theta)$ is monotone in $x$ and Lipschitz continuous in $x$ with constant $L_{x}$.
\end{enumerate}
\end{assumption}
In iterative Tikhonov regularization, one cannot independently choose
$\{\epsilon_k\}$ and $\{\gamma_k\}$; in fact, these sequences are
related and satisfy some collectively imposed requirements.
\begin{assumption}[A2-3] \label{assump:VI_stoch_steplength_monotone}
Let $\{\gamma_{k,x}\}$, $\{\gamma_{k,\theta}\}$, $\{\epsilon_{k}\}$ and some constant $\tau\in(0,1)$ be chosen such that:
\begin{enumerate}
  \item[(i)] $\sum_{k=0}^{\infty}\gamma_{k,x}^{2-\tau} < \infty$  and $\sum_{k=0}^{\infty} \gamma_{k,\theta}^2 < \infty$,
  \item[(ii)] $\sum_{k=0}^{\infty}\gamma_{k,x}\epsilon_k = \infty$ and $\sum_{k=0}^{\infty} \gamma_{k,\theta} =\infty$,
  \item[(iii)] $\beta_k = \frac{\gamma_{k,x}^{\tau}}{2\gamma_{k,\theta} \mu_{\theta} } \downarrow 0$ as $k\to 0$.
    \item[(iv)] $\sum_{k=0}^{\infty}\frac{(\epsilon_{k-1}-\epsilon_k)}{\epsilon_k} <\infty$.
\end{enumerate}
\end{assumption}

Before providing a convergence result for Algorithm
\ref{alg:VI_stoch_grad_grad_monotone}, we introduce the following results.

\begin{lem} \label{lem:VI_monotone_Lipschitz}
Let $H:K\to\mathbb{R}^n$ be a mapping that is monotone over $K$, and Lipschitz
		continuous over $K$ with constant $L$. Then, for any $\gamma>0$ and $\epsilon>0$, we have
$ \| (x-y)-\gamma(H(x)-H(y))-\epsilon\gamma(x-y)\|    \leq q \|x-y\|,$
where $q=\sqrt{1-2\gamma\epsilon + \gamma^2 (L^2+\epsilon^2)}$.
\end{lem}
\begin{proof}
See proof of Theorem 1 in \cite{kannan10distributed}.
\end{proof}

\begin{lem} \label{lem:VI_monotone_Tik}
Let $H:K\to\mathbb{R}^n$ be a mapping that is monotone over $K$.
Given $\epsilon_k>0$, let $y^k$ be a solution to VI$(K, H+\epsilon_k {\bf
		I})$.  Then,
\begin{align*}
\|y^{k}-y^{k-1}\| \leq \frac{M(\epsilon_{k-1}-\epsilon_k)}{\epsilon_k},
\end{align*}
where $M=\|x^*\|$ and $x^*$ is a solution to VI($H$,$K$).
\end{lem}
\begin{proof}
See Lemma 3 in \cite{kannan10distributed}.
\end{proof}

The convergence result for Algorithm
\ref{alg:VI_stoch_grad_grad_monotone} can be stated as
follows.\vspace{0.2in}

\begin{thm}[{\bf Almost-sure convergence under monotonicity of $F$}] \label{thm:VI_stoch_grad_grad_monotone}
Suppose (A1-4) ,  (A2-3) and (A\ref{assump:filtration}) hold.
Suppose $X$ is bounded and the solution set $X^*$ of
\eqref{problem_stoch_optim_1} is nonempty.
Let $\{x^k,\theta^k\}$ be computed via Algorithm \ref{alg:VI_stoch_grad_grad_monotone}.
Then, $\theta^{k} \to \theta^{*}$ $a.s.$ as $k\rightarrow\infty$, and $x^{k}$ converges to a random point in $X^*$ $a.s.$ as $k\rightarrow\infty$.
\end{thm}\vspace{0.2in}

\begin{proof}
We have for any $x^*\in X^*$ that $x^* = \Pi_{X} ( x^{*} - \gamma_{k,x}
		F(x^*;\theta^*)).$
Suppose $y^k$ is a solution to the following fixed-point problem
\begin{align*}
    y^k & = \Pi_{X}(y^{k}-\gamma_{k,x}(F(y^{k};\theta^{*})+\epsilon_k y^{k})).
\end{align*}
Then, by the triangle inequality $\|x^{k+1}-x^*\|$ may be
bounded as follows:
\begin{align*}
    \|x^{k+1}-x^*\| & \leq \underbrace{\|x^{k+1}-y^k\|}_{\textbf{\textrm{Term 1}}} + \underbrace{\|y^k-x^*\|}_{\textbf{\textrm{Term 2}}}.
\end{align*}
Term 2 converges to zero by the convergence statement of Tikhonov regularization
methods~{\cite{Pang03I}}. By using the non-expansivity of the Euclidean
projector, $\|x^{k+1}-y^k\|^2$ can be bounded as follows:
\begin{align*}
 \|x^{k+1}-y^k\|^2 & =\|\Pi_{X}(x^{k}-\gamma_{k,x}(F(x^{k};\theta^{k})+\epsilon_k x^{k} + w^k))  -\Pi_{X}(y^{k}-\gamma_{k,x}(F(y^{k};\theta^{*})+\epsilon_k y^{k}))\|^2\\
\notag  & \leq\|(x^{k}-y^k)-\gamma_{k,x}(F(x^{k};\theta^{k})-F(y^{k};\theta^{*}))  - \epsilon_k \gamma_{k,x} (x^{k}-y^{k}) - \gamma_{k,x}w^k \|^2.
\end{align*}
By adding and subtracting
$\gamma_{k,x} F(x^{k};\theta^{*})$, this expression can be further expanded
as follows:
 \begin{align*}
 & \quad \ \|(x^{k}-y^k)-\gamma_{k,x}(F(x^{k};\theta^{*})-F(y^k;\theta^{*}))  - \gamma_{k,x} (F(x^{k};\theta^{k})-F(x^{k};\theta^{*})) - \epsilon_k \gamma_{k,x} (x^{k}-y^{k}) - \gamma_{k,x}w^k \|^2 \\
\notag &= \|(x^{k}-y^k)-\gamma_{k,x}(F(x^{k};\theta^{*})-F(y^k;\theta^{*}))  - \epsilon_k \gamma_{k,x} (x^{k}-y^{k}) \|^2  +\gamma_{k,x}^2 \|F(x^{k};\theta^{k})-F(x^{k};\theta^{*})\|^2+ \gamma_{k,x}^2 \|w^k\|^2 \\
\notag &  -2 [(x^{k}-y^k)-\gamma_{k,x}(F(x^{k};\theta^{*})-F(y^k;\theta^{*}))   - \epsilon_k \gamma_{k,x} (x^{k}-y^{k})]^T \times(F(x^{k};\theta^{k})-F(x^{k};\theta^{*})) \\
\notag &  -2 [(x^{k}-y^k)-\gamma_{k,x}(F(x^{k};\theta^{*})-F(y^k;\theta^{*}))  - \epsilon_k \gamma_{k,x} (x^{k}-y^{k})]^T w^k +2 \gamma_{k,x}^2 (F(x^{k};\theta^{k})-F(x^{k};\theta^{*}))^T w^k.
\end{align*}
Noting that $\mathbb{E}[w^k \mid \mathcal{F}_k] = 0$, we have
 \begin{align}  \label{eq:VI_grad_grad_monotone_rand_exp}
  \mathbb{E}[\|x^{k+1}-y^{k}\|^2 \mid \mathcal{F}_k]  \leq  \textbf{Term 3}+\textbf{Term 4} + \textbf{Term 5} + \gamma_{k,x}^2 \mathbb{E}[\| w^k \|^2\mid\mathcal{F}_k],
\end{align}
where
\begin{align*}
  \textbf{Term 3}& \triangleq  \|(x^{k}-y^{k})-\gamma_{k,x}(F(x^{k};\theta^{*})-F(y^{k};\theta^{*}))  - \epsilon_k \gamma_{k,x} (x^{k}-y^{k})\|^2 , \\
\textbf{Term 4} &\triangleq   \gamma_{k,x}^2 \|F(x^{k};\theta^{k})-F(x^{k};\theta^{*})\|^2 , \\
\textbf{Term 5}& \triangleq  - 2 \gamma_{k,x} [ (x^{k}-y^{k})-\gamma_{k,x}(F(x^{k};\theta^{*})-F(y^{k};\theta^{*}))  - \epsilon_k \gamma_{k,x} (x^{k}-y^{k}) ]^T(F(x^{k};\theta^{k})-F(x^{k};\theta^{*})) .
\end{align*}
By Lemma \ref{lem:VI_monotone_Lipschitz} and (A1-4), Term 3 can be further bounded by
\begin{align} \label{eq:VI_grad_grad_monotone_rand_exp_term1}
 (1-2\gamma_{k,x} \epsilon_k + \gamma_{k,x}^2 (L_x^2+(\epsilon_k)^2)) \|x^{k}-y^k\|^2.
\end{align}
By the Lipschitz continuity of $F(x;\theta)$ in $\theta$ (A1-4), Term 4 can be further bounded by
\begin{align} \label{eq:VI_grad_grad_monotone_rand_exp_term2}
 \gamma_{k,x}^2 L_{\theta}^2 \|\theta^{k}-\theta^{*}\|^2.
\end{align}
By the Cauchy-Schwarz inequality, Lemma \ref{lem:VI_monotone_Lipschitz}, (A1-4) as well as the fact that $2ab\leq a^2+b^2$, Term 5 can be further bounded by
 \begin{align}
 \begin{aligned} \label{eq:VI_grad_grad_monotone_rand_exp_term3}
     & \quad\ 2\gamma_{k,x}  \| (x^{k}-y^{k})-\gamma_{k,x}(F(x^{k};\theta^{*})-F(y^k;\theta^{*}))   - \epsilon_k \gamma_{k,x} (x^{k}-y^{k})\|\| F(x^{k};\theta^{k})-F(x^{k};\theta^{*}) \| \\
   &  \leq  2\gamma_{k,x}\sqrt{1-2\gamma_{k,x} \epsilon_k + \gamma_{k,x}^2 (L_x^2+(\epsilon_k)^2)} \|x^{k}-y^{k}\|    L_{\theta} \|\theta^{k}-\theta^{*}\|  \\
   & \leq 2 \gamma_{k,x} L_{\theta} \|x^{k}-y^k\| \|\theta^{k}-\theta^{*}\| \\
 & \leq \gamma_{k,x}^{2-\tau} L^2_{\theta} \|x^{k}-y^k\|^2 + \gamma_{k,x}^{\tau} \|\theta^{k}-\theta^{*}\|^2,
\end{aligned}
\end{align}
where $\tau\in(0,1)$ is chosen to satisfy (A2-3).
Combining \eqref{eq:VI_grad_grad_monotone_rand_exp},  \eqref{eq:VI_grad_grad_monotone_rand_exp_term1},  \eqref{eq:VI_grad_grad_monotone_rand_exp_term2} and \eqref{eq:VI_grad_grad_monotone_rand_exp_term3},
we get
\begin{align}
\begin{aligned} \label{eq:VI_grad_grad_monotone_rand_exp_final}
 \mathbb{E}[\|x^{k+1}-y^k\|^2\mid\mathcal{F}_k]
 &\leq (q_{k,x}^2+\gamma_{k,x}^{2-\tau} L^2_{\theta}) \|x^{k}-y^k\|^2  + (\gamma_{k,x}^\tau  + \gamma_{k,x}^2 L_{\theta}^2) \|\theta^{k}-\theta^{*}\|^2 + \gamma_{k,x}^2  \nu_x^2,
 \end{aligned}
\end{align}
where $q_{k,x} = \sqrt{1-2\gamma_{k,x} \epsilon_k + \gamma_{k,x}^2 (L_x^2+(\epsilon_k)^2)}$.

On the other hand, we have that $\theta^{*}$ is the unique solution to
VI$(\Theta,\mathbb{E}[G(\bullet;\eta)])$ and
\begin{align*}
   \theta^{*} = \Pi_{\Theta} ( \theta^{*} - \gamma_{k,\theta}  G(\theta^{*})).
\end{align*}
Therefore,  by the nonexpansivity of the Euclidean projector, $\|\theta^{k+1}-\theta^*\|^2$ may be bounded as follows:
\begin{align*}
  & \quad\ \|\theta^{k+1}-\theta^*\|^2
\notag   =\|\Pi_{\Theta} ( \theta^{k} - \gamma_{k,\theta} (G(\theta^{k})+v^{k}))   -\Pi_{\Theta} ( \theta^* - \gamma_{k,\theta} G(\theta^*))\|^2\\
\notag  & \leq\|(\theta^{k}-\theta^{*})-\gamma_{k,\theta}(G(\theta^{k})-G(\theta^*)) -\gamma_{k,\theta} v^k \|^2\\
\notag  & = \|(\theta^{k}-\theta^{*})-\gamma_{k,\theta}(G(\theta^k)-G(\theta^*))\|^2 + \gamma_{k,\theta}^2 \|v^{k}\|^2  -2\gamma_{k,\theta} [ (\theta^{k}-\theta^{*})-\gamma_{k,\theta}(G(\theta^k)-G(\theta^*)) ]^T v^k.
\end{align*}
By taking conditional expectations and by recalling that $\mathbb{E}[v^{k} \mid \mathcal{F}_k] =  0$ (A\ref{assump:filtration}),
we obtain
the following:
\begin{align}
\begin{aligned} \label{eq:VI_grad_grad_monotone_rand_exp_final_theta}
      \mathbb{E}[\|\theta^{k+1}-\theta^*\|^2\mid\mathcal{F}_k] &  \leq \|(\theta^{k}-\theta^{*})-\gamma_{k,\theta}(G(\theta^k)-G(\theta^*))\|^2 + \gamma_{k,\theta}^2 \mathbb{E}[\|v^{k}\|^2\mid\mathcal{F}_k] \\
& \leq	q_{k,\theta}^2 \|\theta^{k}-\theta^*\|^2 + \gamma_{k,\theta}^2 \nu_{\theta}^2,
\end{aligned}
\end{align}
where $q_{k,\theta} = \sqrt{1-2\gamma_{k,\theta} \mu_{\theta}  + \gamma_{k,\theta}^2 C_{\theta}^2}$,
and the second inequality follows from Lemma \ref{lem:strongly_Lipschitz}, (A1-4) and (A\ref{assump:filtration}).
 Since by (A2-3)
   $\sum_{k=0}^{\infty} (1-q_{k,\theta}^2) = \infty$ and
   $\sum_{k=0}^{\infty} \gamma_{k,\theta}^2 \nu_{\theta}^2 < \infty$,
and
\begin{align*}
\lim_{k\to \infty} \frac{\gamma_{k,\theta}^2\nu_{\theta}^2}{1-q_{k,\theta}^2} & = \lim_{k\to \infty} \frac{\gamma_{k,\theta}^2\nu_{\theta}^2}{2\gamma_{k,\theta} \mu_{\theta}  - \gamma_{k,\theta}^2 C_{\theta}^2}  = \lim_{k\to \infty} \frac{\gamma_{k,\theta}\nu_{\theta}^2	}{2 \mu_{\theta}  - \gamma_{k,\theta} C_{\theta}^2}  = 0,
\end{align*}
we have by Lemma \ref{lem:supermartingale} that $\|\theta^{k}-\theta^{*}\|  \to 0   $ $a.s.$ $   \textrm{as } k\to\infty$.  Choose $\beta_k = \frac{\gamma_{k,x}^{\tau}}{2\gamma_{k,\theta} \mu_{\theta} }$ by (A2-3).
Note that by
assumption $\beta_{k+1}\leq \beta_k$.
By multiplying
the left hand side of \eqref{eq:VI_grad_grad_monotone_rand_exp_final_theta} by $\beta_{k+1}$ and adding to the left hand side  of $\eqref{eq:VI_grad_grad_monotone_rand_exp_final}$,
we get
\begin{align}
\begin{aligned} \label{eq:VI_grad_grad_monotone_rand_exp_theta}
    & \quad \ \mathbb{E}[\|x^{k+1}-y^k\|^2\mid\mathcal{F}_k] + \beta_{k+1} \mathbb{E}[\|\theta^{k+1}-\theta^{*}\|^2 \mid\mathcal{F}_k]  \\
    & \leq \mathbb{E}[\|x^{k+1}-y^k\|^2\mid\mathcal{F}_k] + \beta_{k} \mathbb{E}[\|\theta^{k+1}-\theta^{*}\|^2 \mid\mathcal{F}_k]  \\
 	& \leq (q_{k,x}^2+\gamma_{k,x}^{2-\tau} L^2_{\theta}) \|x^{k}-y^k\|^2   + (\beta_k q_{k,\theta}^2 + \gamma_{k,x}^{\tau} +\gamma_{k,x}^2 L_{\theta}^2 ) \|\theta^{k}-\theta^*\|^2  + \beta_k \gamma_{k,\theta}^2 \nu_{\theta}^2 + \gamma_{k,x}^2  \nu_x^2\\
& = (q_{k,x}^2+\gamma_{k,x}^{2-\tau} L^2_{\theta}) \|x^{k}-y^k\|^2   + \underbrace{\frac{\beta_k q_{k,\theta}^2 + \gamma_{k,x}^{\tau} +\gamma_{k,x}^2 L_{\theta}^2}{\beta_k}}_{\bf Term 6} \cdot \beta_k \|\theta^{k}-\theta^*\|^2   + \beta_k \gamma_{k,\theta}^2 \nu_{\theta}^2  + \gamma_{k,x}^2  \nu_x^2.
\end{aligned}
\end{align}
Term 6 on the right hand side of  \eqref{eq:VI_grad_grad_monotone_rand_exp_theta} can be further expanded as
\begin{align} \begin{aligned} \label{eq:VI_grad_grad_monotone_rand_exp_theta4}
\frac{\beta_k q_{k,\theta}^2 + \gamma_{k,x}^{\tau} +\gamma_{k,x}^2 L_{\theta}^2}{\beta_k}
 & = q_{k,\theta}^2 +  \frac{\gamma_{k,x}^{\tau} +\gamma_{k,x}^2 L_{\theta}^2}{\beta_k}   = 1-2\gamma_{k,\theta} \mu_{\theta}  + \gamma_{k,\theta}^2 C_{\theta}^2 + \frac{\gamma_{k,x}^{\tau}}{\beta_k} + \frac{\gamma_{k,x}^2 L_{\theta}^2}{\beta_k}  \\
  & = 1 + \gamma_{k,\theta}^2 C_{\theta}^2 + 2\gamma_{k,\theta} \gamma_{k,x}^{2-\tau} \mu_{\theta}  L_{\theta}^2.
\end{aligned}
\end{align}
Combining \eqref{eq:VI_grad_grad_monotone_rand_exp_theta} and \eqref{eq:VI_grad_grad_monotone_rand_exp_theta4}, we get
\begin{align*}
  & \quad\ \mathbb{E}[\|x^{k+1}-y^k\|^2\mid\mathcal{F}_k] + \beta_{k+1}
  \mathbb{E}[\|\theta^{k+1}-\theta^{*}\|^2 \mid\mathcal{F}_k]  \\
	& \leq (q_{k,x}^2+\gamma_{k,x}^{2-\tau}L^2_{\theta}) \|x^{k}-y^k\|^2 + (1 + \gamma_{k,\theta}^2 C_{\theta}^2   + 2\gamma_{k,\theta} \gamma_{k,x}^{2-\tau} \mu_{\theta}  L_{\theta}^2)   \beta_k \|\theta^{k}-\theta^*\|^2 + \beta_k \gamma_{k,\theta}^2 \nu_{\theta}^2  + \gamma_{k,x}^2  \nu_x^2 \\
\notag	& = (1 + \gamma_{k,\theta}^2 C_{\theta}^2 + 2\gamma_{k,\theta}
		\gamma_{k,x}^{2-\tau} \mu_{\theta}  L_{\theta}^2)   (
			\|x^{k}-y^k\|^2 + \beta_k \|\theta^{k}-\theta^*\|^2 )   - (\gamma_{k,\theta}^2 C_{\theta}^2 + 2\gamma_{k,\theta} \gamma_{k,x}^{2-\tau} \mu_{\theta}  L_{\theta}^2 + 2\gamma_{k,x} \epsilon_k) \|x^{k}-y^k\|^2 \\
&  + (\gamma_{k,x}^2 (L_x^2+(\epsilon_k)^2) +\gamma_{k,x}^{2-\tau}L^2_{\theta} ) \|x^{k}-y^k\|^2   + \beta_k \gamma_{k,\theta}^2 \nu_{\theta}^2  + \gamma_{k,x}^2  \nu_x^2.
\end{align*}
Note that
$\|x^{k+1}-y^k\|^2 \leq \|x^{k}-y^{k-1}\|^2 + 2 \|x^{k}-y^{k-1}\| \|y^{k}-y^{k-1}\|+ \|y^{k}-y^{k-1}\|^2$. We have
\begin{align*}
\notag  & \quad \ \mathbb{E}[\|x^{k+1}-y^k\|^2\mid\mathcal{F}_k] + \beta_{k+1} \mathbb{E}[\|\theta^{k+1}-\theta^{*}\|^2 \mid\mathcal{F}_k]  \\
 \notag	& \leq (1 + \gamma_{k,\theta}^2 C_{\theta}^2 + 2\gamma_{k,\theta} \gamma_{k,x}^{2-\tau} \mu_{\theta}  L_{\theta}^2)   ( \|x^{k}-y^{k-1}\|^2 + \beta_k \|\theta^{k}-\theta^*\|^2 ) \\
 \notag	&\quad\  + 2(1 + \gamma_{k,\theta}^2 C_{\theta}^2 + 2\gamma_{k,\theta} \gamma_{k,x}^{2-\tau} \mu_{\theta}  L_{\theta}^2)  \|x^{k}-y^{k-1}\| \|y^{k}-y^{k-1}\|  + (1 + \gamma_{k,\theta}^2 C_{\theta}^2 + 2\gamma_{k,\theta} \gamma_{k,x}^{2-\tau} \mu_{\theta}  L_{\theta}^2)  \|y^{k}-y^{k-1}\|^2  \\
\notag	&\quad\  - (\gamma_{k,\theta}^2 C_{\theta}^2 + 2\gamma_{k,\theta} \gamma_{k,x}^{2-\tau} \mu_{\theta}  L_{\theta}^2 + 2\gamma_{k,x} \epsilon_k) \|x^{k}-y^k\|^2 + (\gamma_{k,x}^2 (L_x^2+(\epsilon_k)^2) +\gamma_{k,x}^{2-\tau}L^2_{\theta} ) \|x^{k}-y^k\|^2  + \beta_k \gamma_{k,\theta}^2 \nu_{\theta}^2  + \gamma_{k,x}^2  \nu_x^2,
\end{align*}
which can be further reduced to
\begin{align*}
\notag  & \quad \ \mathbb{E}[\|x^{k+1}-y^k\|^2\mid\mathcal{F}_k] + \beta_{k+1} \mathbb{E}[\|\theta^{k+1}-\theta^{*}\|^2 \mid\mathcal{F}_k]  \\
  \notag	& \leq (1 + \gamma_{k,\theta}^2 C_{\theta}^2 + 2\gamma_{k,\theta} \gamma_{k,x}^{2-\tau} \mu_{\theta}  L_{\theta}^2) ( \|x^{k}-y^{k-1}\|^2 + \beta_k \|\theta^{k}-\theta^*\|^2 ) \\
 \notag	& \quad\ + 2(1 + \gamma_{k,\theta}^2 C_{\theta}^2 + 2\gamma_{k,\theta} \gamma_{k,x}^{2-\tau} \mu_{\theta}  L_{\theta}^2)  \|x^{k}-y^{k-1}\| \|y^{k}-y^{k-1}\|  \\
 &\quad\ + (1 + \gamma_{k,\theta}^2 C_{\theta}^2 + 2\gamma_{k,\theta} \gamma_{k,x}^{2-\tau} \mu_{\theta}  L_{\theta}^2)  \|y^{k}-y^{k-1}\|^2  \\
\notag	& \quad\ - 2\gamma_{k,x} \epsilon_k \|x^{k}-y^k\|^2
 + (\gamma_{k,x}^2 (L_x^2+(\epsilon_k)^2) +\gamma_{k,x}^{2-\tau}L^2_{\theta} ) \|x^{k}-y^k\|^2 + \beta_k \gamma_{k,\theta}^2 \nu_{\theta}^2  + \gamma_{k,x}^2  \nu_x^2.
\end{align*}
By Lemma \ref{lem:VI_monotone_Tik} and (A2-3),
$\sum_{k=0}^{\infty}  \|y^{k}-y^{k-1}\| < \infty.$ and $\sum_{k=0}^{\infty}  \|y^{k}-y^{k-1}\|^2 < \infty.$
Therefore, by boundedness of $X$, (A2-3) and Lemma \ref{lem:supermartingale2}, we have that there exists a random variable $V$ such that
\begin{align*}
    \|x^{k}-y^{k-1}\|^2 + \beta_{k} \|\theta^{k}-\theta^{*}\|^2  \to V \quad  a.s.  \quad \textrm{as } k\to\infty.
\end{align*}
and $\sum_{k=0}^{\infty} 2\gamma_{k,x}\epsilon_k \|x^{k}-y^{k}\|^2 <
  \infty.$ Since $\sum_{k=0}^{\infty}\gamma_{k,x}\epsilon_k = \infty$, we get $\|x^{k}-y^{k}\|  \to 0  $ $a.s.$ $ \textrm{as } k\to\infty$.
This implies  $\|x^{k}-x^{*}\|  \to 0  $ $a.s.$ $ \textrm{as } k\to\infty$.
\end{proof}

\subsection{Diminishing and constant steplength error
	analysis}\label{sec:III.III}
In this section, we estimate the convergence rate of the proposed schemes.
Analogous to Section \ref{sec:II.III}, we obtain the optimal
$\Oscr(1/K)$ rate estimate for the upper
bound on the expected error in the solution $x_K$ when $F(\bullet;\theta^*)$ is strongly monotone
in ($\bullet$).
In addition, when $F(\bullet;\theta^*)$ is merely monotone and the
variational inequality problem possesses the \textit{minimum principle
	sufficiency} (MPS) property (See Lemma \ref{lem:weakly_sharp} for a
			definition of the MPS property), a rate estimate is still available by using averaging.
If we replace
$\nabla_x f$ and $\nabla_{\theta} g$ by $F$ and $G$, respectively, in Theorem \ref{thm:stoch_strongly_optim_error},
then we obtain the following:
\begin{thm}[{\bf Rate estimate for strongly monotone $F$}] \label{thm:stoch_strongly_VI_error}
Suppose (A1-3) and (A\ref{assump:filtration}) hold. Suppose $\gamma_{x,k}=\lambda_{x}/k$
and $\gamma_{\theta,k}=\lambda_{\theta}/k$ with $\lambda_{x}>1/\mu_{x}$
and $\lambda_{\theta}>1/(2\mu_{\theta})$.  Let
$\mathbb{E}[\|F(x^{k};\theta^{k})+ w^k \|^2]\leq M^2$  and
$\mathbb{E}[\|G(\theta^{k})+ v^k \|^2]\leq M_{\theta}^2$  for all
$x^k\in X$ and $\theta^k\in\Theta$.
Suppose $x^*$ is the unique solution to VI$(X,\mathbb{E}[F(\bullet;\theta^*,\xi)])$.
Let $\{x^k,\theta^k\}$
be computed via Algorithm \ref{alg:VI_stoch_grad_grad_strongly}.  Then, the
following hold:
\begin{align*}
  \mathbb{E}[\|\theta^{k} - \theta^{*} \|^2 ] \leq
  \frac{Q_{\theta}(\lambda_{\theta})}{k} \mbox{ and }
   \mathbb{E}[\|x^{k} - x^{*} \|^2 ] \leq \frac{Q_{x}(\lambda_{x})}{k},
\end{align*}
\begin{align*}
  \mbox{ where } Q_{\theta}(\lambda_{\theta})&\triangleq \max\left\{ \lambda_{\theta}^2 M_{\theta}^2 (2\mu_{\theta}\lambda_{\theta} - 1)^{-1}, \mathbb{E}[\|\theta^{1} - \theta^{*} \|^2 ] \right\},
  Q_{x}(\lambda_{x})\triangleq \max\left\{ \lambda_{x}^2 \widetilde{M}^2
(\mu_{x}\lambda_{x} - 1)^{-1}, \mathbb{E}[\|x^{1} - x^{*} \|^2 ]
\right\}, \\
 \mbox{ and } \widetilde{M} &\triangleq \sqrt{ M^2 + \frac{L_{\theta}^2Q_{\theta}(\lambda_{\theta})}{\mu_x \lambda_{x}} }.
\end{align*}
\end{thm}\vspace{0.2in}

Next, we weaken the strong monotonicity of $F$, but assume that
\eqref{problem_stoch_optim_1} satisfies the MPS property,
introduced in the following Lemma. Note that this property  guarantees
weak sharpness of the solution set; this is analogous to  weak-sharpness
of minima in optimization problems~\cite{burke93weaksharp}.
\begin{lem}[Theorem 4.3 in \cite{Marcotte99weaksharp}] \label{lem:weakly_sharp}
Let $H:X\to\mathbb{R}^n$ be a mapping that is monotone over the compact polyhedral set $X$.
Let $X^*$ be the solution set of VI($X,H$).
If the VI$(X,H)$ possesses the minimum principle sufficiency (MPS) property, then
there exists a positive number $\alpha$ such that
$(x-x^*)^TH(x^*)\geq\alpha\; \mathrm{dist}(x,X^*),\quad \forall x\in X,\quad\forall x^*\in X^*,$
where $\mathrm{dist}(x,X^*)\triangleq\min_{x^*\in X^*}\|x-x^*\|$.
We say that the VI($X,H$) possesses the MPS property if
$\Gamma(x^*)=X^*$ for every $x^*$ in $X^*$, where
$\Gamma(x)=\arg \max_{y\in X} (x-y)^T H(x).$
\end{lem}

By leveraging this property, we may estimate the convergence rate by
using averaging as in Theorem
\ref{thm:stoch_monotone_optim_error}.\vspace{0.2in}
\begin{thm} [{\bf Rate estimates under monotonicity of $F$}] \label{thm:stoch_monotone_VI_error}
Suppose (A1-4) and (A\ref{assump:filtration}) hold.
Suppose $\mathbb{E}[\|x^{k}-x^{*}\|^2]\leq M_x^2$, $\mathbb{E}[\|F(x^{k};\theta^{k})+ w^k \|^2]\leq M^2$  and $\mathbb{E}[\|G(\theta^{k})+ v^k \|^2]\leq M_{\theta}^2$  for all $x^{k}\in X$ and $\theta^{k}\in\Theta$.
Suppose $X$ is a compact polyhedral set, the solution set $X^*$ of VI$(X,\mathbb{E}[F(\bullet;\theta^*,\xi)])$ is nonempty, and $x^*$ is a point in $X^*$.
Suppose  VI$(X,\mathbb{E}[F(\bullet;\theta^*,\xi)])$ possesses the MPS property.
Let $\{x^k,\theta^k\}$ be computed via Algorithm \ref{alg:VI_stoch_grad_grad_strongly}.
For $1\leq i,t\leq k$, we define $v_t\triangleq\frac{\gamma_{x,t}}{\sum_{s=i}^k \gamma_{x,s}}$, $\tilde{x}_{i,k}\triangleq \sum_{t=i}^k v_t x^t$ and $D_X\triangleq\max_{x\in X}\|x-x^1\|$.
Suppose
 for $1\leq t\leq k$
$$\gamma_x = \sqrt{\frac{4D_X^2  + L_{\theta}^2Q_{\theta}(\lambda_{\theta}) (1+\ln k)}{(M^2 +  M_x^2)k}},$$
where $Q_{\theta}(\lambda_{\theta})\triangleq \max\left\{ \lambda_{\theta}^2 M_{\theta}^2 (2\mu_{\theta}\lambda_{\theta} - 1)^{-1}, \mathbb{E}[\|\theta^{1} - \theta^{*} \|^2 ] \right\}$,
and $\gamma_{\theta,k}=\lambda_{\theta}/k$ with
$\lambda_{\theta}>1/(2\mu_{\theta})$.  Then there exists a positive number $\alpha$ such that for
$1\leq i\leq k$:
\begin{align}
\notag\mathbb{E}\left[\alpha\;\mathrm{dist}(\tilde{x}_{i,k},X^*)\right]   &\leq  C_{i,k}\sqrt{\frac{ B_k}{k}},
\end{align}
where
$C_{i,k}=\frac{k}{k-i+1}$ and $B_k=(4D_X^2  + L_{\theta}^2Q_{\theta}(\lambda_{\theta}) (1+\ln k))(M^2 +  M_x^2)$.
\end{thm}\vspace{0.2in}
\begin{proof}
By using the same notation in Theorem  \ref{thm:stoch_monotone_optim_error} except that we replace $\nabla_x f$ and $\nabla_{\theta} g$ by $F$ and $G$, respectively, we have from \eqref{alg:stoch_grad_grad_monotone_constant_a} that
\begin{align}
\begin{aligned} \label{alg:VI_stoch_grad_grad_monotone_constant_a}
  a_{k+1}
    & \leq a_k + \frac{1}{2} \gamma_{x,k}^2 M^2 - \gamma_{x,k}
   \mathbb{E}[(x^{k}-x^{*})^T F(x^{k};\theta^{*})] -   \gamma_{x,k} \mathbb{E}[(x^{k}-x^{*})^T (F(x^{k};\theta^{k})
		   - F(x^{k};\theta^{*}) ) ].
\end{aligned}
\end{align}
By Lemma \ref{lem:weakly_sharp}, we have that
there exists a positive number $\alpha$ such that
\begin{align}
\begin{aligned}\label{eq:VI_MPS}
\alpha\; \mathrm{dist}(x^k,X^*)& \leq (x^k-x^*)^T F(x^*;\theta^*)
 = (x^k-x^*)^T F(x^k;\theta^*) - (x^k-x^*)^T (F(x^k;\theta^*)-F(x^*;\theta^*)) \\
  & \leq (x^k-x^*)^T F(x^k;\theta^*),
\end{aligned}
\end{align}
where the last inequality follows from the monotonicity of $F(\bullet;\theta^*)$ in ($\bullet$).
Combining \eqref{alg:VI_stoch_grad_grad_monotone_constant_a} and \eqref{eq:VI_MPS},
\begin{align}
\begin{aligned} \label{eq:combine_sharp}
  \alpha \gamma_{x,k} \mathbb{E}[\mathrm{dist}(x^k,X^*)]
  & \leq \gamma_{x,k}
   \mathbb{E}[(x^{k}-x^{*})^T F(x^{k};\theta^{*})] \\
   &\leq   a_k -a_{k+1} + \frac{1}{2} \gamma_{x,k}^2 M^2 -   \gamma_{x,k} \mathbb{E}[(x^{k}-x^{*})^T (F(x^{k};\theta^{k})
		   - F(x^{k};\theta^{*}) ) ].
\end{aligned}
\end{align}
Next, we follow the same proof method in Theorem  \ref{thm:stoch_monotone_optim_error}.
We define $v_t\triangleq\frac{\gamma_{x,t}}{\sum_{s=i}^k
	\gamma_{x,s}}$ and $D_X\triangleq\displaystyle \max_{x\in X}\|x-x^1\|$.
It follows from \eqref{alg:stoch_grad_grad_monotone_constant_ab_final_average} and \eqref{eq:combine_sharp} that
\begin{align}
\begin{aligned}  \label{alg:stoch_grad_grad_monotone_constant_ab_final_average_sharp}
 \mathbb{E}\left[\alpha\sum_{t=i}^kv_t \mathrm{dist}(x^t,X^*) \right]&
   \leq  \frac{a_i  +\frac{1}{2} \sum_{t=i}^k \gamma_{x,t}^2 (M^2 +  M_x^2) +  \frac{1}{2} L_{\theta}^2Q_{\theta}(\lambda_{\theta}) (1+\ln k)}{\sum_{t=i}^k\gamma_{x,t}}.
\end{aligned}
\end{align}
Next, we consider points given by $\tilde{x}_{i,k}\triangleq \sum_{t=i}^k v_t x^t$.
Since $F(x;\theta^*)$ is monotone in $x$, we have that $X^*$ is convex, which implies that $\mathrm{dist}(x,X^*)$ is convex in $x$.
So, we get
$\mathrm{dist}(\tilde{x}_{i,k},X^*)  \leq \sum_{t=i}^k v_t \mathrm{dist}(x^t,X^*) $.
It follows from \eqref{alg:stoch_grad_grad_monotone_constant_ab_final_average_convex} and \eqref{alg:stoch_grad_grad_monotone_constant_ab_final_average_sharp} that for $1\leq i \leq k$
\begin{align} \label{eq:average_sharp}
  \quad\ \mathbb{E}\left[\alpha\;\mathrm{dist}(\tilde{x}_{i,k},X^*)\right]  \leq  \frac{4D_X^2   +\sum_{t=i}^k \gamma_{x,t}^2 (M^2 +  M_x^2) + L_{\theta}^2Q_{\theta}(\lambda_{\theta}) (1+\ln k)}{2\sum_{t=i}^k\gamma_{x,t}}.
\end{align}
Suppose $\gamma_{x,t}=\gamma_{x}$ for $t=1,\ldots,k$. If we follow the same proof method in Theorem  \ref{thm:stoch_monotone_optim_error},
then we can get from \eqref{eq:error_monotone_f} and \eqref{eq:average_sharp} that
\begin{align*} 
   \mathbb{E}\left[\alpha\;\mathrm{dist}(\tilde{x}_{i,k},X^*)\right]   &\leq  C_{i,k}\sqrt{\frac{ B_k}{k}}.
\end{align*}
\end{proof}

{\indent The following corollary is a special case of Theorem \ref{thm:stoch_monotone_VI_error}, an avenue that has been adopted in \cite{Nemirovski09}.}\\

\begin{cor} [{\bf Rate estimates under monotonicity of $F$}]
Suppose (A1-4) and (A\ref{assump:filtration}) hold.
Suppose $\mathbb{E}[\|x^{k}-x^{*}\|^2]\leq M_x^2$, $\mathbb{E}[\|F(x^{k};\theta^{k})+ w^k \|^2]\leq M^2$  and $\mathbb{E}[\|G(\theta^{k})+ v^k \|^2]\leq M_{\theta}^2$  for all $x^{k}\in X$ and $\theta^{k}\in\Theta$.
Suppose $X$ is a compact polyhedral set, the solution set $X^*$ of VI$(X,\mathbb{E}[F(\bullet;\theta^*,\xi)])$ is nonempty, and $x^*$ is a point in $X^*$.
Suppose  VI$(X,\mathbb{E}[F(\bullet;\theta^*,\xi)])$ possesses the MPS property.
Let $\{x^k,\theta^k\}$ be computed via Algorithm \ref{alg:VI_stoch_grad_grad_strongly}.
For $k/2\leq t\leq k$, we define $v_t\triangleq\frac{\gamma_{x,t}}{\sum_{s=k/2}^k \gamma_{x,s}}$, $\tilde{x}_{k/2,k}\triangleq \sum_{t=k/2}^k v_t x^t$ and $D_X\triangleq\max_{x\in X}\|x-x^1\|$.
Suppose
 for $1\leq t\leq k$
$$\gamma_x = \sqrt{\frac{4D_X^2  + L_{\theta}^2Q_{\theta}(\lambda_{\theta}) (1+\ln 2)}{(M^2 +  M_x^2)k}},$$
where $Q_{\theta}(\lambda_{\theta})\triangleq \max\left\{ \lambda_{\theta}^2 M_{\theta}^2 (2\mu_{\theta}\lambda_{\theta} - 1)^{-1}, \mathbb{E}[\|\theta^{1} - \theta^{*} \|^2 ] \right\}$,
and $\gamma_{\theta,k}=\lambda_{\theta}/k$ with
$\lambda_{\theta}>1/(2\mu_{\theta})$.  Then there exists a positive number $\alpha$ such that
\begin{align}
\notag\mathbb{E}\left[\alpha\;\mathrm{dist}(\tilde{x}_{k/2,k},X^*)\right]   &\leq  2\sqrt{\frac{ B }{k}},
\end{align}
where
$B=(4D_X^2  + L_{\theta}^2Q_{\theta}(\lambda_{\theta}) (1+\ln 2))(M^2 +  M_x^2)$.
\end{cor}\vspace{0.2in}
\begin{proof}{When $i=k/2$ where $k$ is a positive even number,  then by utilizing the same approach as in Corollary \ref{cor:optim}, inequality \eqref{eq:average_sharp} becomes the following:
\begin{align}
  \quad\ \mathbb{E}\left[\alpha\;\mathrm{dist}(\tilde{x}_{k/2,k},X^*)\right]  \leq  \frac{4D_X^2   +\sum_{t=k/2}^k \gamma_{x,t}^2 (M^2 +  M_x^2) + L_{\theta}^2Q_{\theta}(\lambda_{\theta}) (1+\ln 2)}{2\sum_{t=k/2}^k\gamma_{x,t}}.
\end{align}
}{
Suppose $\gamma_{x,t}=\gamma_{x}$ for $t=1,\ldots,k$. By utilizing the same techniques as in Theorem  \ref{thm:stoch_monotone_VI_error},
then we obtain the following bound:
\begin{align*} 
   \mathbb{E}\left[\alpha\;\mathrm{dist}(\tilde{x}_{k/2,k},X^*)\right]   &\leq  2\sqrt{\frac{ B }{k}},
\end{align*}
where $B \triangleq (4D_X^2  + L_{\theta}^2Q_{\theta}(\lambda_{\theta}) (1+\ln 2))(M^2 +  M_x^2)$.
}
\end{proof}

Next, we present a constant steplength error bound.
\vspace{0.2in}
\begin{prop}[{\bf Constant steplength error bound}] 
	Suppose (A\ref{assump:filtration}) holds. Suppose $\gamma_{\theta,k} := \gamma_{\theta}$
and $\gamma_{x,k} := \gamma_{x}$.  Suppose $\mathbb{E}[\|x^{k}-x^{*}\|^2]\leq M_x^2$ and
$\mathbb{E}[F(x^{k};\theta^{k})+ w^k \|^2]\leq M^2$    for all
$x^k\in X$.
Suppose $A_k \triangleq \frac{1}{2} \|x^k - x^*\|^2$ and $a_k \triangleq \mathbb{E}[A_k]$.
Suppose $X$ is a compact polyhedral set, the solution set $X^*$ of VI$(X,\mathbb{E}[F(\bullet;\theta^*,\xi)])$ is nonempty, and $x^*$ is a point in $X^*$.
Suppose  VI$(X,\mathbb{E}[F(\bullet;\theta^*,\xi)])$ possesses the MPS property.
Let $\{x^k,\theta^k\}$
be computed via Algorithm \ref{alg:stoch_strongly_optim}.
\begin{itemize}
  \item[(i)] Suppose (A1-3) holds. Then, the following holds:
  \begin{align*}
 \limsup_{k \to \infty}   a_{k} &  \leq  \frac{1}{2\mu_x} \gamma_{x} M^2  +
	 \frac{1}{2}\frac{1}{\mu^2_x}L_{\theta}^2 \frac{ \gamma_\theta
		\nu_{\theta}^2}{2\mu_{\theta} - \gamma_{\theta} C_{\theta}^2};
		\end{align*}
  \item[(ii)] Suppose (A1-4) holds. Then,   there exists a positive number $\alpha$ such that:
 \begin{align}
\notag & \quad\   \limsup_{k \to \infty} \mathbb{E}[\mathrm{dist}(x^k,X^*)]    \leq \frac{1}{\alpha} \left[\frac{1}{2} \gamma_{x} M^2 + \frac{1}{2}\gamma_{x}^{1-\tau}  M_x^2 +\frac{1}{2} \gamma_{x}^{\tau-1} L_{\theta}^2 \frac{ \gamma_{ \theta}
		\nu_{\theta}^2}{2\mu_{\theta} - \gamma_{ \theta} C_{\theta}^2} \right],
\end{align}
where $0<\tau<1$.
\end{itemize}
\end{prop}\
\begin{proof}
If we replace
$\nabla_x f$ and $\nabla_{\theta} g$ by $F$ and $G$ in Proposition
\ref{prop:error_optim}, we obtain that
\begin{align*}
  \limsup_{k \to \infty} \mathbb{E}[\|\theta^k - \theta^*\|^2] \leq \frac{ \gamma_{ \theta}
		\nu_{\theta}^2}{2\mu_{\theta} - \gamma_{ \theta} C_{\theta}^2},
\end{align*}
and the following can be derived based on the properties of $F$:
\begin{enumerate}
\item[(i)] $F$ is strongly monotone:
\begin{align*}
 \limsup_{k \to \infty}   a_{k} &  \leq  \frac{1}{2\mu_x} \gamma_{x} M^2  +
	 \frac{1}{2}\frac{1}{\mu^2_x}L_{\theta}^2 \frac{ \gamma_\theta
		\nu_{\theta}^2}{2\mu_{\theta} - \gamma_{\theta} C_{\theta}^2};
		\end{align*}
\item[(ii)] $f$ is convex:
From \eqref{eq:combine_sharp}, for $\gamma_{x,k} := \gamma_x$, we have that there exists a positive number $\alpha$ such that:
\begin{align*}
  \alpha \gamma_{x} \mathbb{E}[\mathrm{dist}(x^k,X^*)]
  & \leq \gamma_{x}
   \mathbb{E}[(x^{k}-x^{*})^T F(x^{k};\theta^{*})] \\
   &\leq   a_k -a_{k+1} + \frac{1}{2} \gamma_{x}^2 M^2 -   \gamma_{x} \mathbb{E}[(x^{k}-x^{*})^T (F(x^{k};\theta^{k})
		   - F(x^{k};\theta^{*}) ) ]\\
&\leq a_k -a_{k+1}+ \frac{1}{2} \gamma_{x}^2 M^2 +
\frac{1}{2}\gamma_{x}^{2-\tau} \mathbb{E}[\|x^{k}-x^{*}\|^2]  +\frac{1}{2}\gamma_{x}^{\tau}
\mathbb{E}[\|F(x^{k};\theta^{k}) - F(x^{k};\theta^{*}) \|^2 ] \\
\notag    &\leq a_k -a_{k+1}+ \frac{1}{2} \gamma_{x}^2 M^2 + \frac{1}{2}\gamma_{x}^{2-\tau} M_x^2 +\frac{1}{2} \gamma_{x}^{\tau} L_{\theta}^2\mathbb{E}[\|\theta^{k} -  \theta^{*} \|^2],
\end{align*}
where $0<\tau<1$.
It follows that
\begin{align*}
 \alpha \gamma_{x} \mathbb{E}[\mathrm{dist}(x^k,X^*)]    &\leq \limsup_{k \to \infty}  a_k - \limsup_{k \to \infty}  a_{k+1} + \frac{1}{2} \gamma_{x}^2 M^2 + \frac{1}{2}\gamma_{x}^{2-  \tau } M_x^2+\frac{1}{2}\gamma_{x}^{\tau} L_{\theta}^2\limsup_{k \to \infty} \mathbb{E}[\|\theta^{k} -  \theta^{*} \|^2] \\
\notag    &\leq  \frac{1}{2} \gamma_{x}^2 M^2 + \frac{1}{2}\gamma_{x}^{2-\tau} M_x^2 +\frac{1}{2}\gamma_{x}^{\tau} L_{\theta}^2 \frac{ \gamma_{ \theta}
		\nu_{\theta}^2}{2\mu_{\theta} - \gamma_{ \theta} C_{\theta}^2}.
\end{align*}
It follows that
\begin{align}
\notag & \quad\   \limsup_{k \to \infty} \mathbb{E}[\mathrm{dist}(x^k,X^*)]    \leq \frac{1}{\alpha} \left[\frac{1}{2} \gamma_{x} M^2 + \frac{1}{2}\gamma_{x}^{1-\tau}  M_x^2 +\frac{1}{2} \gamma_{x}^{\tau-1} L_{\theta}^2 \frac{ \gamma_{ \theta}
		\nu_{\theta}^2}{2\mu_{\theta} - \gamma_{ \theta} C_{\theta}^2} \right].
\end{align}

\end{enumerate}
\end{proof}

\section{Numerical results} \label{sec:IV}

In this section, we apply the developed algorithms
on a class of misspecified economic dispatch problems described in Section~\ref{sec:num_prob}.
In Section~\ref{sec:num_comp}, we apply the proposed schemes for
purposes of learning optimal solutions and the misspecified parameters.
Note that the simulations were carried
out on Tomlab 7.4.  The complementarity solver
\texttt{PATH}~\cite{Ferris98complementarity} was utilized for obtaining
solutions to these problems which subsequently formed the basis for
comparison.

\subsection{Problem description} \label{sec:num_prob}

We consider a  setting where there are $N$
firms competing over a $W$-node network.  Firm $f$ may produce and sell
its good at node $i$, where $f=1,\ldots,N$ and
 $i=1,\ldots,W$.  We assume that for a given firm $f$, the cost of
 generating $x_{fi}$ units of power at node $i$ is random and is given
 by $c_{fi}(x_{fi})=d_{fi}x_{fi}^2+h_{fi}x_{fi}+\xi_{fi}$, where $d_{fi}$ and $h_{fi}$ are positive
parameters, and $\xi_{fi}$ is a random variable with mean zero  for all $f$ and $i$. Furthermore, the generation level associated with
 firm $f$ is bounded by its  production capacity,  which is denoted by $\textrm{cap}_{fi}$.
The aggregate sales of all firms at node $i$ has to satisfy the demand $D_i$ at node $i$.  A given firm can produce at any node and then sell at different
nodes, provided that the aggregate production at all nodes matches the
aggregate sales at all nodes for each firm. For simplicity, we assume
that there is no limit of sales at any node.  Then, the resulting
problem faced by the grid operator can be stated as follows:
\begin{align}
 \begin{aligned}  \label{prob:nash_network}
    \min_{ x_{fi} \geq 0} \quad & \mathbb{E}\left[\sum_{f=1}^N\sum_{i=1}^W c_{fi}(x_{fi})\right]\\
    \st \quad &  x_{fi} \leq \textrm{cap}_{fi}, \qquad\qquad \textrm{for all } f,i \\
    & \sum_{f=1}^N x_{fi} = D_i.
\end{aligned}
\end{align}
The resulting optimal solution is given by $x^*$. {Suppose firm $f$ generates $y_{fi}$ units of power at node $i$. We use $c_{fi}(y_{fi})=d_{fi}(y_{fi})^2+h_{fi}y_{fi}+\xi_{fi}$ to denote the
cost associated with firm $f$ at node $i$.
The operator will solve the following (regularized) problem to estimate $c_{fi}$ and $d_{fi}$:
\begin{align} \label{prob:nash_network_theta}
 \begin{aligned}
    \min_{ \{d_{fi} ,h_{f,i}\}\in \Theta} \quad & \mathbb{E}\left[   ( d_{fi}(y_{fi})^2+h_{fi}y_{fi} - c_{fi}(y_{fi}) )^2 +
     \mu_{\theta}   d_{fi}^2 +  \mu_{\theta}
	  h_{fi}^2\right].
\end{aligned}
\end{align}
The resulting optimal solution is given by $\theta^*$.
We assume that $y_{fi}$ is distributed as per a uniform distribution
and is specified  by $y_{fi}\sim U[0, \textrm{cap}_{fi}]$,
while that the noise $\xi_{fi}$ is distributed as per a uniform distribution
and is specified  by $\xi_{fi}\sim U[-\theta_{fi}^*/2,\theta_{fi}^*/2]$.}
\subsection{Results} \label{sec:num_comp}
In this subsection, we employ Algorithm \ref{alg:stoch_strongly_optim}  proposed in Section \ref{sec:II} for  learning parameters and computing optimal solutions.
We will examine the behavior and error bounds of the algorithm.
\subsubsection{Behavior of the algorithm}
In this part, we consider a special case when $N=5$ and $W=5$.
Suppose, the noise $\xi$ is distributed as per a uniform distribution
and is specified  by $\xi\sim U[-\theta^*/2,\theta^*/2]$. Suppose
the steplength sequences $\{\gamma_{k,x}\}$ and $\{\gamma_{k,\theta}\}$ are chosen according to
Proportion \ref{thm:stoch_strongly_optim_error}: {$\gamma_{k,x}=1/k$ and $ \gamma_{k,\theta}=40/k$}.
Figure \ref{figure:behavior} illustrates the  scaled error of the learning scheme
when the number of steps increases.
\begin{figure}[htb]
\subfigure[Normalized error]{  \begin{minipage}[b]{.49\textwidth}
    \centering
    \includegraphics[width=\textwidth]{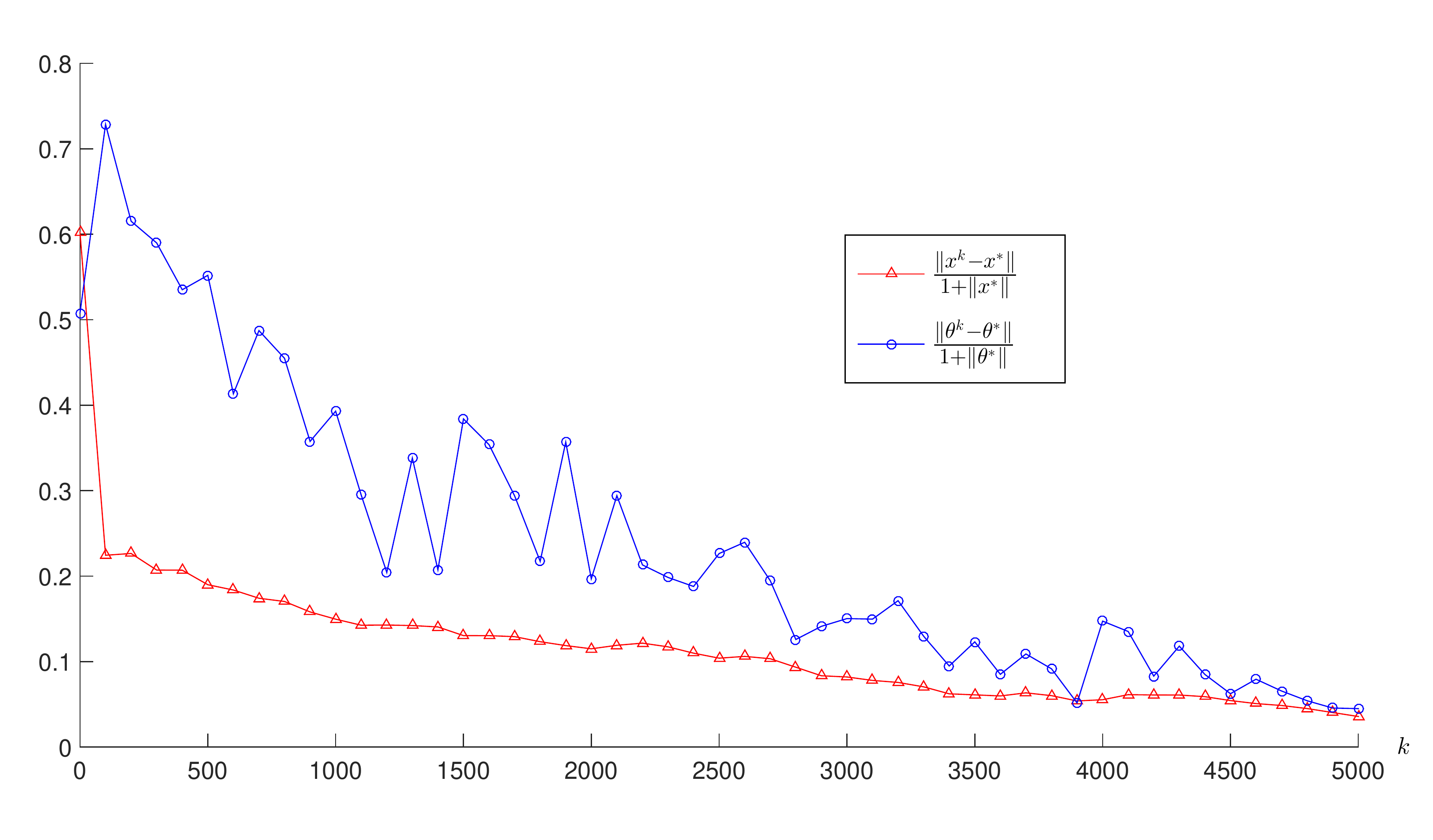}
   \label{figure:behavior}
  \end{minipage}}
\subfigure[Normalized regret]{  \begin{minipage}[b]{0.49\textwidth}
    \centering
    \includegraphics[width= \textwidth]{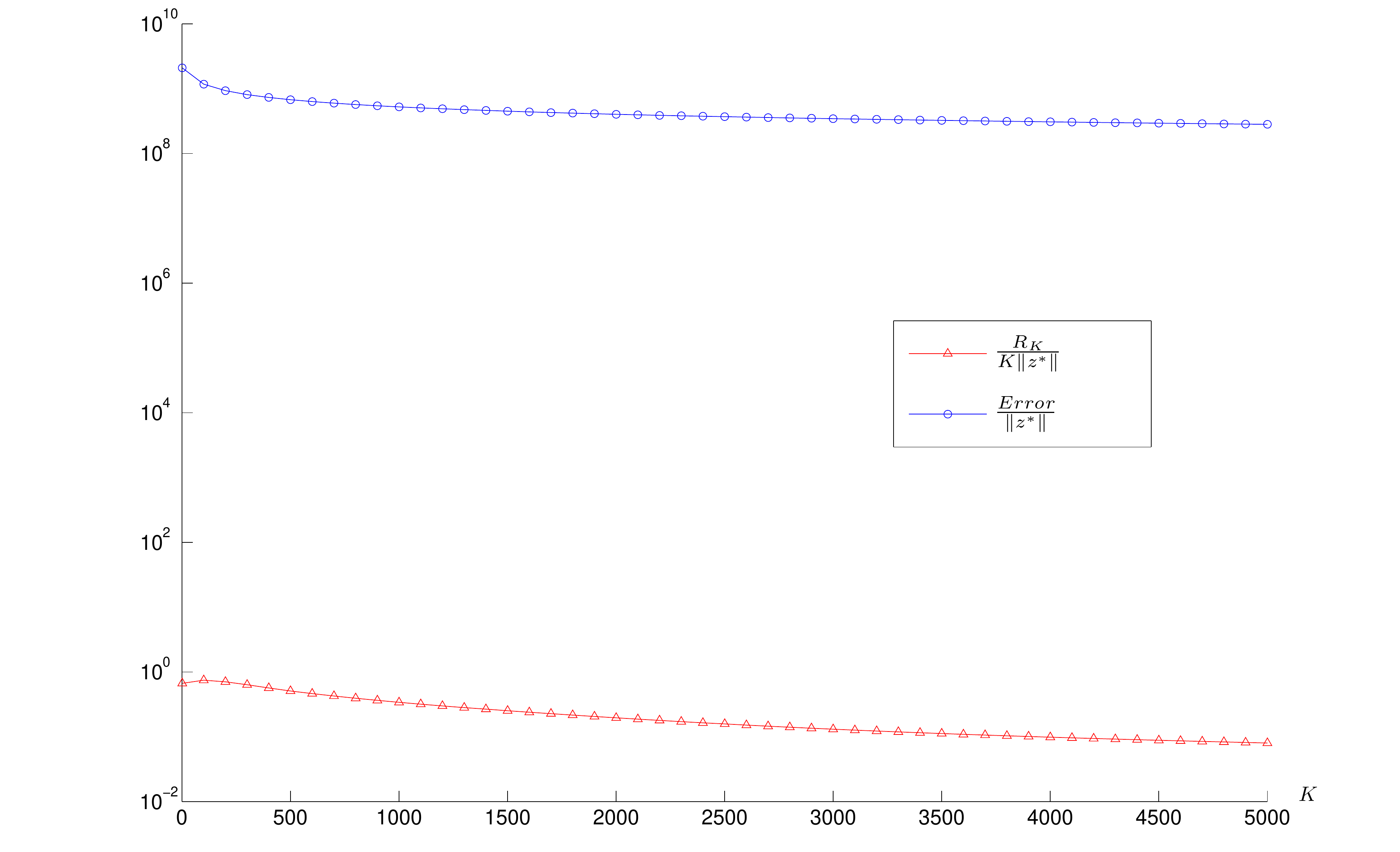}
   \label{figure:Regret_norm}
  \end{minipage}}
  \caption{Computing $x^*$ and learning $\theta^*$ ($\xi\sim U[-{\theta}^*/2,{\theta}^*/2]$,
			$N=5$, $W=5$)}
\end{figure}

\subsubsection{Error bounds}
In this part, we examine the errors of the algorithm and compare them with the theoretical error bounds proposed in Section \ref{sec:II}.
Suppose, the noise $\xi$ is distributed as per a uniform distribution
and is specified  by $\xi\sim U[-\theta^*/2,\theta^*/2]$.
\begin{itemize}
  \item[(a)]
  In the strongly convex regime, suppose
the steplength sequences $\{\gamma_{k,x}\}$ and $\{\gamma_{k,\theta}\}$ are chosen according to
Proportion \ref{thm:stoch_strongly_optim_error}: {$\gamma_{k,x}=1/k$ and $ \gamma_{k,\theta}=40/k $}.
We use \textrm{ERR} to denote the theoretical error provided in Proportion \ref{thm:stoch_strongly_optim_error}.
The algorithm was terminated at $K=10000$.  Table
	\ref{table:optim_error} (L) shows the scaled errors of the learning scheme.
  \item[(b)] In the merely convex regime, suppose the steplength $\gamma_{x} $ and
the steplength sequence  $\{\gamma_{k,\theta}\}$ are chosen according to
Theorem \ref{thm:stoch_monotone_optim_error}:  $ \gamma_{x} $ is chosen by Table \ref{table:optim_error} (R) and {$\gamma_{k,\theta}=40/k$}.
We use \textrm{ERR} to denote the theoretical error provided in Theorem \ref{thm:stoch_monotone_optim_error} while $z^*$ denotes $f(x^*;\theta^*)$.
The algorithm was terminated at $K=10000$ and Table \ref{table:optim_error}(R) shows the scaled errors of the learning scheme.

  \item[(c)]

  Suppose
the steplength sequences $\{\gamma_{k,x}\}$ and  $\{\gamma_{k,\theta}\}$ are chosen according to
Theorem \ref{thm:regret}:  {$\gamma_{k,x}=k^{-\alpha}$ and $ \gamma_{k,\theta}=40/k $}.
We employ \textrm{ERR} to denote the theoretical error provided in Theorem \ref{thm:regret} while $z^*$ denotes $f(x^*;\theta^*)$.
The algorithm was terminated after $K=10000$ iterations.  Figure \ref{figure:Regret_norm}
illustrates the  scaled regret and scaled theoretical error of the learning scheme
when the number of steps increases {($\alpha=\beta=0.5$)}. Table \ref{table:optim_error_alpha}
shows the scaled theoretical error of the learning scheme
for different chosen $\gamma_{k,x}=k^{-\alpha}$ with $\alpha=0.5,0.6,0.7,0.8,0.9$ {when $\beta=0.5$}.
We see that when $\alpha$ changes, error bounds change marginally primarily because {the last term} in Theorem \ref{thm:regret} dominates the bound.

\end{itemize}

{
\begin{table}[htbp]
\caption{Learning $x^*$ and $\theta^*$ in a strongly convex (L) and convex (R) regime:
$\xi\sim U[-{\theta}^*/2,{\theta}^*/2]$}   \label{table:optim_error}
\tiny
\centering
\begin{tabular}{|c|c|c|c|c|c|c|c|c|c|c|c|}
  \hline
  N & W & $\frac{\mathbb{E}[\|x^K- x^*\|]}{1+\|x^*\|}$  & $\frac{\textrm{ERR}}{1+\|x^*\|}$  &  $\frac{\|\mathbb{E}[ \theta^K- \theta^*\|]}{1+\|\theta^*\|}$   & $\frac{\textrm{ERR}}{1+\|\theta^*\|}$   \\
 \hline
  10 &  	 2&	 7.3$\times 10^{-3}$  &	 9.2$\times 10^{9}$  & 4.8$\times 10^{-2}$ & 3.7$\times 10^{4}$   \\
  \hline
  10  &	  4 & 	  3.7$\times 10^{-2}$  &	 2.1$\times 10^{10}$  & 4.9$\times 10^{-2}$ & 3.1$\times 10^{4}$  \\
  \hline
   10 & 	 6 &    3.8$\times 10^{-2}$  &	 7.8$\times 10^{10}$  & 4.7$\times 10^{-2}$  & 8.3$\times 10^{4}$   \\
  \hline
  10 & 	 8 &    1.7$\times 10^{-2}$  &	 9.1$\times 10^{10}$  & 4.8$\times 10^{-2}$ & 8.5$\times 10^{4}$  \\
  \hline
  10 & 	  10 &   2.4$\times 10^{-2}$  &	 1.2$\times 10^{11}$  & 4.3$\times 10^{-2}$ & 8.6$\times 10^{4}$   \\
  \hline
\end{tabular}
\begin{tabular}{|c|c|c|c|c|c|c|c|c|c|c|c|}
  \hline
  N & W & $\frac{\mathbb{E}[ f(\tilde{x}_{1,K};\theta^{K}) - z^{*}  ]}{1+\|z^*\|}$  & $\frac{\textrm{ERR}}{1+\|x^*\|}$   & $\gamma_x$  \\
 \hline
  10 &  	 2&	 1.9$\times 10^{-1}$  &	 2.5$\times 10^{5}$   & 72 \\
  \hline
  10  &	  4 & 	  6.5$\times 10^{-2}$  &	 1.1$\times 10^{5}$ & 93 \\
  \hline
   10 & 	 6 &    2.7$\times 10^{-1}$  &	 2.6$\times 10^{5}$  & 127  \\
  \hline
  10 & 	 8 &    1.3$\times 10^{-1}$  &	 1.7$\times 10^{5}$  &  131\\
  \hline
  10 & 	  10 &   1.4$\times 10^{-1}$  &	 2.6$\times 10^{5}$  &  133  \\
  \hline
\end{tabular}
\end{table}
}
{
\begin{table}[htbp]
\caption{Investigation of regret when learning $x^*$ and $\theta^*$ in a stochastic convex regime: 
$\xi\sim U[-{\theta}^*/2,{\theta}^*/2]$, $N=5$, $W=5$}   \label{table:optim_error_alpha}
\tiny
\centering
\begin{tabular}{|c|c|c|c|c|c|c|c|c|c|c|c|}
  \hline
  $\alpha$ & $\frac{R_K}{K \|z^*\| } $  & $\frac{\textrm{ERR}}{\|z^*\| }$     \\
 \hline
  0.5 &	 4.8$\times 10^{-2}$  &	 3.1$\times 10^{8}$    \\
  \hline
  0.6 & 	  3.3$\times 10^{-2}$  &	 3.1$\times 10^{8}$  \\
  \hline
  0.7 &     2.3$\times 10^{-2}$  &	 3.1$\times 10^{8}$    \\
  \hline
  0.8 &    1.8$\times 10^{-2}$  &	 3.1$\times 10^{8}$ \\
  \hline
  0.9 &  1.5$\times 10^{-2}$  &	 3.1$\times 10^{3}$ \\
  \hline
\end{tabular}
\end{table}
}

\section{Concluding remarks}\label{sec:V}
Traditionally, much of the field of optimization has been defined by
problems in which the functions and sets are known to the
decision-maker. However, as problems grow in their
reliance on data, such knowledge cannot be taken for granted. We
consider one such instance of such problems where functions may be
misspecified and the associated vector may be learnt through the
parallel solution of a suitably defined problem. It is worth emphasizing
the problem in the {\em full} space of learning and optimization
variables is a challenging (non-monotone) stochastic variational
problem for which no first-order methods are currently available. Yet,
by leveraging the structure of the problem, we show that such
problems can indeed be efficiently solved.

We consider a problem of solving a stochastic optimization problem in
which the objective is parameterized by a vector that can be learnt by
solving a suitably defined learning problem, captured by a stochastic
optimization problem. In both strongly convex and merely convex regimes,
we develop a set of coupled stochastic approximation schemes which
produces a sequence of iterates that are shown to converge to the
solution and unknown parameter in an almost sure sense. Additionally, we
provide rate estimates for the prescribed schemes in both strongly
convex and convex regimes. Through an analysis of the rate of
convergence under a diminishing steplength setting, it is seen that the
optimal rate of convergence is observed in strongly
convex problems while in convex regimes, we see a degradation introduced
by learning from $\Oscr\left(\frac{1}{\sqrt{K}}\right)$ to
$\Oscr\left(\frac{\sqrt{\ln(K)}}{\sqrt{K}}\right)$. This degradation is seen to disappear if the averaging window is modified appropriately. Similar rate
statements are also provided in a constant steplength regime.  In fact,
		   we may also cast this problem as an online decision-making
		   problem where a decision-maker sees a collection of
		   misspecified functions. In a stochastic regime, we observe
		   that an upper bound on the average regret can be shown to
		   decay at a rate no worse than
		   {$\Oscr\left(\frac{\ln K}{\sqrt{K}}\right)$ for a suitably chosen steplength}.

Unfortunately, the optimization-based model cannot accommodate settings
where there is misspecification in the constraints or, more generally,
if the associated decision-making problem is an equilibrium problem.
Motivated by this gap, we consider a misspecified stochastic variational
inequality problem and propose analogous stochastic approximation
schemes for computation and learning. To resolve the challenge
associated with merely monotone maps, we employ (Tikhonov) regularized
counterparts for which almost-sure convergence statements can be
provided. Additionally, we provide rate statements for constant and
diminishing steplength regimes, of which the latter requires imposing a
suitable weak-sharpness assumption on the original problem. Again, it is
seen that while the schemes display the  optimal rate of convergence
under strongly monotone regimes, a degradation in the rate is seen in
the monotone regime.

\bibliographystyle{siam}
\bibliography{ref}

\appendix

\section{Theorem 1 in \cite{zinkevich03}} \label{appendix:zinkevich}

\begin{defi}
Given an algorithm $A$, a convex set $F\subseteq\mathbb{R}^n$ and an
infinite sequence $\{c^1,c^2,\ldots\}$ where each $c^t:F\rightarrow\mathbb{R}$ is
a convex function, if $\{x^1,x^2,\ldots\}$ are the vectors selected by $A$, then the cost of $A$ until
time $T$ is defined as
  $C_A(T) = \sum_{t=1}^T c^t(x^t)$.
The cost of a static feasible solution $x\in F$ until time
$T$ is defined as $C_x(T) = \sum_{t=1}^T c^t(x).$
The regret of algorithm $A$ until time $T$ is defined as
  $R_A(T) = C_A(T) - \min_{x\in F} C_x(T).$
\end{defi}

The Greedy Projection algorithm proposed in \cite{zinkevich03} is as follows.\vspace{0.2in}

\begin{alg}[Greedy Projection] \label{alg:greedy_projection}
Select an arbitrary
$x^1\in F$ and a sequence of learning rates
$\eta_1,\eta_2,\ldots\in\mathbb{R}^+$. In time step $t$, after receiving a cost
function, select the next vector $x^{t+1}$ according to:
\begin{align*}
  x^{t+1}=\Pi_F(x^t-\eta_t\nabla c^t(x^t)).
\end{align*}
\end{alg}
\vspace{0.2in}
Then, we have the following result.\vspace{0.2in}

\begin{thm}[Theorem 1 in \cite{zinkevich03}] \label{thm:zinkevich_regret}
If $\eta_t=t^{-1/2}$, the regret of the Greedy
Projection algorithm is:
\begin{align*}
  R_G(T)\leq \frac{\|F\|^2\sqrt{T}}{2} + \left(\sqrt{T}-\frac{1}{2}\right)\|\nabla c\|^2,
\end{align*}
where $\|F\|\triangleq \max_{x,y\in F} d(x,y)$ and $\|\nabla c\|\triangleq \max_{x\in F, t\in \{1,2,\ldots\}} \|\nabla c^t(x)\|$.
\end{thm}\vspace{0.2in}

\textit{Proof sketch:} The regret of the Greedy
Projection algorithm can be bounded as follows:
\begin{align*}
  R_G(T)\leq \frac{\|F\|^2}{2\eta_T} + \frac{\|\nabla c\|^2}{2} \sum_{t=1}^T \eta_t.
\end{align*}
The result can be immediately obtained when $\eta_t=t^{-1/2}$.\qed

\end{document}